\documentclass[a4paper,11pt,reqno,english]{smfart}

\usepackage{enumerate}
\usepackage{amssymb,amsmath,latexsym,amsthm}
\usepackage[a4paper, tmargin=1in, bmargin=1in, lmargin=1in,rmargin=1in]{geometry}
\usepackage{url}
\usepackage[main=english]{babel}
\usepackage[utf8]{inputenc}
\usepackage{mathrsfs}
\usepackage{xcolor}
\usepackage{bigints}
\usepackage{comment}
\definecolor{violet}{rgb}{0.0,0.2,0.7}
\definecolor{rouge2}{rgb}{0.8,0.0,0.2}
\usepackage{tikz}
\usepackage{empheq}
\usepackage{tikz-cd}
\usetikzlibrary{matrix,arrows,decorations.pathmorphing}
\usepackage{hyperref}
\usepackage{mathpazo} 
\usepackage{smfhyperref}
\hypersetup{
	bookmarks=true,         
	unicode=false,          
	pdftoolbar=true,        
	pdfmenubar=true,        
	pdffitwindow=false,     
	pdfstartview={FitH},    
	pdftitle={},    
	pdfauthor={},     
	colorlinks=true,       
	linkcolor=rouge2,          
	citecolor=violet,        
	filecolor=black,      
	urlcolor=cyan}           
\usepackage{enumitem}
\usepackage{appendix}

\usepackage{mathtools}
\usepackage{dsfont}

\theoremstyle{plain}
    \newtheorem{thm}{Theorem}[section]
	
	\newtheorem{lem}[thm]{Lemma}
	\newtheorem{prop}[thm]{Proposition}

\theoremstyle{plain}
	\newtheorem{bigthm}{Theorem}
	
	\newtheorem{bigcor}[bigthm]{Corollary}
	
	\newtheorem{taggedbigcondx}{Condition} 
	\newenvironment{taggedbigcond}[1]
    {\renewcommand\thetaggedbigcondx{#1}\taggedbigcondx}
    {\endtaggedbigcondx}
        \newtheorem{taggedbigsetx}{Setting} 
	\newenvironment{taggedbigset}[1]
    {\renewcommand\thetaggedbigsetx{#1}\taggedbigsetx}
    {\endtaggedbigsetx}

    \newtheorem*{bigrmk*}{Remark}

\theoremstyle{definition}
	\newtheorem{defn}[thm]{Definition}
    \newtheorem*{defn*}{Definition}
	\newtheorem{eg}[thm]{Example}

	\newtheorem{sett}[thm]{Setting}

    \newtheorem*{conj*}{Conjecture}
	\newtheorem*{claim*}{Claim}
        
	\newtheorem*{ack*}{Acknowledgements}
\theoremstyle{remark}
	\newtheorem{rmk}[thm]{Remark}
	\newtheorem*{rmk*}{Remark}

	\newtheorem*{ques*}{Question}
	\newtheorem*{ans*}{Answer}

\numberwithin{equation}{section}
\renewcommand{\labelenumi}{(\roman{enumi})}
\newlist{steps}{enumerate}{1}
\setlist[steps, 1]{label = Step \arabic*:}

\mathcode`l="8000
\begingroup
\makeatletter
\lccode`\~=`\l
\DeclareMathSymbol{\lsb@l}{\mathalpha}{letters}{`l}
\lowercase{\gdef~{\ifnum\the\mathgroup=\m@ne \ell \else \lsb@l \fi}}%
\endgroup

\DeclareFontFamily{U}{MnSymbolC}{}
\DeclareSymbolFont{MnSyC}{U}{MnSymbolC}{m}{n}
\DeclareFontShape{U}{MnSymbolC}{m}{n}{
	<-6>  MnSymbolC5
	<6-7>  MnSymbolC6
	<7-8>  MnSymbolC7
	<8-9>  MnSymbolC8
	<9-10> MnSymbolC9
	<10-12> MnSymbolC10
	<12->   MnSymbolC12}{}
\DeclareMathSymbol{\intprod}{\mathbin}{MnSyC}{'270}

\DeclareMathOperator{\Hom}{Hom}

\DeclareMathOperator{\im}{im}
\DeclareMathOperator{\tr}{tr}

\DeclareMathOperator{\Ric}{Ric}

\DeclareMathOperator{\ext}{ext}

\DeclareMathOperator{\PH}{PH}

\DeclareMathOperator{\PSH}{PSH}
\DeclareMathOperator{\MA}{MA}

\DeclareMathOperator{\Bl}{Bl}
\DeclareMathOperator{\Exc}{Exc}

\DeclareMathOperator{\Spec}{Spec}

\DeclareMathOperator{\Aut}{Aut}

\DeclareMathOperator{\Bisec}{Bisec}
\DeclareMathOperator{\Scal}{S}
\DeclareMathOperator{\Lie}{Lie}
\DeclareMathOperator{\red}{red}

\DeclareMathOperator{\QPSH}{QPSH}

\def\1{\mathds{1}}

\def\E{\mathbf{E}}

\def\H{\mathbf{H}}

\def\L{\mathbf{L}}
\def\M{\mathbf{M}}

\newcommand{\ii}{\mathrm{i}}
\newcommand{\iii}{\sqrt{-1}}
\newcommand{\loc}{\mathrm{loc}}

\newcommand{\nmlz}{{\mathrm{norm}}}

\newcommand{\wX}{{\widetilde{X}}}

\newcommand\sm{\sigma}
\newcommand\dt{\delta}

\newcommand\vep{\varepsilon}
\newcommand\vph{\varphi}

\newcommand\om{\omega}
\newcommand\ta{\theta}
\newcommand\gm{\gamma}

\newcommand\af{\alpha}
\newcommand\bt{\beta}
\newcommand\ld{\lambda}

\newcommand\Dt{\Delta}
\newcommand\Om{\Omega}

\newcommand\Gm{\Gamma}
\newcommand\Ld{\Lambda}
\newcommand\Ta{\Theta}

\newcommand\RSL{\mathrm{SL}}

\newcommand\BN{\mathbb{N}}

\newcommand\BQ{\mathbb{Q}}
\newcommand\BR{\mathbb{R}}
\newcommand\BC{\mathbb{C}}

\newcommand\BD{\mathbb{D}}

\newcommand\CB{\mathcal{B}}
\newcommand\CC{\mathcal{C}}

\newcommand\CE{\mathcal{E}}

\newcommand\CH{\mathcal{H}}

\newcommand\CL{\mathcal{L}}

\newcommand\CO{\mathcal{O}}

\newcommand\CX{\mathcal{X}}
\newcommand\CY{\mathcal{Y}}
\newcommand\CZ{\mathcal{Z}}

\newcommand\lt{\left}
\newcommand\rt{\right}

\newcommand\pl{\partial}
\newcommand\db{\bar{\partial}}

\newcommand\ddb{\partial \bar{\partial}}

\newcommand\dd{d}
\newcommand\dc{d^c}
\newcommand\ddc{dd^c}

\newcommand\abs[1]{\left\lvert {#1} \right\rvert}

\newcommand\w{\wedge}

\newcommand\reg{\mathrm{reg}}
\newcommand\sing{\mathrm{sing}}

\newcommand\set[2]{\left\{ {#1} \, \middle| \, {#2} \right\}}

\newcommand{\RN}[1]{\textup{\uppercase\expandafter{\romannumeral#1}}}

\makeatletter
\newsavebox{\@brx}
\newcommand{\llangle}[1][]{\savebox{\@brx}{\(\m@th{#1\langle}\)}%
  \mathopen{\copy\@brx\kern-0.5\wd\@brx\usebox{\@brx}}}
\newcommand{\rrangle}[1][]{\savebox{\@brx}{\(\m@th{#1\rangle}\)}%
  \mathclose{\copy\@brx\kern-0.5\wd\@brx\usebox{\@brx}}}
\makeatother

\title{Weighted cscK metrics on K\"ahler varieties}

\author{Chung-Ming Pan}
\address[Chung-Ming Pan]{Centre interuniversitaire de recherches en g\'eom\'etrie et topologie (CIRGET); Universit\'e du Qu\'ebec \`a Montr\'eal; Case postale 8888, Succursale centre-ville, Montr\'eal, Qu\'ebec, H3C 3P8, Canada}
\email{\href{mailto:pan.chung_ming@uqam.ca}{pan.chung\_ming@uqam.ca} \qquad\qquad\qquad\qquad\qquad\qquad\qquad\qquad\qquad\qquad\qquad\qquad\qquad\qquad}
\urladdr{\href{https://chungmingpan.github.io/}{https://chungmingpan.github.io/}}

\author{Tat Dat T\^o}
\address[Tat Dat T\^o]{Sorbonne Université, Université Paris Cité, CNRS, IMJ-PRG, F-75005 Paris, France}
\email{\href{mailto:tat-dat.to@imj-prg.fr}{tat-dat.to@imj-prg.fr} \qquad\qquad\qquad\qquad\qquad\qquad\qquad\qquad\qquad\qquad\qquad}
\urladdr{\href{https://sites.google.com/site/totatdatmath/}{https://sites.google.com/site/totatdatmath/}}

\date{\today}

\begin{document} 

\maketitle

\begin{abstract} 
We study the weighted constant scalar curvature K\"ahler equations on mildly singular K\"ahler varieties.
Assuming the existence of a suitable resolution of singularities, we establish the existence of singular weighted cscK metrics when the weighted Mabuchi functional is coercive for an extremal weight. This extends the works of Chen-Cheng and He to the singular weighted setting. 
Moreover, we provide a method for constructing examples of singular cscK metrics inspired by the work of Arezzo--Pacard. 
In contrast to the usual gluing techniques, our approach does not require a precise understanding about of the metric behavior near the singular locus.  
\end{abstract}

\tableofcontents

\section*{Introduction}
The constant scalar curvature K\"ahler (cscK) metric problem has become one of the central focus in K\"ahler geometry during the last decades. The
Yau--Tian--Donaldson conjecture asserts that given a compact K\"ahler manifold with a fixed K\"ahler class, the existence of cscK metrics in the K\"ahler class is equivalent to an algebro-geometric notion called "K-stability". 

\smallskip
Several progresses in the literature \cite{Darvas_Rubinstein_2017, Berman_Darvas_Lu_2020, Chen_Cheng_2021_1, Chen_Cheng_2021_2} have shown that the existence of a unique cscK metric in a K\"ahler class is equivalent to the coercivity of the Mabuchi functional, whose minimizers are cscK metrics. 
Boucksom--Hisamoto--Jonsson~\cite{Boucksom_Hisamoto_Jonsson_2019} demonstrated that the coercivity of the Mabuchi functional implies uniform K-stability (see \cite{Dervan_Ross_2017, Dervan_2018, SD_2018, SD_2020} for a transcendental setup). 
Conversely, C. Li~\cite{Chi_Li_2022_2} (and the recent transcendental version by Mesquita-Piccione~\cite{Piccione_2024}) showed that the strong uniform K-stability implies the coercivity of Mabuchi functional. The remaining challenge in proving the uniform Yau–Tian–Donaldson conjecture lies in establishing strong uniform K-stability from uniform K-stability.
Very recently, Boucksom--Jonsson \cite{Boucksom_Jonsson_2025} and Darvas--Zhang \cite{Darvas_Zhang_2025} independently proved different versions of the Yau--Tian--Donaldson correspondence.

\smallskip
Typical examples of cscK metrics are K\"ahler--Einstein metrics. 
Motivated by Minimal Model Program and moduli theory, K\"ahler--Einstein metrics have been well studied on smooth and mildly singular K\"ahler varieties \cite{Aubin_1978, Yau_1978, EGZ_2009, Chen_Donaldson_Sun_2015, BBEGZ_2019, BBJ_2021, Li_Tian_Wang_2021, Li_Tian_Wang_2022, Chi_Li_2022} and their families \cite{Koiso_1983, Rong_Zhang_2011, Ruan_Zhang_2011, Spotti_Sun_Yao_2016, Li_Wang_Xu_2019, DGG2023, Pan_Trusiani_2023} etc. 

\smallskip
However, there are very few results regarding 
cscK metrics in the singular setting. 
We shall focus on 
the analytic part of the Yau--Tian--Donaldson conjecture on mildly singular varieties, particularly the relation between the existence of singular cscK metrics and the coercivity of the Mabuchi functional and explore under the weighted formalism. 
In a recent joint work with Trusiani \cite{Pan_To_Trusiani_2023}, when the Mabuchi functional is coercive, we establish the existence of singular cscK metrics on {\it $\BQ$-Gorenstein smoothable} K\"ahler varieties with log terminal singularities. 
One of the key ingredients is the stability of the coercivity of the Mabuchi functional \cite[Thm.~A]{Pan_To_Trusiani_2023}. 
A similar strategy for establishing openness of coercivity has been applied to the resolution setting by Boucksom--Jonsson--Trusiani \cite{Boucksom_Jonsson_Trusiani_2024} under an appropriate condition on the resolution. 

\smallskip
This article aims to remove the additional smoothable assumption and to investigate existence results in a more general weighted setting. 
The weighted framework introduced by Lahdili \cite{Lahdili_2019} (see
also \cite{Inoue_2022}) includes various notions of canonical K\"ahler metrics, for example, extremal K\"ahler metrics and K\"ahler--Ricci solitons. 
For further results on the existence of weighted cscK metrics, we refer to \cite{Apostolov_Jubert_Lahdili_2023, Lahdili_23, DiNezza_Jubert_Lahdili_2024, DiNezza_Jubert_Lahdili_2024_b, Han_Liu_2024} and the references therein. 

\smallskip
We quickly review the basic setup and notations for the weighted cscK metrics below (see Section~\ref{sec:prelim} for more details).

\begin{taggedbigset}{(GS)}\label{bigset:GS}
Let $X$  be an $n$-dimensional compact K\"ahler variety with log terminal singularities, let $T \subset \Aut_{\red}(X)$ be a maximal real torus in the reduced automorphism group, and let $\om$ be a $T$-invariant K\"ahler metric on $X$. 
Denote by $T_\BC$ the complexification of $T$,  $\frak{t}$ the Lie algebra of $T$ and $m_\om: X \to \mathfrak{t}^\vee$ a moment map associated to $\om$. 
Consider functions $v,w \in \CC^\infty(\frak{t}^\vee, \mathbb{R})$ with $v > 0$ on $m_\om(X)$. 
\end{taggedbigset}

A metric $\om'$ is a singular {\it $(v, w)$-cscK metric} if it has locally bounded potentials on $X$ and it solves the weighted cscK equation $\Scal_v(\omega')=w(m_{\om'})$ on $X^\reg$, where $\Scal_v$ is the weighted scalar curvature. 
Such an equation is the Euler--Lagrange equation of the $(v,w)$-weighted Mabuchi functional $\M_{v,w}$.
When $w = \ell w_0$, where $w_0 > 0$ on \( m_\omega(X) \) and \( \ell \) is an affine function such that the Mabuchi functional \( \M_{v,w} \) is $T_\BC$-invariant, the corresponding \( (v, w) \)-cscK metrics are called {\it\( (v, w_0) \)-extremal metrics}. 
In this context, \( w \) is referred to as an {\it extremal weight}. The cscK metrics and extremal K\"ahler metrics correspond to the special case of $ (1,1) $-extremal metrics.

\smallskip
Before presenting our results, we introduce the following additional condition:

\begin{taggedbigcond}{(A)}\label{cond_A}
There exists a $T$-equivariant resolution of singularities $\pi: Y \to X$ such that $Y$ is K\"ahler and $\pi$ is an isomorphism over $\pi^{-1}(X^\reg)$. 
Also, there exist a K\"ahler metric $\om_Y$ on $Y$, a positive constant $K_1 > 0$, and a 
function $\rho_1 \in \QPSH(Y)$ such that
\begin{equation*}
    \Ric(\omega_Y) \geq -K_1 (\pi^\ast \om + \ddc \rho_1) 
    \quad\text{and}\quad 
    \int_Y e^{-K_1 \rho_1} \omega^n_Y< +\infty.
\end{equation*}
\end{taggedbigcond}

In the above notation, $\QPSH(Y)$ is the set of all quasi-plurisubharmonic functions on $Y$.
A resolution of singularities described in Condition \ref{cond_A} is referred to as 
{\it a resolution of Fano type} in \cite{Boucksom_Jonsson_Trusiani_2024} (see Section \ref{sect_examples} and \cite[Sec.~4.1]{Boucksom_Jonsson_Trusiani_2024} for further discussions and examples). 

\smallskip
Under Condition~\ref{cond_A}, we establish the following existence theorem for singular weighted cscK metrics: 

\begin{bigthm}
\label{bigthm_existence_0}
Let $(X,\om)$ be a compact K\"ahler variety with log terminal singularities that satisfies Setting~\ref{bigset:GS}. 
Assume that $(X,\om)$ satisfies Condition~\ref{cond_A},  $v$ is $\log$-concave and $w$ is an extremal weight.
If the weighted Mabuchi functional $\M_{v,w}$ is $T_\BC$-coercive,
then $X$ admits a singular $(v,w)$-cscK metric in $\{\om\}$, which also minimizes $\M_{v,w}$. 
\end{bigthm}

For cscK and extremal metrics, the weight $v \equiv 1$ on $\frak{t}^\vee$ (hence $\log$-concave). 
In the smooth setup, existence results under the coercivity of the Mabuchi functional are obtained by \cite{Chen_Cheng_2021_2} and \cite{He_2019} for cscK and extremal metrics, respectively. 
Our Theorem~\ref{bigthm_existence_0}, in particular, extends their results to singular settings. 
We state a direct corollary for the case of cscK metrics, corresponding $v=1$ and $w = \frac{4 \pi n c_1(X) \cdot [\om]^{n-1}}{[\om]^n}$ in the following.
\begin{bigcor}\label{bigthm:cscK}
Let $X$ be a compact K\"ahler variety with log terminal singularities and let $T \subset \Aut_{\red}(X)$ be a maximal real torus. 
Assume that $\omega$ is a $T$-invariant K\"ahler metric and that $(X, \omega)$ satisfies Condition~\ref{cond_A}. 
If the Mabuchi functional $\M_{\omega}$ on $X$ is $T_\mathbb{C}$-coercive, then $X$ admits a singular cscK metric in $\{\om\}$  which also minimizes $\M_\om$. 
In particular, if $\Aut(X)$ is discrete, the coercivity of the Mabuchi functional implies the existence of singular cscK metrics in $\{\om\}$. 
\end{bigcor}

The strategy for proving Theorem~\ref{bigthm_existence_0} is to establish uniform a priori estimates for the solutions to a family of  
{weighted} cscK equations on the resolution of singularities $\pi: Y \to X$.
Denote by $E$ the exceptional set of $\pi$. 
Consider a family of perturbed K\"ahler metrics $\om_\vep := \pi^\ast \om + \vep \om_Y$ on the resolution for $\vep \in (0,1]$. 
Under Condition~\ref{cond_A},  and assuming $w$ is an extremal weight,   \cite[Thm.~A]{Boucksom_Jonsson_Trusiani_2024}  has proved the openness of uniform coercivity on $Y$. 
By \cite{DiNezza_Jubert_Lahdili_2024_b, Han_Liu_2024}, there exists a K\"ahler metric $\om_{\vep,\vph_\vep} := \om_\vep + \ddc \vph_\vep$ solving the weighted cscK equation on $Y$.  
Then the uniform coercivity of weighted Mabuchi functional yields a uniform control on the entropy $\H_\vep(\varphi_\vep)$.  Under Condition \ref{cond_A} and the bound on the entropy  we establish priori estimates for the weighted cscK equation extending the estimates of Chen and Cheng to the  {\it degenerate weighted} setting.
Consequently, one can then extract a subsequence $\vph_\vep$ converging in $\CC^\infty_\loc(Y \setminus E)$ to a bounded $\pi^\ast \om$-psh function $\vph_0$, which is smooth on $Y \setminus E \simeq X^\reg$ and solves the weighted cscK equation there. 

In particular, within the context of Corollary~\ref{bigthm:cscK},  these estimates  shows that singular K\"ahler--Einstein metrics on a K\"ahler varieties can be approximated by extremal metrics, provided Condition~\ref{cond_A} holds, and this generalizes a recent result by Sz\'ekelyhidi \cite[Thm.~3]{Szekelyhidi_2024} for non-discrete automorphism group. 

\smallskip
We highlight below the main difficulties and contributions made in the article: 

\begin{itemize}
    \item {\bf $L^\infty$-estimate and  Condition~\ref{cond_A}.}
    The main difficulty to obtain  uniform $L^\infty$-estimate is the lack of uniform bounds for the Ricci curvature of $\omega_Y$ with respect to the reference metrics $(\omega_\vep)_\vep$ on the resolution of singularities.
    That is the reason why the estimates of Chen--Cheng \cite{Chen_Cheng_2021_1} and Guo--Phong \cite{Guo_Phong_2022} cannot be applied directly.   
    To overcome this difficulty, we follow and generalize the Guo-Phong's approach by incorporating Condition~\ref{cond_A}, which provides a weak version of a lower bound for $\Ric(\omega_Y)$, together with the strong openness \cite{Berndtsson_2013_openness_conjecture, Guan_Zhou_2015_strong_openness}, and Demailly--Kollar's theorem \cite{Demailly_Kollar_2001}. 
    Moreover, as in the smooth setting \cite{Han_Liu_2024, DiNezza_Jubert_Lahdili_2024}, our $L^\infty$-estimate, Theorem~\ref{thm_uniform_cscK}, does not require that $v$ be log-concave and $w$ be an extremal weight.
    
    \item {\bf Constructing examples.} In Section~\ref{sec:examples}, we present a method for constructing singular cscK metrics inspired by Arezzo--Pacard \cite{Arezzo_Pacard_2006} under Condition~\ref{cond_A}. 
Additionally, we propose a mixed construction that integrates the result on the smoothable setting \cite{Pan_To_Trusiani_2023}. 
An important ingredient in the construction is the stability of the coercivity of the Mabuchi functional for the blow-up along compact submanifolds in the smooth locus of singular varieties (cf. Lemma \ref{lem:openness_coercivity}) instead of desingularizations as in \cite{Boucksom_Jonsson_Trusiani_2024}.
Comparing to the usual gluing technique (cf. \cite{Arezzo_Pacard_2006}), our method does not require a very precise understanding of how the metric behaves near the singular locus. 
\end{itemize}

We provide further detailed comments on the proof and the technical assumption on the $\log$-concavity of $v$.  
For higher order estimates, we adapt the strategy of Chen--Cheng \cite{Chen_Cheng_2021_2} and its generalization for the weighted setting \cite{Han_Liu_2024, DiNezza_Jubert_Lahdili_2024}. 
Due to the degeneration of $\omega_\epsilon$, we modify the test function for the Laplacian by adding a strictly $\pi^*\omega$-psh function $\rho$ with analytic singularities, ensuring $\omega_\vep+dd^c \rho\geq C \omega_Y$ to absorb problematic terms. 
We obtain an integral Laplacian estimate with respect to $\mu= e^{C'\rho} \omega_Y^n$, giving a local integral Laplacian estimate away from the exceptional locus. 
The $\log$-concavity of $v$ is used in the uniform integral Laplacian estimate in order to eliminate a problematic term in the weighted Aubin--Yau inequality. 
Local Laplacian and higher order estimates then follow from a generalization of Chen--Cheng's local estimates for the weighted setting (cf.  Appendix~\ref{sect_local_chen_cheng}).   
We remark that our estimates apply to general $w$ without requiring it to be an extremal weight.

\begin{ack*}
The authors are grateful to C.~Arezzo, S.~Boucksom, E.~Di~Nezza, S.~Jubert, A.~Lahdili, Y.~Odaka, S.~Sun, G.~Sz\'ekelyhidi, and A.~Trusiani for helpful and inspiring discussions. 
The authors would like to thank  S.~Boucksom, M.~Jonsson, A.~Trusiani, and G.~Sz\'ekelyhidi for kindly sharing their articles.
The authors would also like to thank V.~Guedj and H.~Guenancia for their suggestions that helped improve the exposition. 
We would like to thank the referee for very helpful comments and suggestions. 

\smallskip
Part of this article is based upon work supported by the National Science Foundation under Grant No. DMS-1928930, while the first-named author was in residence at the Simons Laufer Mathematical Sciences Institute (formerly MSRI) in Berkeley, California, during the Fall 2024 semester. 
The second-named author is partially supported by ANR-21-CE40-0011-01 (research project MARGE), PEPS-JCJC-2024 (CRNS) and Tremplins-2024 (Sorbonne University). 
\end{ack*}

\section{Preliminaries}\label{sec:prelim}
In this section, we recall several basic notions of pluripotential theory on singular spaces and weighted cscK metrics. 
We define $\dc := \ii (\db - \pl)$ and then we have $\ddc = 2 \ii \ddb$.
By variety, we always mean an irreducible reduced complex analytic space. 

\subsection{Pluripotential theory on normal K\"ahler varieties} 
Let $X$ be an $n$-dimensional compact normal complex variety. 
A {smooth K\"ahler metric} on $X$ is defined by a K\"ahler metric $\om$ on $X_\reg$, and it is locally a restriction of a K\"ahler metric defined near the image of a local embedding $j:X \underset{\loc.}{\hookrightarrow} \BC^N$. 
By {K\"ahler variety}, we mean a complex variety $X$ equipped with a smooth K\"ahler metric $\om$. 
More generally, a smooth form $\af$ on $X$ is defined as a smooth form on $X^\reg$ such that $\af$ extends smoothly under any local embedding $X \underset{\loc.}{\hookrightarrow} \BC^N$. 

\begin{defn}
A function $\phi: X \to \BR \cup \{-\infty\}$ is $\om$-plurisubharmonic ($\omega$-psh for short) if $\phi + u$ is plurisubharmonic where $u$ is a local potential of $\omega$; i.e. $\phi+u$ is the restriction of an psh function defined near an open neighborhood of $\im(j)$. 
Denote by $\PSH(X, \om)$ the set of all integrable $\om$-psh functions. 
\end{defn}

By Bedford--Taylor's theory \cite{Bedford_Taylor_1982}, the complex Monge--Amp\`ere operator can be extended to bounded $\om$-psh functions on smooth complex manifolds. 
In the singular setting, the complex Monge--Amp\`ere operator of locally bounded psh functions can also be defined by taking zero through singular locus (cf. \cite{Demailly_1985} for more details).

\subsubsection{Finite energy class}
Set $V := \int_X \omega^n$.  
For all $\varphi\in \PSH(X,\omega)\cap L^{\infty}(X)$, the Monge--Amp\`ere energy is defined by
\[
    \E(\varphi) := \frac{1}{(n+1)} \sum_{j=0}^n \int_X \vph \, \om_\vph^j \w \om^{n-j},
\]
where $\om_\vph := \om + \ddc \vph$.
The energy satisfies $\E(\varphi+ c) = \E(\varphi )+ cV$ for all $c \in \mathbb{R}$ and for $\varphi,\psi\in \PSH(X,\omega)\cap L^{\infty} (X) $, if $\varphi\leq \psi$ then $\E(\varphi)\leq \E(\psi)$ with equality iff $\varphi = \psi$. 
By the later property, $\E$ extends uniquely to $\PSH(X, \om)$ by 
\[
    \E(\varphi):= \inf \set{\E(\psi)}{ \varphi \leq \psi \in \PSH(X,\omega) \cap L^\infty(X)}.
\]
The finite energy class is then given as 
$$
    \mathcal{E}^1(X, \omega)= \set{\varphi \in \PSH(X, \omega)}{\E(\varphi) > -\infty}.
$$

\smallskip
The space of finite energy potentials $\CE^1(X,\om)$ admits a metric topology induced by the so-called $d_1$-distance, which can be expressed as follows (cf. \cite[Thm.~2.1]{Darvas_2017}, \cite[Thm.~B]{Dinezza_Guedj_2018})
\[
    d_1(u,v) = \E(u)+\E(v)-2\E(P_\omega(u,v)),
\]
where $P_{\omega}(u,v):=\left( \sup\set{w\in \PSH(X,\omega)}{w\leq \min(u,v)}\right)^*$. 

\subsection{Weighted cscK  metrics}
In this section, we review some basic concepts related to weighted cscK metrics. 

\subsubsection{Automorphism group and holomorphic vector fields}\label{sect_w_smt_setting} We recall here certain well-known properties of the Lie algebra of $\Aut(X)$ and some of its subgroups (cf. \cite{LeBrun_Simanca_1994, Gauduchon_book}).

\smallskip
Let $(X,\om)$ be a compact K\"ahler manifold and denote $J$ to be its complex structure.  
The automorphism group $\Aut(X)$ is a complex Lie group, whose Lie algebra  $\mathfrak{h} = \set{\xi \in \Gm(TX)}{\mathcal{L}_\xi J=0}$ consists of real holomorphic vector fields on $X$. 
A vector field $\xi$ is real holomorphic if and only if $\xi':= \frac{1}{2}(\xi-iJ\xi )$ is a holomorphic, so $\frak{h}$ can be identified with $H^0(X, T_X)$ the space of holomorphic vector fields. 
Denote by $\Aut_0(X)$ the identity component of $\Aut(X)$. 

\smallskip
The Albanese torus is defined by (cf. \cite[Thm.~9.7]{Ueno_1975})
\[
    \text{Alb}(X)= H^0(X, \Omega_X^1)^\vee /H_1(X, \mathbb{Z}),
\]
where the inclusion $H_1(X, \mathbb{Z} )\xhookrightarrow{}  H^0(X, \Omega_X^1)^\vee $ is induced via the map $\af \mapsto \int_c \af$ for any loop $ c$ and holomorphic $1$-form $\af$. 
Then one can define a Lie group homomorphism $\tau: \Aut_0(X)\rightarrow  \text{Alb}(X)$ as follows. 
Fix a $x_0\in X$, for any $\sm \in \Aut_0(X)$, we define $\tau(\sm): \af \mapsto \int_{x_0}^{\sm \cdot x_0} \af$, which does not depend on $x_0$ as $\af$ is harmonic.  
The homomorphism $\tau$ induces an action of $\Aut_0(X)$ on $\text{Alb}(X)$ by translation. 
Moreover, the derivative of $\tau$ at the identity is $\tau': \mathfrak{h}\rightarrow  H^0(X, \Om_X^1)^\vee$, $\xi \mapsto (\tau'(\xi): \alpha \mapsto \alpha(\xi))$. 
Define $\mathfrak{h}_{\red}:= \ker \tau'$, which consists of all holomorphic vector field $\xi \in  H^0(X,T_X)$ such that $\alpha(\xi)= 0$ for all $\alpha \in H^0(X,\Om_X^1)$. 
By \cite[Thm. 1]{LeBrun_Simanca_1994}, we have 
\begin{equation*}
    \frak{h}_{\red} 
    = \{\xi \in  H^0(X,T_X) \mid \xi = \nabla^{1,0}f = g^{j\bar k}\partial_{\bar k} f \frac{\partial}{\partial z^j},\, f\in \CC^\infty(X,\BC)\}. 
\end{equation*}
Then $\mathfrak{h}_{\red}$ generates a Lie subgroup $\Aut_{\red}(X) \subset \Aut_0(X)$ which acts trivially on the Albanese torus.  
We also identify  $\frak{h}_{\red}\subset \frak{h}$ with the Lie algebra of real holomorphic vector fields with zeros (cf. \cite[Thm. 1]{LeBrun_Simanca_1994}). 

\subsubsection{Weighted setting}
Fix a  real torus $T\subset\Aut_{\red}(X)$ with Lie algebra $\mathfrak{t}\subset \frak{h}_{\red}$. 
Any closed, real, $T$-invariant $(1,1)$-form $\omega$ admits a moment map $m_\omega: X\rightarrow \mathfrak{t}^\vee$, that is a unique (up to additive constant) $T$-invariant smooth map such that for each $\xi\in \mathfrak{t}$, $m_\om^\xi := \langle m_\om, \xi\rangle: X\rightarrow \mathbb{R}$ satisfies $-dm_\om^\xi = i_\xi\omega = \omega(\xi, \cdot)$; in other words, $m^\xi_\om$ is a Hamiltonian function of $\xi$ with respect to $\om$. 

\smallskip
Denote by $P_\om := \im(m_\omega)$ a compact set in $\mathfrak{t}^\vee$ and $\CH^T_\om:=\CH^T(X,\om)$ the set of $T$-invariant, strictly positive, smooth, $\om$-psh functions.   
For any $\phi\in \mathcal{H}^T_\om$, we normalize $m_\phi := m_{\om_\phi}$ by $m_\phi^\xi =m_\om^\xi + d^c \phi(\xi)$ for all $\xi \in \frak{t}$. 
Then from \cite[Lem.~1]{Lahdili_2019}, under this normalization, for any $\phi \in \mathcal{H}^T_\om$, one has $P_{\om_{\phi}} = P_\om$ and
\begin{equation}\label{eq_fixed_vol}
    \int_X m^\xi_\phi \om^n_\phi= \int_X m^\xi_\om \om^n, \quad  \forall \xi\in \frak{t}.
\end{equation}

\smallskip
Let $v,w\in \CC^\infty(\frak{t}^\vee, \mathbb{R})$ with $v>0$ on $P:=P_\om$. 
We recall the following notations and definitions from \cite[Sec.~3]{Boucksom_Jonsson_Trusiani_2024}, 
\begin{defn}\label{def_weighted_setting}
Let $(\ta, m_\ta)$ be a $T$-invariant pair and $f$ be a $T$ distribution. 
\begin{enumerate}
    \item A moment map $m_{dd^c f}$ for the $T$-invariant $(1,1)$-current $dd^c f$ is defined by
    \[
        m_{dd^c f}^\xi:= d^c f(\xi) = i_\xi d^c f;
    \]
    
    \item For a smooth function $g: \BR_{>0} \to \BR$, and another $T$-invariant pair $(\om, m_\om)$, 
    \[
        \langle (g(v))'(m_\omega), m_\ta \rangle 
        := g'(v(m_\omega)) \sum_{\af} v_{\af} (m_\om) m_\ta^{\xi_\af}
    \]
    where $(\xi_\af)_\af$ is a basis of $\frak{t}$, and $(v_{\af})_\af$ are the partial derivatives of $v$ with respect to the dual basis $(\xi^\af)_\af$ in $\frak{t}^\vee$;
    
    \item The $v$-weighted trace is given by 
    \[
        \tr_{\omega,v} \theta := \tr_\omega \theta + \langle (\log v)'(m_\omega), m_\theta \rangle;
    \]
    
    \item The $v$-weighted Laplacian is defined as
    \[
        \Delta_{\omega, v} f := \tr_{\omega,v} (dd^c f)= \Delta_{\omega} f + \langle (\log v)'(m_\omega), m_{dd^c f}\rangle;
    \]
    
    \item The $v$-weighted Ricci curvature is 
    \begin{equation*}
        \Ric_{v}(\omega):= \Ric(v(m_\om) \omega^n ) = -dd^c \log (v(m_\om) \omega^n) 
        = \Ric(\om) - dd^c \log v(m_\om);
    \end{equation*}
    
    \item The $v$-weighted scalar curvature is 
    \begin{equation*}
        \Scal_{v}(\omega):= \tr_{\omega,v } \Ric_{v}(\omega)= \Scal(\omega)-\frac{\langle v'(m_\omega), \Delta_\omega m_\omega \rangle + \Delta_\omega v(m_\omega)}{v(m_\omega)}. 
    \end{equation*}
    \end{enumerate}
\end{defn}

As \cite[(3.4)]{Boucksom_Jonsson_Trusiani_2024}, by applying the interior product $i_\xi$ to the trivial relation $d^c f\wedge \om^n=0$, one can obtain $m_{dd^c f}^\xi \om^n= n d^c f \w i_\xi \omega \wedge \om^{n-1}$.
Combining this formula with
\[
    dv(m_\om)= \sum_{\alpha}v_\alpha(m_\om)dm_\om^{\xi_\alpha}=-\sum_\alpha v_{\alpha} (m_\om)i_{\xi_\alpha}\om,
    \quad\text{and}\quad
    \langle v'(m_\omega), m_{dd^c f} \rangle 
    :=  \sum_{\af} v_{\af} (m_\om) m_{dd^c f}^{\xi_\af},
\]
one can derive that  
\[
\langle v'(m_\om), m_{dd^c f}\rangle= n\frac{dv(m_\om) \wedge d^c f\wedge \om^{n-1}}{\om^n}=\tr_\om(dv(m_\om)\wedge d^cf). 
\]
Therefore, we have the following formula for the weighted Laplacian
\begin{equation}\label{eq_w_laplacian}
        \Delta_{\omega, v}f =   \Delta_{\omega} f+ \tr_{\om} (d\log v(m_\om)\wedge d^cf).
\end{equation}

\smallskip
The $v$-weighted Monge--Amp\`ere operator is defined as 
\[
    \MA_v(\phi):= \MA_{\om,v}(\phi) := v(m_{\phi})(\omega+ dd^c \phi)^n
\] 
for any $\phi\in \CH^T_\om$.
From \cite[(3.32)]{Boucksom_Jonsson_Trusiani_2024}, one has the following integration-by-parts formula
\begin{equation}\label{eq:weighted_IPP}
    \int_X g(\Delta_{\om_{\phi}, v}f) \MA_v(\phi)=\int_X f(\Delta_{\om_\phi, v}g) \MA_v(\phi) = -\int_X \langle df,dg\rangle_{\om_\phi} \MA_v(\phi),
\end{equation}
for all $T$-invariant distributions $f,g$ such that at least one of which is smooth, 
where 
\[
    \langle df, dg \rangle_\om := \tr_\om (df \w \dc g ).
\]

\subsubsection{Weighted cscK equations}
Consider the weighted cscK problem
\begin{equation*}
    \Scal_v(\omega_\phi)= w(m_\phi).  
\end{equation*}
The above equation can be rewritten as 
\begin{equation*}
\begin{cases}
    v(m_\phi)(\omega+dd^c \phi)^n = e^F \omega^n, \quad \sup_X \phi= 0,\\ 
    \Delta_{\phi,v} F = -w(m_\phi)+ \tr_{\phi, v} (\Ric(\omega)),
\end{cases}  
\end{equation*}
where $\tr_{\phi,v} := \tr_{\om_{\phi},v}$ and $\Delta_{\phi,v} := \Delta_{\om_\phi, v}$. 
For any volume form $\mu_X$ on $X$, taking $\widetilde{F} := F+ \log \frac{\omega^n}{\mu_X} $, the original coupled equations are equivalent to the following coupled equations 
\begin{equation*}
\begin{cases}
    v(m_\phi)(\omega+dd^c \phi)^n = e^{\widetilde{F}} \mu_X, \quad \sup_X \phi = 0,\\ 
    \Delta_{\phi, v} \widetilde{F} = - w(m_\phi)+ \tr_{\phi, v} (\Ric(\mu_X)).  
\end{cases}
\end{equation*}

\smallskip
Since $P$ is compact, there are positive constants $C_{v}$ and $C_{w}$ such that, on $P$,   
\begin{equation*}
    C_{v}^{-1}
    \leq v + \sum_\af |v_\alpha| + \sum_{\af,\bt} |v_{\alpha\beta}| 
    \leq C_{v},  \text{ and }\,   
    C_{w}^{-1}
    \leq |w| + \sum_\af |w_\alpha| + \sum_{\af,\bt} |w_{\alpha\beta}| 
    \leq C_{w}.
\end{equation*}
This yields that 
\begin{equation*}
    |\tr_{\om_\phi,v} (\eta) - \tr_{\om_\phi} (\eta))|= |\langle (\log v)'(m_\phi), m_{\eta} \rangle
    |\leq C_{v, \eta},  
\end{equation*}
where $C_{v, \eta}$ depending only on $C_{v}$ and $\eta$. 
In particular, for any $\vph, \psi \in \CH^T(X,\om)$ and $\eta := \om_\psi$,
\begin{equation}\label{eq_compare_trace_2}
    |\tr_{\om_\vph, v} (\om_\psi ) - \tr_{\om_\vph} (\om_\psi)|= |\langle (\log v)'(m_\vph), m_{\psi} \rangle
        |\leq C_{v},  
\end{equation}
since $m_\vph(X) = m_\om(X) = P$ is compact.  

\subsubsection{Weighted Mabuchi functional}
Let $\mu_X$ be a fixed $T$-invariant  volume form on $X$. 
The $v$-weighted relative entropy is defined by
\[
    \H_{v}(\phi):=\H_{v,\mu_X}(\phi)
    :=\int_X \log \lt(\frac{\MA_v(\phi)}{\mu_X}\rt) \MA_v(\phi).
\] 
The $w$-weighted Monge--Amp\`ere energy $\E_{w}:=\E_{\om,w}: \CH_\om^T\rightarrow \BR$ is given by
\[
    (d \E_{w})_\phi(f):=  \int_X f w(m_{\om_\phi})\omega^n_\phi ,\quad \E_{w}(0)=0 
\]
for any $ f\in \CC^\infty(X)^T$ (cf. \cite[Lem. 3]{Lahdili_2019}).

\smallskip
Fix $\theta$ a $T$-invariant closed $(1,1)$-form and $m_\ta$ is its moment map. 
Then the twisted weighted Monge--Amp\`ere energy  $\E_{v}^\theta:=\E_{\om, v}^\theta: \mathcal{H}^T_\om \rightarrow \BR$ is the primitive of $\lt( v(m_\phi )n\theta\wedge \om_\phi^n+ \langle v'(m_\phi), m_\theta\rangle\om_\phi^n \rt)$ with $\E_{\om, v}^\theta(0) = 0$; in other words,
\[
   (d \E_{\om, v}^\theta)_\phi(f) 
   = \int_X f \lt(v(m_\phi) n \theta \wedge \om_\phi^n+ \langle v'(m_\phi), m_\theta\rangle\om_\phi^n\rt), 
   \quad \E_{\om, v}^\theta(0) = 0.
\]
for any $ f\in \CC^\infty(X)^T$ (see \cite[Lem. 4]{Lahdili_2019} for the well-definedness).  

\begin{defn}
    The weighted Mabuchi energy $\M_{v,w}$ is an Euler--Lagrange functional of $\CH^T_\om\ni \phi \mapsto \M'_{v,w}(\phi):=(w(m_\phi) -\Scal_v(\om_\phi))\MA_{v}(\phi)$.
\end{defn}

From \cite{Lahdili_2019, Boucksom_Jonsson_Trusiani_2024}, we have the following lemma. 
\begin{lem}
     The following formula holds on $\CH_\om^T$
    \[\M_{v,w}:=\M_{v,w, \mu_X}= \H_{v, \mu_X} 
  +\E_v^{-\Ric(\mu_X) }  + \E_{vw}.\]
\end{lem}
The above lemma follows from the fact that the functional $\CH^T_\om \ni \phi \mapsto \H_{v} +\E_v^{-\Ric(\mu_X)}$ is an Euler--Lagrange functional for $\phi\mapsto -\Scal_{v}(\om_\phi) \MA_{v}(\phi)$ (cf. \cite[Lem. 3.29]{Boucksom_Jonsson_Trusiani_2024}).
\begin{lem}\label{lem_minimizer}
   Any $(v,w)$-cscK metric in $\{\om\}$ is a minimizer of $\M_{v,w, \mu_X}$ for any choice of $\mu_X$.
\end{lem}
\begin{proof}
It follows from \cite[Cor.~1]{Lahdili_23} that  any $(v,w)$-cscK metric in $\{\om\}$ is a minimizer of $\M_{v,w, \om^n}$. 
On the other hand, by \cite[Lem. 3.48]{Boucksom_Jonsson_Trusiani_2024} there is a constant $C$ such that  $\M_{v,w, \mu_X}=\M_{v,w, \om^n}+C$, hence   any $(v,w)$-cscK metric in $\{\om\}$ is  also a minimizer of $\M_{v,w, \mu_X}$.
 \end{proof}

\subsubsection{Weighted extremal metrics}\label{sect_w_ext}
We recall here the definition of weighted extremal metrics and relative weighted Mabuchi energy, which is a modification of Mabuchi energy such that it is $T_\BC$-invariant. 
We refer to \cite[Sec.~3.6]{Boucksom_Jonsson_Trusiani_2024} for more details. 
\begin{lem}[{\cite[Lem. 3.34]{Boucksom_Jonsson_Trusiani_2024}}] \label{lem_Mabuchi_invariant}
The weighted Mabuchi functional $\M_{v,w}$ is translation invariant and $T_\BC$-invariant iff for all affine function $\ell =\xi+ c\in \frak{t}\bigoplus \BR$ on $\frak{t}^\vee$ we have $\int_X \ell (m_\om)\M'_{v,w}(0)=0$, i.e.
\[
    \int_X \ell(m_\om) w(m_\om)\MA_v(0)=\int_X \ell(m_\om) \Scal_v(\om)\MA_v(0).
\]
\end{lem}

Now, let $w_0 \in \CC^\infty(P, \BR_{>0})$. 
The weighted Futaki--Mabuchi pairing  on the space $\frak{t}\bigoplus \BR$ of affine functions on $\frak{t}^\vee$ is defined by 
\[
    \langle \ell,\ell'\rangle := \int_X \ell(m_\om) \ell'(m_\om) w_0(m_\om)\MA_v(0)=\int_X \ell(m_\om)\ell'(m_\om)\MA_{vw_0}(0).
\]
Then this is positive definite, and there exists the unique affine function $\ell^{\ext} =\ell^{\ext}_{\om,v,w_0}$ on $\frak{t}^\vee$ such that 
\begin{equation}\label{eq_ext_affine}
    \langle\ell, \ell^{\ext}\rangle= \int_X \ell(m_\om) \Scal_v(\om)\MA_v(0).
\end{equation}
Then the weighted Mabuchi energy $\M_{v,w}$ with $w=w_0\ell^{\ext}$ is $T_\BC$-invariant.  We define $\M^{\rm rel}_{v,w_0}:=\M_{v,w_0\ell^{\ext}}$ is the {\it relative weighted Mabuchi energy}.

\begin{defn} 
Let $v, w_0\in \CC^\infty(P, \BR_{>0})$. 
A metric $\om$ is called a $(v,w_0)$-extremal metric of it is $(v,w)$-cscK with $w= w_0\ell^{\ext}$.  In this case,  $w$ is called  an {\it extremal weight}.  
\end{defn}

\subsubsection{Extremal K\"ahler metrics}
\label{eg_extremal}
We explain here how the problem of finding extremal K\"ahler metrics is a special case of the one for weighted setting. This corresponds to the $(1,1)$-extremal metric, i.e. $v=1$ and $w_0=1$.

\smallskip
Let $(X, \omega)$ be a compact K\"ahler manifold, and let $g$ be the Riemannian metric defined by $\om$. 
The metric $\omega$ is said to be extremal if the Hamiltonian vector field $\xi_\om = J \nabla \Scal(\omega) $ is a Killing vector field for $g$, i.e $\mathcal{L}_{\xi_\om}g=0$.  

\smallskip
Let $T$ be a maximal compact torus of $\Aut_{\red}(X)$ and $\alpha$ be a K\"ahler class. 
From \cite{Futaki_Mabuchi_1995} (see also \cite[Sec.~3.1]{Lahdili_2019}), the projection of $\Scal(\om)$ with respect to $L^2(\om^n)$-inner product, to the sub-space $\{m_\omega^\xi + c: \xi\in \frak{t}, c\in \mathbb{R}\}$ is written as $\Pi_\om^T(\Scal(\om)) = m_{\om}^{\xi_{\ext}}+ c_{\ext}$ for some $\xi_{\ext}\in \frak{t}$ where $\xi_{\ext}$ only depends  on $T$ and $ \{\omega\}$.  
In particular, it implies that $c_{\ext}$ also depend only on $T$ and $\{\omega\}$, since 
\[
    nc_1(X)\cdot \{\omega\}^{n-1}
    = \int_X \Scal(\om)\om^n 
    = \int_X \Pi^T_\om(\Scal(\om))\om^n 
    = \int_X m_{\om}^{\xi_{\ext}} \omega^n + c_{\ext} \{\omega\}^n,
\]
where the last integral only depends  $\{\omega\}$  and $ \xi_{\ext}$ by \eqref{eq_fixed_vol}. 
One can also normalize $\int_X m_{\om}^{\xi_{\ext}} \omega^n=0$ to get $c_{\ext}= \overline{s}:= n\frac{ c_1(X)\cdot \{\omega\}^{n-1}}{\{\om\}^n}$.
Therefore, we obtain the unique affine function $\ell^{\ext}(p) = \langle \xi_{\ext}, p \rangle + c_{\ext}$ defined in \eqref{eq_ext_affine} and the problem of finding extremal metric is equivalent to the one for $(1, \ell^{\ext})$-cscK metric.  

\subsubsection{Extension on $\CE^{1,T}$ and coercivity} 
Denote by $\CE^{1,T}_\om := \CE^1(X,\om)^T \subset \CE^1(X,\om)$ the space of $T$-invariant finite energy potentials and $\CE^{1,T}_{\rm norm}(X,\om) := \{u \in \CE^{1,T}_\om \mid \E_\omega( u)=0\}$. 
From \cite[Prop.~3.41]{Boucksom_Jonsson_Trusiani_2024}, one can extend all functionals above on $\CE^{1,T}_\om$.

\smallskip
Since $\Aut_0(X)$ acts trivially on cohomologies, for each $\sm \in \Aut_0(X)$, one can find a unique function $\tau_\sm \in \PSH(X,\om) \cap \CC^\infty(X)$ such that 
\[
    \sm^\ast \om = \om + \ddc \tau_\sm,
    \quad \E(\tau_\sm) = 0.
\]
For $u \in \PSH(X,\om)$ and $\sm \in \Aut_0(X)$, define $\sm \cdot u := \sm^\ast u + \tau_\sm$.
Set
\begin{equation}\label{eq_def_d_1}
    d_{1,T}(u,0) := \inf_{\sigma\in T_{\BC}} d_1(\sigma \cdot u,0).
\end{equation}

\begin{defn}
The weighted Mabuchi functional $\M_{v, w}$ is $T_\BC$-coercive if there exist constants $A > 0$ and $B > 0$ such that
\begin{equation*}
    \M_{v,w} (\phi)\geq A d_{1,T} (\phi,0)-B
\end{equation*}
for any $\phi\in \CE^{1,T}_{\rm norm}(X,\om)$. 
\end{defn}
 By definition, if   $\M_{v, w}$ is $T_\BC$-coercive, then it  is $T_\BC$-invariant since it is bounded from below (cf. \cite[Section 1.1]{Boucksom_Jonsson_Trusiani_2024}).
\subsection{Weighted variational formalism in the singular setting}
We recall here the weighted formalism for the variational problem of singular weighted cscK metrics on K\"ahler varieties with log terminal singularities as introduced in \cite[Sec.~3.8, 4.1]{Boucksom_Jonsson_Trusiani_2024}.

\subsubsection{Reduced automorphism group and moment maps}\label{sect_red_aut}
Let $(X, \omega)$ be a normal compact K\"ahler variety. 
The automorphism group $\Aut(X)$ is a complex Lie group, whose Lie algebra is $H^0(X, T_X) \simeq H^0(X^\reg, T_{X^\reg})$ the space of holomorphic vector fields, that are global section of the tangent sheaf $T_X := \Hom(\Om_X^1, \mathcal{O}_X)$. 
There exists an $\Aut_0(X)$-equivariant resolution of singularities $\pi: \tilde{X} \to X$, i.e. $\pi$ is an isomorphism over $X^\reg$ and any $\sigma\in \Aut_0(X)$ can be extended to a unique $\sigma'\in \Aut_0(\tilde{X} )$. 
Hence, we get the inclusion $\Aut_0(X)\subset \Aut_0(\tilde{X} )$.  Moreover, any holomorphic vector field on $\tilde{X} $ descends to an element in   $H^0(X^\reg, T_{X^\reg})$, so we get $\Aut_0(X)\simeq \Aut_0(\tilde{X} )$. 

\smallskip
Given an $\Aut_0(X)$-equivariant resolution of singularities $\pi: \wX \to X$, the {reduced automorphism group} $\Aut_{\red}(X) \subset \Aut_0(X)$ is defined as the subgroup of $\Aut_0(X) \simeq \Aut_0(\wX)$ acting trivially on the Albanese torus of $\wX$. 
This definition is independent of the choice of $\Aut_0(X)$-equivariant resolution of singularities $\pi: \tilde{X} \rightarrow X$, by the bimeromorphic invariance of Albanese torus (cf. \cite[Prop. 9.12]{Ueno_1975}).  
Since $\Aut_0(X)\simeq  \Aut_0(\tilde{X} )$, we get $\Aut_{\red}(X) \simeq \Aut_{\red}(\tilde{X} )$.

\smallskip
Denote by $\mathcal{Z}$ (reps. $\mathcal{Z}'$) the space of locally $dd^c$-exact real $(1,1)$-form (resp. currents) on $X$. 
In particular, any current $\theta\in \mathcal{Z}'$ can be written as $\theta = \om + dd^c u$ for some $\om \in \mathcal{Z}$ and $u$ is a distribution \cite[Sec.~4.6.1]{Boucksom_Guedj_2013}. 
Then the group $\Aut(X)$ acts on $\CZ$ and $\CZ'$. 
In particular, $\Aut_0(X)$ acts trivially on the the classes of $\CZ$ (cf. \cite[Lem. 3.54]{Boucksom_Jonsson_Trusiani_2024}): for any $\sigma \in \Aut_0(X)$, there exist $f \in \CC^\infty(X)$ such that $\sigma^* \om = \om + dd^c f$. 

\smallskip
Fix a compact torus $T\subset\Aut_{\red}(X)$ with Lie algebra $\mathfrak{t}$. 
Any $T$-invariant $(1,1)$-form (resp. current) $\omega \in \mathcal{Z}$ admits a moment map $m_\omega: X \rightarrow \mathfrak{t}^\vee$ which is a $T$-invariant smooth map (resp. a distribution), unique up to an additive constant, such that for each $\xi\in \mathfrak{t}$, $m_\om^\xi:=\langle m_\om, \xi\rangle: X\rightarrow \mathbb{R}$ satisfies $- \dd m_\om^\xi = i_\xi \omega = \omega(\xi, \cdot)$.
Denote by $P := \im(m_\omega)$ which is a compact subset in $\mathfrak{t}^\vee$ and  normalize  such that  $P= \im(m_{\omega_\phi})$ for all $T$-invariant smooth $\om$-psh function $\phi$. 
Take $v,w\in \CC^\infty(t^\vee, \BR)$ with $\inf_P v > 0$. 

\subsubsection{Singular weighted cscK metrics}\label{sect_singular_recall}
Suppose that $X$ is $\BQ$-Gorenstein, meaning that $X$ is normal and $K_X$ is $m$-Cartier for some $m \in \BN^\ast$. 
Below, we recall the definitions of adapted measures and log terminal singularities from \cite[Sec.~5-6]{EGZ_2009}.

\begin{defn}\label{def_adapted_measure}
Let $h^m$ be a smooth hermitian metric on $m K_X$.
Taking $\Om$ a local generator of $m K_X$, the {\it adapted measure} associated with $h^m$ is defined by 
\[
    \mu_h 
    := \ii^{n^2} \lt(\frac{\Om \w \overline{\Om}}{\abs{\Om}_{h^m}^2}\rt)^{1/m}.
\]
This definition does not depend on the choice of $\Om$, and two adapted measures differ by a smooth positive density.
The {\it Ricci form} of the adapted measure $\mu_h$ is given as 
\[
    \Ric(\mu_h) :=  \ddc \log |\Omega |^{2/m}_{h^m}
\] 
which belongs to $\CZ$. 
The $\BQ$-Gorenstein variety $X$ has {\it log terminal singularities} if the measure $\mu_h$ has finite masses near $X^\sing$. 
\end{defn}

We now further assume $X$ is log terminal. 
As in \cite[Sec.~5]{EGZ_2009}, $\Ric(\mu_h)$ is canonically attached to an element in $H^0(X, \CC^\infty_X/\PH_X)$ where $\CC^\infty_X$ (resp. $\PH_X$) is the subsheaf of continuous functions on $X$ that are local restrictions of smooth functions (resp. pluriharmonic functions) under local embeddings. 
The {\it first Chern class of $X$}, denoted $\{\Ric(\mu_h)\}$, is the image of $\Ric(\mu_h)$ in $H^1(X,\PH_X)$ via the connecting homomorphism $\{\bullet\}$ in the following exact sequence 
\[
    H^0(X, \CC^\infty_X) \to H^0(X, \CC^\infty_X/ \PH_X) \xrightarrow[]{\{\bullet\}} H^1(X, \PH_X) \to 0
\]
which is induced by the short exact sequence $0 \to \PH_X \to \CC^\infty_X \to \CC^\infty_X/ \PH_X \to 0$.

\smallskip
Assume that $h^m$ is a $T$-invariant metric on $mK_X$ so that  $\Ric(\mu_h)$ is an equivariant curvature form.
 Fix a compact torus $T\subset\Aut_{\red}(X)$ with Lie algebra $\mathfrak{t}$. 
Set $P := P_\om := m_{\om}(X) \subset \frak{t}^\vee$ and take $v \in \CC^\infty(P,\BR_{>0})$. 
For any $\phi\in \CH^T_{\om}$, the {\it weighted Ricci current} of $\om_\phi:=\om+dd^c\phi$ is defined as
\[
    \Ric_{v}^T(\om_\phi)=\Ric^T(\MA_v(\phi)) 
    := -\ddc \log \lt(\frac{v(m_{\phi})\om_\phi^n}{\mu_h}\rt) + \Ric(\mu_h).
\]
and the $v$-weighted scalar curvature is the distribution expressed by 
\[
    \Scal_v({\om_\phi}):= \tr_{\om_\phi} (\Ric^T (\om_\phi)).
\]
One can extend the weighted energy and the weighted Ricci energy on $\CE^{1,T}_\om$ (cf. \cite[Sec.~4.1.2]{Pan_To_Trusiani_2023}, \cite[Sec.~3.8]{Boucksom_Jonsson_Trusiani_2024}). 
Fix a $T$-invariant adapted measure $\mu_X$, for any $w \in \CC^\infty(P,\BR)$, the {\it weighted Mabuchi energy} $\M_{v,w}: \CE^{1,T}_\om \rightarrow \BR\cup\{+\infty\}$ can be expressed as  
\[
    \M_{v,w}:=\M_{v,w, \mu_X}:= \H_{v} + \E^{-\Ric(\mu_X)}_v + \E_{vw}. 
\]

\begin{defn} 
Let $v \in \CC^\infty(P,\BR_{>0})$ and $w \in \CC^\infty(P, \BR)$. 
Then $\om_\phi \in \{\om\}$ is a singular $(v,w)$-cscK metric if $\phi \in \PSH(X,\om) \cap L^\infty(X)$ which is also smooth on $X^\reg$, and  
$$
    \Scal_v(\omega_\phi)= w(m_\phi) \quad \text{on } X^\reg.
$$
 Let $w_0\in \CC^\infty(P, \BR_{>0})$ and $\ell^{\ext}$ be
 the unique affine function $\ell^{\ext} =\ell^{\ext}_{\om,v,w_0}$ on $\frak{t}^\vee$ such that 
\[
    \langle\ell, \ell^{\ext}\rangle= \int_X \ell(m_\om) \Scal_v(\om)\MA_v(0).
\]
Then $\om_\phi$ is called a singular $(v,w_0)$-extremal metric if it is a singular $(v,w)$-cscK metric with $w = w_0\ell^{\ext}$.  
\end{defn}
Let $\pi: Y\rightarrow X$ be an $T$-equivariant resolution of singularities. We can also define $\M^{\rm rel}_{\pi^*\om}: \pi^*\CH^T_\omega \rightarrow \BR$ by
\[
    \M^{\rm rel}_{\pi^*\om}(\pi^* u):=\H_{\om_Y^n}(\pi^* u) + \E_{\pi^* \om}^{-\Ric(\om_Y)} (\pi^*u) + \E_{\pi^\ast\om, vw\ell^{\ext}} (\pi^\ast u).
\]
Under Condition \ref{cond_A}, following \cite[Lem.~4.22]{Boucksom_Jonsson_Trusiani_2024}, one can extend the weighted energy, the weighted Ricci energy, and $\M^{\rm rel}_{\pi^*\om}$ to $\CE^{1,T}_{\pi^*\om}$.

\section{A priori estimates}
In this section, we shall establish a priori estimates for weighted cscK equations on compact K\"ahler manifolds when the reference metrics are degenerating. 
These estimates are crucial to obtain the existence of singular weighted cscK metrics in Section \ref{sect_existence}. 

\smallskip 
In the sequel, we always assume the following setting:

\begin{sett}\label{sett:set_sec2} 
We fix $(X, \omega_X)$ to be an $n$-dimensional compact K\"ahler manifold, $T \subset \Aut_{\red}(X)$ a compact torus with Lie algebra $\mathfrak{t}$, and $\om$ a $T$-invariant K\"ahler metric on $X$ with the moment polytope $P$ of $\om$ and  $V:=\int_X \omega^n$. 
Take $v\in \CC^\infty(\frak{t}^\vee, \BR)$ such that $v > 0$ on $P$. 
\end{sett} 

\subsection{$L^\infty$-estimates}
In the following, we shall establish a uniform $L^\infty$-estimate following and generalizing the approach of Guo and Phong \cite[Thm.~3]{Guo_Phong_2022} with certain modifications (cf. \cite[Thm.~5.4]{Pan_To_Trusiani_2023}).

\begin{thm}\label{thm_uniform_cscK} 
Assume the Setting \ref{sett:set_sec2}. 
Let $\mu$ be a $T$-invariant  smooth volume form such that $\mu=g\omega^n_X $ with $K_0^{-1}\leq g\leq K_0$ for some constant $K_0 > 0$.
Suppose that $(\varphi, F) \in \CH^T(X,\om) \times \CC^\infty(X)^T$ is a solution to the coupled equations
\begin{equation*}
\begin{cases}
    v(m_\vph)(\omega + dd^c \varphi)^n = e^{F}\mu, 
    \quad \sup_X \varphi = 0,\\
    \Delta_{\varphi,v} F = -S + \tr_{\varphi,v}(\Ric(\mu)),
\end{cases}
\end{equation*}
for some $S\in \CC^\infty(X)^T$. 
In addition, assume that there are positive constants $K_1, K_3, K_4$ such that 
\begin{enumerate}[label=(\arabic*)]
    \item\label{cond_ric_1} 
    $-K_1 (\omega + dd^c\rho_1)\leq  {\rm Ric} (\mu)$,
    where $\rho_1$ is a $T$-invariant quasi-psh function, $\sup_X \rho_1 = 0$ and $\int_X e^{-K_1 \rho_1} d \mu < +\infty$; 

    \item\label{cond_ent} 
    ${\bf H}_\mu(\varphi)= \int_X \log\lt(\frac{\omega_\varphi^n}{\mu}\rt) \omega_\varphi^n \leqslant K_3$; 
    
    \item\label{cond_skoda}
    there exists $\alpha>0$ such that $\int_X e^{-\alpha(\phi-\sup_X \phi)} \mu \leqslant K_4$ for all $\phi\in {\rm PSH}(X, \omega)$.
\end{enumerate}
Then there is a uniform constant $C > 0$ depending only on $n, \max_X S, V, \alpha, K_1, K_3, K_4, C_v, \mu, \rho_1$ such that 
$$ 
    \|\varphi\|_{L^\infty} \leq C \, \text{ and } \, F \leq C. 
$$
Furthermore, if we have that   
\begin{equation*}
    \Ric(\mu) \leq K_2 (\om + \ddc \rho_2),
\end{equation*} 
for a $T$-invariant $\omega$-psh function $\rho_2 \in \CC^\infty(X\setminus Z)$, with $Z:=\{\rho_2 = -\infty\}$ and $\sup_X \rho_2=0$, there is a constant $C_2>0$ depending only on $n,\max_X |S|, V,\alpha, K_0, K_1, K_2, K_3, K_4, C_{v}, \mu, \rho_1$ and $\int_X \om\wedge \omega_X^{n-1}$ such that
$$F\geq K_2\rho_2-C_2.$$
\end{thm}

\begin{proof}
Consider $\tau_k: \BR \to \BR_{>0}$ a sequence of positive smooth functions decreasing towards the function $x \mapsto x\cdot \1_{\BR_{>0}}(x).$  
Let $\phi_k$ be a solution to the following auxiliary complex Monge--Amp\`ere equation
\begin{equation}\label{eq_aux_MA}
    V^{-1} (\om + \ddc \phi_{k})^n = \frac{\tau_k(-\varphi + \ld F)+1}{A_k} e^F \mu, 
    \quad \sup_X \phi_k=0,
\end{equation} 
where 
$$
    A_k = \int_X (\tau_k (-\vph  + \ld F) +1) e^F \mu
    \xrightarrow[k \to +\infty]{}  
    \int_X (-\vph + \ld F)_+ e^F \mu + \int e^F \mu = A_\infty.
$$
We remark that $\phi_k$ is also $T$-invariant since the RHS of \eqref{eq_aux_MA} is $T$-invariant. 
Applying Young's inequality with $\chi(s)=(s+1)\log (s+1)-s$ and $\chi^\ast(s) = e^s - s - 1$, 
\[
    \int_X (-\varphi) e^F\mu \leq \int_X \chi(\alpha^{-1}e^F)\mu + \int_X \chi^*(-\alpha\varphi)\mu
\]
where $\alpha > 0$ is the constant in \ref{cond_skoda}. 
It follows from \ref{cond_ent} and \ref{cond_skoda} that $C_v^{-1} V \leq  A_\infty \leq C(K_3, K_4, V, \alpha)$.
Thus $C_v^{-1} V \leq A_k \leq C_1 = C(K_3, K_4, V, \alpha)$ for $k$ sufficiently large. 

\smallskip
By strong openness \cite{Berndtsson_2013_openness_conjecture, Guan_Zhou_2015_strong_openness}, we have a constant $\dt >0$ such that $I := \int_X e^{-(K_1+\delta)\rho_1}< +\infty$. 
Using Demailly's approximation theorem \cite{Demailly_1992} (see also \cite[Sec.~14B]{Demailly_amag}) and Demailly--Koll\'ar's convergence result \cite[Main Thm. 0.2 (2)]{Demailly_Kollar_2001}, there exists a sequence quasi-plurisubharmonic functions $(\rho_{1, k})_k$ with analytic singularities and smooth away from their singular locus, such that $-K_1(\omega+dd^c\rho_{1,k}) -\frac{1}{k}\omega \leq \Ric(\mu)$, $\rho_{1,k} \to \rho_1$ in $L^1$ as $k \to +\infty$, and $\int_X e^{-(K_1+ \delta)\rho_{1, k}}\leq 2I$ for all $k \gg 1 $. 
Hence, $-K_1' (\omega +dd^c \rho'_{1,k}) \leq \Ric(\mu) $ where $K_1'= K_1+\frac{1}{k}$ with $k>1/\delta$ and $\rho'_{1,k}= \frac{K_1}{K_1 +1/k}\rho_{1,k}$. 
Replacing $\rho_{1}$ by $\rho'_{1,k}$ and $K_1$ by $K_1'$, one can assume $\rho_1$ has analytics singularities. 

\smallskip
Consider the function 
$$
    \Phi=-\epsilon(-\phi_k+\Lambda)^{\frac{n}{n+1}}-\varphi+\lambda (F+K_1\rho_1) 
$$ 
with $\Ld = \left(\frac{2n}{n+1} \epsilon\right)^{n+1}$ and $\epsilon = \left(\frac{(n+1)(n + \lambda \bar s + L)}{n^2}\right)^{\frac{n}{n+1}} A_k^{\frac{1}{n+1}}$, where $\bar s := \max_X S$, $L := \lambda(C_{v, \mu} +K_1 C_v) +2 C_v$ with $C_{v,\mu} > 0$ a constant independent of $\vph$ so that $\langle (\log v)'(m_\vph), m_{\Ric(\mu)}\rangle\geq -C_{v, \mu}$, and $\lambda > 0$ is a constant such that $n + \lambda \bar s > 0$ and $\ld K_1 < 1/2$. 
Since $\rho_1(x) \to -\infty$ as $x \to  Z_1:=\{\rho_1=-\infty\}$, the maximal points of $\Phi$ only occur in $X \setminus Z_1$. 
Fix $x_0 \in X \setminus Z_1$ a maximal point of $\Phi$. 
At $x_0$, we have 
\begin{align*}
    0&\geq \Delta_{\varphi,v}\Phi  
    = \Delta_{\om_\varphi}\Phi  
    + \langle (\log v)'(m_\vph), m_{dd^c \Phi} \rangle\\
    &\geq \frac{\epsilon n}{n+1}(-\phi_k + \Lambda) ^{ -\frac{1}{n+1}} \Dt_{\om_\varphi} \phi_k - \Dt_{\om_\vph} \vph  + \ld \Dt_{\vph,v} (F + K_1 \rho_1)  \\
    & \quad + \frac{\epsilon n}{n+1} (-\phi_k+ \Lambda)^{\frac{-1}{n+1}}\langle (\log v)'(m_\vph), m_{dd^c \phi_k}\rangle - \langle (\log v)'(m_\vph), m_{dd^c \vph} \rangle.
\end{align*}
Since $-C_v \leq \langle (\log v)'(m_\phi), m_{dd^c \psi} \rangle\leq C_v$ for any $\phi, \psi \in \PSH(X,\om)^T$ (see \eqref{eq_compare_trace_2}), we infer that 
\begin{align*}
    0 &\geq \frac{\epsilon n}{n+1} (-\phi_k + \Lambda)^{-\frac{1}{n+1}}  \Delta_{\om_\varphi} \phi_k - \Delta_{\om_\varphi} \varphi  + \lambda \Delta_{\varphi, v} (F + K_1\rho_1) - C_v \lt(\frac{\epsilon n}{n+1} \Lambda^{\frac{-1}{n+1}} + 1\rt)\\
    &= \frac{\epsilon n}{n+1}(-\phi_k+ \Lambda) ^{ -\frac{1}{n+1}}  (\tr_{\om_\varphi} \omega_{\phi_k} - \tr_{\om_\varphi}\omega) -\tr_{\om_\varphi} (\omega_\varphi - \omega) + \lambda (-S + \tr_{\om_\varphi} ({\rm Ric}(\mu) + K_1dd^c\rho_1))\\
    &\quad + \lambda \left(\langle (\log v)'(m_\vph), m_{\Ric(\mu)  }\rangle + K_1 \langle (\log v)'(m_\vph), m_{dd^c \rho_1}\rangle   \right)  -2 C_v.
\end{align*} 
Then 
\begin{align*}
    0&\geq  \frac{n^2\epsilon}{n+1} (-\phi_k+\Lambda)^{-\frac{1}{n+1}} \left( \frac{\tau_k(-\varphi+\lambda F)+1}{A_k}\right)^{1/n} -n- \lambda \bar s -L+  \left(1-\frac{n\epsilon}{n+1}\Lambda^{-\frac{1}{n+1}}  -\lambda K_1 \ \right)\tr_{\omega_\varphi}\omega\\
    &\geq \frac{n^2\epsilon}{n+1} (-\phi_k+\Lambda)^{-\frac{1}{n+1}} \left( \frac{\tau_k(-\varphi+\lambda F)+1}{A_k}\right)^{1/n} -n- \lambda \bar s -L,
\end{align*} 
where the second inequality above uses $\Ld = \left(\frac{2n}{n+1} \epsilon\right)^{n+1}$ and $\lambda K_1<1/2$; hence, $1-\frac{n\epsilon}{n+1}\Lambda^{-\frac{1}{n+1}}  -\lambda K_1 \geq 0$. 
Therefore we obtain
$$(\tau_k(-\varphi+\lambda F)+1)^{1/n}\leq \frac{(n+\lambda \bar s+L)(n+1)}{n^2\epsilon} A_k^{1/n}(-\phi_k+\Lambda)^{\frac{n}{n+1}},$$
and thus, 
\begin{equation}\label{eq:ineq_test}
   -\varphi+\lambda (F+K_1\rho_1) \leq -\varphi+\lambda F\leq  \left(\frac{(n+\lambda \bar s+L)(n+1)}{n^2\epsilon}\right)^n A_k(-\phi_k+\Lambda)^{\frac{n}{n+1}};
\end{equation} 
therefore, $\Phi(x_0) \leq 0$ and $\Phi \leq 0$ on $X$. 
By the choice of $\epsilon, \Ld$ and $V \leq  A_k\leq C(K_3, K_4, V)$, and Young's inequality, we derive that for any $\delta>0$
\begin{equation}\label{ineq_F}
    \lambda (F+K_1\rho_1)
    \leq -\varphi+\lambda (F+K_1\rho_1)
    \leq C(V, K_1,K_3,K_4) (-\phi_k + \Lambda)^{\frac{n}{n+1}}
    \leq -\delta \phi_k+ C_2 , 
\end{equation}
with $C_2=C(\delta,V,K_2, K_3, K_4)$.

\smallskip
From Condition~\ref{cond_A}, we have $\int_Y e^{-K_1 \rho_1} \mu < +\infty$. 
The strong openness \cite{Berndtsson_2013_openness_conjecture, Guan_Zhou_2015_strong_openness} yields a constant $0 < a \ll 1$ such that $\int_X e^{-(1+a) K_1\rho_1} \mu \leq C_a$ for some constant $C_a > 0$.
By \eqref{ineq_F}, H\"older inequality and \ref{cond_skoda}, for $\bt = \frac{1 + a/2}{\ld}$ and $\dt > 0$ such that $\dt \bt < \af / \gm^*$, with $\gm = \frac{1+a}{1+a/2}$ and $\frac{1}{\gm} + \frac{1}{\gm^\ast} = 1$, we obtain
\begin{equation}\label{ineq_exp_F}
    \int_X e^{\bt \ld F} \mu 
    \leq e^{\bt C_2} \int_X e^{-\frac{\af}{\gm^*} \phi_k - (1+\frac{a}{2}) K_1 \rho_1} \mu
    = e^{\bt C_2} \int_X e^{-\frac{\af}{\gm^*} \phi_k -  (1+a) K_1 \rho_1/\gm} \mu
    \leq C(a, C_a, \af).
\end{equation}

\smallskip
By a refined version of Ko{\l}odziej's $L^\infty$-estimate \cite{Kolodziej_1998} (see \cite[Thm.~A]{DGG2023} for the version we referred), a uniform control $C_{v}^{-1} \leq v(m_\vph) \leq C_{v}$ and \eqref{ineq_exp_F}, we obtain 
$\|\varphi\|_{L^\infty}\leq C(n,V,\alpha, K_1,K_3,K_4, C_{v}, \mu, \rho_1)$. Also, for $a \in (0,1)$ as above and $p':= 1+ a/4$,  we have the $y^{p'}\leq C_{a} e^{ay/4}, \forall y\geq 0$; hence
\begin{equation*}
    \int_X |F|^{p'}e^{p'F}\mu\leq C_a\int_Xe^{p'F+ \frac{a}{4}|F|}\mu \leq C_a\int_X e^{ (p'-\frac{a}{4})F }\mu +  C_a\int_X e^{ (p'+\frac{a}{4})F }\mu. 
\end{equation*}
Since $p'+ a/4= 1+a/2= \beta \lambda$, it follows from the estimate \eqref{ineq_exp_F} and H\"older inequality that
\begin{align}\label{est_FeF_Lp}
    \int_X |F|^{p'}e^{p'F}\mu\leq  C(a, C_a, \af),  \text{ for } p'=1+a/4. 
\end{align}
Combining with the $L^\infty$-estimate of $\vph$ yields  
\[
    \|(\tau_k(-\varphi +\lambda F)+1 )e^F \|_{L^{p'}(X, \mu)}\leq  C(n,\alpha, V, K_1, K_3, K_4, C_v, \mu, \rho_1)
\]   
for all $k>0$ sufficiently large. 
Again, Ko{\l}odziej's $L^\infty$-estimate yields a uniform bound $\|\phi_k\|_{L^\infty}\leq C(n,V,\alpha, K_1,K_3,K_4, C_v, \rho_1)$. 
Then the inequality \eqref{eq:ineq_test} provides a uniform upper bound for $F$. 

\smallskip 
In the second part, we consider the function $H := -F-(K_2+1)\varphi+ K_2\rho_2 + C_3$ where $$C_3 = \frac{1}{V}\int_X \left(F+(K_2+1)\varphi - K_2\rho_2 \right)\omega_\vph^n = \frac{1}{V}\int_X \left(F+(K_2+1)\varphi - K_2\rho_2 \right)v(m_\vph)^{-1}e^{F}\mu,$$ so that $\int_X H \omega_\vph^n = 0$.  
One can check that $C_3$ is uniformly bounded from \eqref{est_FeF_Lp}, the upper bound of $F$, the $L^\infty$-estimate for $\vph$, and the quasi-plurisubharmonicity of $\rho_2$. 
We have 
\begin{align*}
  \Delta_{\varphi, v} H
    &= S - \tr_{\omega_\varphi} (\Ric(\mu)) 
    -(K_2+1)n-(K_2+1) \tr_{\omega_\varphi} \omega 
    + K_2 \tr_{\omega_\varphi} (\om + dd^c \rho_2) - K_2 \tr_{\om_\vph} \om\\
    &\quad -\langle (\log v)'(m_\vph ), m_{\Ric(\mu)} \rangle
    - (K_2+1) \langle (\log v)'(m_\vph),  m_{dd^c \vph} \rangle 
    +K_2 \langle (\log v)'(m_\vph),  m_{dd^c \rho_2} \rangle \\
    &\geq S - \tr_{\omega_\varphi} (\Ric(\mu) -K_2 (\om+ dd^c\rho_2)) - (K_2+1) n + \tr_{\omega_\varphi} \omega + C_{v, \mu} - 2(K_2+1)C_{v} \\
    &\geq S -(K_2+1)n   - C_{v, \mu} - 2(K_2+1)C_{v}.
 \end{align*}
This implies that $H$ is a quasi-subharmonic function with $\Delta_{\varphi,v} H \geq -b$ for some $b>0$ depending on $n,S, K_2, C_v, C_{v, \mu}$. Moreover, we have $\omega_\vph^n = f\mu^n $ with $f= v^{-1}(m_\vph)e^F$ satisfying $\|f\|_{L^{1+a/2}(X,\mu)} \leq C(a, C_a, \alpha, C_v) $ by \eqref{ineq_exp_F} and $\mu= g\omega_X^n$ with $K^{-1}_0\leq g\leq K_0$.
Therefore, from Lemma~\ref{lem_upper_bound_sh} below, there is a uniform constant $C$ depending on $C(a, C_a, \alpha, C_v)$, $K_0$ and $\int_X \omega_\vph\wedge\omega_X^{n-1} = \int_X \omega\wedge\omega_X^{n-1}$ such that 
\begin{align*}
    \sup_X H &\leq C\left( b + \frac{1}{V_{\omega_\vph}} \int_X |H| \omega^n_\vph\right)\\
    &= C\left( b + \frac{1}{V} \int_X \big|-F - (K_2+1)\vph + K_2 \rho_2+C_3\big|  v^{-1}(m_\vph) e^F \mu\right)\\
    &\leq C\left( C_4 +  \frac{C_v(K_2+ 1)\|\vph\|_{L^\infty}}{V} \int_X e^F \mu +\frac{C_v K_2}{V\alpha'} \int_X  e^{-\alpha' \rho_2} e^{F}   \mu + \frac{C_v}{V}\int_X |F| e^F\mu\right).
\end{align*}
Here we choose $\alpha' = \frac{a+2}{a}\alpha$ so that by H\"older inequality, \ref{cond_skoda} and  \eqref{ineq_exp_F}
\begin{equation*}
    \int_X e^{-\alpha' \rho_2} e^F\mu\leq \left(\int_X e^{-\alpha\rho_2 }\mu   \right)^\frac{a}{a+2}\left(\int_Xe^{(1+a/2)F}\mu\right)^{\frac{2}{2+a}}\leq K_4^\frac{a}{a+2} C(a, C_a, \af)^{\frac{2}{2+a}}.
\end{equation*}
By \eqref{ineq_exp_F} and \eqref{est_FeF_Lp}, we also have 
\[
 \int_X e^{F}\mu\leq C(a, C_a, \af) \text{ and } \int_X |F|e^{F}\mu\leq C(a, C_a, \af).\]
Therefore, we obtain a uniform upper bound for $H$, and this yields a uniform lower bound for $F$ as required.
\end{proof}

We now provide a proof of a version of \cite[Lem.~5.1]{Guo_Phong_Song_Sturm_2024} (see also \cite[Lem.~2.1]{Guedj_To_24}) for the weighted Laplacian. 

\begin{lem}\label{lem_upper_bound_sh}
Assume the Setting \ref{sett:set_sec2}.  
Fix $b>0$ and let $u$ be in $L^1(X, \omega_X^n)$ such that $u\in C^2(\overline{\Omega})$ for some open set $\Omega \supset \{u>0\}$,   $\Delta_{\omega, v} u \geq -b$ on $\Omega$, and  $\int_X u \omega^n=0$.
Then 
$$
\sup_X u \leq C \left[ b+ \frac{1}{V_{\omega}}  \int_X |u| \omega^n \right],
$$
where $C$ only depends on $C_v$, $ \int_X \omega\wedge \omega_X^{n-1}$ and $ \int_X  \left(\frac{1}{V_\omega}\frac{\omega^n}{\omega_X^n}\right)^p \omega_X^n$, for any fixed $p>1$.  
\end{lem}

\begin{proof}
After scaling $u$, we may assume that  $b=n$.
Set $u_+=\widetilde{\max}(u,0)$, where 
$\widetilde{\max}$ denotes a convex regularized maximum satisfying $\max(x,0) \leq \widetilde{\max} (x,0) \leq 1+ \max(x,0)$.
Let $\vph \in \PSH(X,\omega)$ be the unique solution to
$$
    (\omega+dd^c \vph)^n=\frac{1+u_+}{1+M} \omega^n
    \quad\text{and}\quad \sup_X \vph=-1
$$
where  $M=\int_X u_+  \frac{\omega^n}{V_\omega} \leq 1+ \frac{1}{2} \int_X |u| \frac{\omega^n}{V_\omega}$.

Set $H=1+u-\vep (-\vph)^{\alpha}$, where $\alpha=\frac{n}{n+1}$ and 
$\vep\geq 1$ will be chosen later.
We now show that $H \leq 0$. 

Let $x_0$ be a maximum point of $H$.  
If $u(x_0)\leq 0$, then $H(x_0)\leq 0$, since $-\vph\geq 1$ and $\vep\geq 1$. 
Otherwise, we must have $x_0\in \Omega$. 
In this case, the following computations are carried out inside $\Omega$. 
Since
$$
-dd^c (-\vph)^{\alpha}=\alpha(1-\alpha) (-\vph)^{\alpha-2} d\vph \wedge d^c \vph+\alpha(-\vph)^{\alpha-1} dd^c \vph,
$$
we obtain $\Delta_{\omega}(-\vep (-\vph)^{\alpha}) \geq \alpha \vep(-\vph)^{\alpha-1} \Delta_{\omega} \vph
 \geq n \alpha \vep(-\vph)^{\alpha-1} \left[ \left(\frac{1+u_+}{1+M} \right)^{\frac{1}{n}} -1 \right]$.
Hence, one can infer that
 \begin{align*}
     \Delta_{\omega,v}H &\geq -n+n \alpha \vep(-\vph)^{\alpha-1} \left[ \left(\frac{1+u_+}{1+M} \right)^{\frac{1}{n}} -1 \right] + \vep\alpha (-\vph)^{\alpha-1} \tr_{\omega}(d\log v(m_\omega)\wedge d^c  \vph  )\\
     &\geq  -n + n \alpha \vep(-\vph)^{\alpha-1} \left[ \left(\frac{1+u_+}{1+M} \right)^{\frac{1}{n}} -1-2C_v \right].
 \end{align*}
In the second inequality above, we apply \eqref{eq_compare_trace_2} to gain the following estimate: 
$$|\tr_{\omega}(d\log v(m_\omega)\wedge d^c  \vph ) | = | \Delta_{\omega, v} \vph -\Delta_{\omega}\vph |= |\tr_{\omega, v} \omega_\vph -\tr_{\omega} \omega_\vph -(\tr_{\omega, v} \omega -\tr_{\omega} \omega) |\leq 2C_v.$$
 
Then, at $x_0\in \Omega$, we have 
\begin{equation}
    0 \geq \Delta_{\omega, v}H
    \geq -n + n \alpha \vep(-\vph)^{\alpha-1} \left[ \left(\frac{1+u_+}{1+M} \right)^{\frac{1}{n}} -1-2C_v \right].
\end{equation}
Thus using $(-\vph)^{1-\alpha} \geq 1$, at $x_0$, we obtain
$$
(1+ (2C_v+1)\alpha \vep)(-\vph)^{1-\alpha} \geq (-\vph)^{1-\alpha}+(2C_v+1)\alpha \vep \geq 
\alpha \vep \left(\frac{1+u_+}{1+M} \right)^{\frac{1}{n}}.
$$
Therefore,
$$
\vep (-\vph)^{\alpha}=\vep (-\vph)^{n(1-\alpha)} \geq \frac{\alpha^n \vep^{n+1}}{(1+(2C_v+1)\alpha \vep)^n} \frac{1+u_+}{1+M}=1+u_+\geq 1+ u,
$$
where $\vep$ is chosen such that $\frac{\vep^{n+1}\alpha^n}{(1+(2C_v+1) \vep)^n}=1+M$. This shows that $H (x_0)\leq 0$, and thus $H\leq 0$ on $X$.
Hence $$1+u \leq \vep (-\vph)^{\alpha},$$
and so
$$u_+ \leq \vep (-\vph)^{\alpha}.$$
Following the same argument in \cite[Lem.~2.1]{Guedj_To_24}, one obtains a uniform bound for $\vph$. 
Hence
$$
\sup_X u \leq \vep  (-\vph)^{\alpha} \leq c_n [1+M] C_0,
$$
and the conclusion follows since $M\leq \frac{1}{2V_\omega} \int_X |u| \omega^n+1$.
\end{proof}

\subsection{Local $L^p$-estimate for Laplacian}
Let $\om_X$ be another $T$-invariant K\"ahler metric on $X$. 
Since $P_{\om_X} = \im(m_{\om_X})$ is compact, one can further assume that 
\[
    C_v^{-1} 
    \leq |v| + \sum_\af |v_\af| + \sum_{\af,\bt} |v_{\af\bt}| 
    \leq C_v \quad\text{on } P_{\om_X}.
\]
Consider the following weighted cscK equations 
\begin{equation}\label{eq:wcscK_C2}
\begin{cases}
    v(m_\vph)(\om + dd^c \vph)^n = e^{F} \omega_X^n, \quad \sup_X \vph = 0,\\ 
    \Dt_{\vph, v} F = - S+ \tr_{\vph, v} (\Ric(\om_X)).  
\end{cases}
\end{equation} 
In this section, we shall further assume $\frak{t}^\vee \ni \zeta \mapsto \log v(\zeta)\in \BR$ to be {\it concave}.
With the concavity condition on $\log v$, from Lemma~\ref{lem:weighted_AY}, we have the following weighted Aubin--Yau type inequality: 
\[
    \Dt_{\vph, v} \log \tr_{\om_X} \om_\vph 
    \geq \frac{\Delta_{\omega_X} F}{\tr_{\om_X} \om_\vph} - \CB \tr_{\om_\vph}\om_X - \CC
\]
for some uniform constant $\CB, \CC > 0$.

\begin{prop}\label{prop_c^2} 
Suppose that $(\vph,F) \in \CH^T(X,\om) \times \CC^\infty(X)^T$ is a solution to \eqref{eq:wcscK_C2}.
Fix $p>1$. 
Assume that $\omega \leq C_\om\omega_X$, $\Ric(\omega_X) \leq A\omega_X$, $\Bisec(\om_X) \geq - B$,  $\omega+dd^c\rho\geq C_\rho \omega_X$ with $C_\rho > 0$, and 
\begin{equation*}
    \|\varphi \|_{L^\infty} \leq C_0,\quad  
    F \leq C_0, \quad \text{and}\quad  F \geq K_2 \rho-C_0,
\end{equation*}
where $\rho\in \CC^\infty(X\setminus \{\rho = -\infty\})$ is a $T$-invariant $\omega$-psh function. 
Then for any $\mathcal{K}$ compact set of $X\setminus \{\rho=-\infty\}$, one has the following estimate
\[
    \|\tr_{\omega_{X}} \omega_\vph\|_{L^{2p+2}(\mathcal{K}, \omega_X^n)}\leq C_1, 
\]
where $C_1$ only depends on $\mathcal{K},n,p, A, B, C_0, C_\rho, \min_X S$.
\end{prop}

\begin{proof} 
We shall adapt the approach in the smooth setting of Chen and Cheng \cite{Chen_Cheng_2021_1} for cscK metrics and also its generalization by \cite{DiNezza_Jubert_Lahdili_2024} and \cite{Han_Liu_2024} for weighted cscK metrics. 
We highlight some differences in the following:
\begin{itemize}
    \item 
    We shall work with the trace taken with respect to the reference metric $\om_X$ instead of $\om$ as $\om$ is moving when we are going to apply the result. 
    However, the metric $\om$ still plays a role and must be carefully merged during the computations, especially since we only have the upper bound $\om \leq C_\om \om_X$.

    \item  
    In the last step, we need to use bounds on $\vph$ and $F$. 
    Special attention is required for $F$ since its lower bound is uniform only up to a term involving $K_2 \rho$, which is not bounded from below.
\end{itemize}

\smallskip
Take $u := e^{-a(F+ b\vph-b\rho) }\tr_{\omega_X}\omega_{\vph}\geq 0$, where $a,b$ are positive constants to be determined later. 

Recall that from \eqref{eq_w_laplacian}, 
\[
    \Delta_{\omega, v}f =   \Delta_{\omega} f+ \tr_{\om} (d\log v(m_\om)\wedge d^cf);
\]
hence, we have on $X\setminus \{\rho=-\infty\}$
\begin{equation}\label{eq:int_C2_1}
\begin{split}
    \Delta_{\vph, v} u
    &= \Delta_{\vph, v} e^{\log u} 
    = \frac{|d u|^2_{\omega_\vph}}{u} + u\Delta_{\om_\vph} \log u + \tr_{\om_\vph}(d\log v(m_{\om_\vph})\wedge d^c u)\\
    &=\frac{|d u|^2_{\omega_\vph}}{u} + u\Delta_{\om_\vph} \log u + u\tr_{\om_\vph}(d\log v(m_{\om_\vph})\wedge d^c \log u)
    =\frac{|d u|^2_{\omega_\vph}}{u} + u \Delta_{\vph, v}\log u \\
    &\geq u \Delta_{\vph, v}\log u = -au \Delta_{\vph, v} ( F+ b \vph-c\rho) + u \Delta_{\vph, v}\log (\tr_{\om_X}\om_\vph). 
\end{split}
\end{equation}
Using $\Ric(\omega_X)\leq A\omega_X$, $|\langle (\log v)'(m_\vph), m_{\Ric(\om_X)} \rangle| \leq C_{v,\Ric(\om_X)}$, $\tr_{\vph, v} \omega_\vph  \leq n+ C_v$, and $ \tr_{\omega_\vph}{\om_\rho} \geq C_\rho \tr_{\vph, v} \omega_X - C_v$, we derive 
\begin{equation}\label{eq:int_C2_2}
\begin{split}
    \Delta_{\vph, v} (F + b\vph -b\rho) &= -S + \tr_{\vph, v}(\Ric(\omega_X))  + b\tr_{\vph, v}( dd^c \vph  -dd^c\rho)\\
    &\leq -\min_X S + A \tr_{\om_\vph} \omega_X + \langle (\log v)'(m_\vph), m_{\Ric(\om_X)} \rangle - b \tr_{\vph, v}\omega_\rho + b\tr_{\vph, v} \omega_\vph \\ 
    &\leq  -\min_X S  + (A -bC_\rho )\tr_{\om_\vph} \omega_X + b(2 C_v + n) + C_{v, \Ric(\om_X)}\\
    &=C_1 +(A - bC_\rho)\tr_{\om_\vph} \omega_X  
\end{split} 
\end{equation}
where $C_1>0$ only depends on $-\min_X S, b, C_v, C_{v, \Ric(\om_X)}$.

\smallskip
By a weighted Aubin--Yau's inequality for weighted Monge--Amp\`ere equation obtained in \cite{DiNezza_Jubert_Lahdili_2024} (see also Lemma \ref{lem:weighted_AY} for the version we apply), we have 
\begin{equation}\label{eq:int_C2_3}
    \Delta_{\vph, v}\log \tr_{\om_X}\om_\vph \geq \frac{1}{\tr_{\om_X} \om_\vph} \Delta_{\omega_X} F - \CB\tr_{\om_\vph}\om_X -\CC 
\end{equation}
where $\CB, \CC$ depend on $\omega_X, C_v$. 
Combining \eqref{eq:int_C2_1}, \eqref{eq:int_C2_2}, and \eqref{eq:int_C2_3}, we get
\[
    \Delta_{\vph, v} u
    \geq e^{-a(F+ b \vph-b \rho)} \left\{ -(a C_1+\CC) \tr_{\om_X} \om_\vph +\Delta_{\om_X}F + (abC_\rho-Aa -\CB)\tr_{\om_X}\om_\vph \tr_{\om_\vph}\om_X \right\}.
\]
Using the fact that 
\[
    \tr_{\om_X} \om_\vph \tr_{\omega_\vph}\omega_X 
    \geq (\tr_{\om_X} \om_\vph)^{\frac{n}{n-1}}  \left(\frac{\om_X^n}{\om_\vph^n} \right)^{\frac{1}{n-1}} 
    = (\tr_{\om_X} \om_\vph)^{\frac{n}{n-1}}  \left(v e^{-F}\right)^{\frac{1}{n-1}} 
    \geq (\tr_{\om_X} \om_\vph)^{\frac{n}{n-1}}  \left(C_v e^F\right)^{\frac{-1}{n-1}},
\]
and choosing $b $ large enough such that $abC_\rho-Aa-\CB\geq abC_\rho/2$, we get on $X\setminus \{\rho=-\infty\}$
\begin{align}\label{eq_delta_u}
    \Delta_{\vph, v} u\geq 
    -(aC_1+ \CC)u+ \frac{abC_\rho}{2} C^{-\frac{1}{n-1}} e^{\frac{-F}{n-1}}(\tr_{\om_X}\om_\vph)^{\frac{1}{n-1}} u+ e^{-a(F+b\vph-b\rho)} \Delta_{\om_X}F.   
\end{align}
Since $|\nabla u|_{\om_\vph}^2 \tr_{\om_X}\om_\vph \geq |\nabla u|^2_{\om_X}$, by \eqref{eq_delta_u}, we infer that on $X\setminus \{\rho=-\infty\}$
\begin{align}\nonumber
    \frac{1}{2p+1}\Delta_{\vph, v}u^{2p+1}
    &= u^{2p}\Delta_{\vph, v} u + 2pu^{2p-1}|\nabla u|^2_{\om_\vph} \geq u^{2p}\Delta_{\vph, v} u + 2pu^{2p-2}e^{-a(F+b\vph-b\rho)}|\nabla u|^2_{\om_X}\\ \nonumber
    &\geq u^{2p}\left(-(aC_1+ \CC)u+ \frac{abC_\rho}{2} C^{-\frac{1}{n-1}} e^{\frac{-F}{n-1}}(\tr_{\om_X}\om_\vph)^{\frac{1}{n-1}} u+ e^{-a(F+b\vph-b\rho)} \Delta_{\om_X}F\right) \\
    &\quad  + 2pu^{2p-2}e^{-a(F+b\vph-b\rho)}|\nabla u|^2_{\om_X} \label{eq_delta_u_2p+1}
\end{align}
Remark that $u^\alpha, \, \alpha\geq 1$ and $e^{ab\rho}$ are bounded quasi-psh functions. For bounded quasi-psh function $\psi$, for any K\"ahler form $\eta$, the distributions $(\Delta_\eta  \psi)\eta^{n}= ndd^c \psi\wedge \eta^{n-1} $ and $|\nabla \psi|^2_\eta \eta^n= n d\psi\wedge d^c\psi\wedge \eta^{n-1}$ do not charge pluripolar sets (see e.g. \cite[Thm.~3.14 \& Def.~3.2]{GZbook} for local versions).
Therefore, the integration on both sides of \eqref{eq_delta_u_2p+1} on $X \setminus \{\rho=-\infty\}$ extends the integration on the whole $X$.
Using the weighted integration-by-parts formula \eqref{eq:weighted_IPP}, we obtain 
\begin{equation}\label{eq_main_ineq}
\begin{split}
    0&= \int_X \frac{1}{2p+1}\Delta_{\vph, v}u^{2p+1} v(m_\vph) \om_\vph ^n \\
    &\geq \int_X  2pu^{2p-2}e^{-a(F+b\vph-b\rho) +F}|\nabla u|^2_{\om_X} \om_X^n -(aC_1 + \CC) \int_X u^{2p+1}e^F \om_X^n \\
    &\quad \quad +\frac{abC_\rho}{2} C_v^{-\frac{1}{n-1}}\int_X u^{2p+1} e^{\frac{n-2}{n-1}F} (\tr_{\om_X}\om_\vph)^{\frac{1}{n-1}}\om_X^n +\underset{=: I}{\underbrace{\int_X u^{2p} e^{-a(F+b\vph-b\rho)+F} \Delta_{\om_X}F\om_X^n}}.
\end{split}    
\end{equation}

Now we estimate the integral $I$ in \eqref{eq_main_ineq}.  Since we do not need to use the fact that $\omega+dd^c \rho\geq C_\rho\omega_X$, but only $\rho$ is $\omega$-psh, we can approximate $\rho$ by a decreasing sequence $(\rho_j)$ of smooth $\omega$-psh functions (see \cite{Blocki_Kolodziej_2007}). Then, at the end, we use the monotone convergence theorem to obtain an estimate for $I$. Therefore, in the estimate for $I$ below, we shall work with $\rho_j$, which is smooth, but for simplicity we still denote it by $\rho$. 

Set $G= (1-a) F-ab(\vph-\rho)$ and consider $a>1$. 
The last term $I$ can be written as the following two parts
\begin{align*}
     I=  \int_X u^{2p}\Delta_{\om_X}Fe^G \om_X^n= \underset{=: I_1}{\underbrace{\frac{1}{1-a} \int_X u^{2p}\Delta_{\om_X}(G) e^G \om_X^n}}  + \underset{=: I_2}{\underbrace{\frac{ab}{1-a} \int_X u^{2p}\Delta_{\omega_X} (\vph-\rho)e^G \om_X^n}}.
 \end{align*}
For $I_1$, the weighted integration-by-parts formula \eqref{eq:weighted_IPP} implies
\begin{align*}
    I_1&=\frac{1}{1-a}\int_X u^{2p}\Delta_{\om_X}(G) e^G \omega_X^n \\
    &=\frac{1}{a-1}\int_Xu^{2p}| \nabla G|^2_{\om_X} e^G \om_X^n  + \frac{1}{a-1} \int_X n 2p u^{2p-1}e^{G}du\wedge d^c G\wedge \om_X^{n-1}.
\end{align*}
By Cauchy--Schwarz inequality, 
\[
    \lt|2p u^{2p-1}  \frac{du\wedge d^c G\wedge \om_X^{n-1}}{\om_X^n} \rt| 
    \leq \frac{4p^2}{2}u^{2p-2} |\nabla u|^2_{\om_X}  +\frac{1}{2} u^{2p}|\nabla G |^2_{\om_X}.
\]
The above two inequalities yield
\begin{align}\label{eq_I_1}
     I_1\geq - \frac{2p^2}{a-1} \int_X u^{2p-2} |\nabla u|^2_{\om_X} e^G \om_X^n. 
\end{align}
We next control the term $I_2$. 
We compute
\begin{align*}
     \Delta_{\omega_X, v} (\vph-\rho) 
     = \tr_{\om_X, v} ( \om_\vph-\om_\rho) 
     &=  \tr_{\om_X}( \om_\vph-\om_\rho) + \langle (\log v)'(m_{\om_X}), (m_\vph-m_\rho) \rangle\\
     & \leq  \tr_{\om_X}  \om_\vph +C_3
 \end{align*} 
where $C_3$ only depends on $P, C_v$. 
Since $a>1$, we get 
\begin{align}\label{eq_I_2}
     I_2\geq  \frac{ab}{1-a}\int_X u^{2p} \tr_{\om_X}\om_\varphi e^G \om_X^n  + \frac{ab C_3}{1-a}\int_X   u^{2p} e^G \om_X^n.
\end{align}
Therefore, \eqref{eq_I_2} and  \eqref{eq_I_1}  yields 
\begin{equation}
I\geq - \frac{2p^2}{a-1} \int_X u^{2p-2} |\nabla u|^2_{\om_X} e^G \om_X^n +   \frac{ab}{1-a}\int_X u^{2p} \tr_{\om_X}\om_\varphi e^G \om_X^n  + \frac{ab C_3}{1-a}\int_X   u^{2p} e^G \om_X^n.
\end{equation}
 
\smallskip
Combining this with and \eqref{eq_main_ineq}, one can derive
{\small 
\begin{align*}
    0 &\geq 2\lt(p - \frac{p^2}{a-1}\rt) \int_X  u^{2p-2}e^G |\nabla u|^2_{\om_X} \om_X^n -(aC_1+\CC)\int_X u^{2p+1}e^F \om_X^n \\ 
    &\qquad+\frac{abC_\rho}{2} C_v^{-\frac{1}{n-1}}\int_X u^{2p+1} e^{\frac{n-2}{n-1}F} (\tr_{\om_X}\om_\vph)^{\frac{1}{n-1}}\om_X^n  
    + \frac{ab}{1-a}\int_X u^{2p} \tr_{\om_X}\om_\varphi e^G \om_X^n  + \frac{ab C_3 }{1-a}\int_X   u^{2p} e^G \om_X^n.
\end{align*}
}%
Taking $a > 1$ sufficiently large so that $p > p^2/(a-1)$, and using the upper bound $F \leq C_0$, one gains $\frac{n-2}{n-1} F \geq F -C_0/(n-1)$ and thus,  
{\small
\begin{align*}
    0&\geq -C_4 \int_X (\tr_{\om_X} \om_\vph)^{2p+1} e^{(2p+1)(G-F)+F} \om_X^n  -C_5 \int_X u^{2p} e^G \om_X^n
    +C_6 \int_X (\tr_{\om_X} \om_\vph)^{2p+1+\frac{1}{n-1}} e^{ (2p+1)(G-F) + F}\om_X^n,
\end{align*}
}%
where $C_4, C_5, C_6$ only depend on $ a, b, p, C_0, C_\rho, C_v,n$ and $|\nabla (v(m_{\omega_X}))|^2_{\om_X}$.  
Set $\mu := e^{(2p+1)(G-F) + F} \om_X^n = e^{-[a(2p+1)-1] F -ab(2p+1) \vph + ab(2p+1) \rho)} \om_X^n$. 
Choose $b \gg 1$ such that $ab(2p+1) > K_2 (a(2p+1)-1)$.  
By $F \geq K_2 \rho - C_0$ in the assumption, 
\[
    C_7^{-1}e^{ab(2p+1)\rho}\leq   \frac{\mu}{\om_X^n}\leq C_7,
\]
and it implies that $C_8^{-1} \leq \int_X \mu \leq C_8$.
Therefore, we obtain 
\begin{align*}
    0&\geq -C_4 \int_X (\tr_{\om_X} \om_\vph)^{2p+1} d\mu  -C_5   \int_X (\tr_{\om_X} \om_\vph)^{2p} d\mu  
    +C_6 \int_X (\tr_{\om_X} \om_\vph)^{2p+1+\frac{1}{n-1}} d\mu, 
\end{align*}
and H\"older's inequality shows $\| \tr_{\om_X} \om_\vph\|_{L^{2p+1}(X,\mu)}\leq C$. 
This concludes the proof.
\end{proof}

\subsection{Local higher-order estimates}
We next provide higher-order estimates that are away from the degenerating locus. 
We shall use the following local estimate, which generalizes
\cite{Chen_Cheng_2021_1} for cscK equations (see also \cite{DiNezza_Jubert_Lahdili_2024, Han_Liu_2024}). 
Its proof follows a similar argument in \cite{Chen_Cheng_2021_1} with further analysis of the weighted terms. 
For full details, the reader is referred to Appendix~\ref{sect_local_chen_cheng}.

\begin{thm}\label{thm_local_higher_est} 
Assume that $v\in  \CC^\infty(\frak{t}^\vee, \BR^+)$ and $w\in \CC^\infty(\frak{t}^\vee, \BR)$.
Let $\phi$ be a smooth solution of 
\begin{equation*}
\begin{cases}
    v(m_{dd^c\phi})\det{(\phi_{k \bar j})} = e^G \\ 
    \Delta_{\phi,v} G = -w(m_{\ddc\phi})
\end{cases}
\end{equation*}
in $B_1(0)\subset \mathbb{C}^n$, where $\Delta_{\phi,v} f:= \phi^{k\bar j} \partial_{k}\partial _{\bar j } f+ \tr_{dd^c\phi}(d\log v(m_{dd^c\phi}) \wedge d^c f )$, such that $\Delta \phi\in L^{p}(B_1(0))$ and $\tr_\phi \om_{\mathrm{eucl}}\in L^p (B_1(0))$ for some $p>3n$. Then there exists a constant $A>0$, depending only on $C_v$, $C_w$, $p$,  $\|\Delta \phi\|_{L^p (B_1(0))}$, $\|\tr_\phi \om_{\mathrm{eucl}}\|_{L^p(B_1(0))}$, such that $|\phi_{k\bar j}| \leq A$,  $|\nabla G|\leq A$ in $B_{\frac{1}{2}} (0))$ and for any $k\geq 2$
\[
    \|D^k \phi\|_{L^\infty(B_{1/2} (0))}\leq C(k, A, \|\phi\|_{L^\infty (B_1(0))}).
\]  
\end{thm}

Combining Theorem~\ref{thm_uniform_cscK}, Proposition~\ref{prop_c^2}, and Theorem~\ref{thm_local_higher_est}, we obtain the following result:
\begin{thm}\label{thm_est_cscK_general_2} 
Under the same assumption of Theorem \ref{thm_uniform_cscK} and assuming that $\log v$ is concave and $S= w(m_\phi), w\in \CC^\infty(\frak{t}^\vee, \BR)$, there is a uniform constant $C > 0$ depending only on $n, V, \alpha, K_0, K_1,K_2, K_3, K_4, C_v,C_w, \mu, \rho_1$ such that 
$$ 
    \|\varphi\|_{L^\infty} \leq C \, \text{ and } \, F \leq C \, \text{ and } F\geq K_2\rho_2-C, 
$$ 
where $\rho_2$ is a quasi-plurisubharmonic function with analytic singularities, such that $\om + \ddc \rho_2 $ is a K\"ahler current with $\Ric(\mu)\leq K_2(\om + \ddc \rho_2)$.

Moreover,  for any compact set $K \subset X \setminus \{\rho_2=-\infty\}$ and $l \in \BN$, we have $\|\vph_\vep\|_{\CC^l(K, \omega_X)} \leq C_{K,l}$ for some uniform constant $C_{K,l} > 0$, where $\om_X$ is a fixed K\"ahler metric on $X$.
\end{thm} 

\section{Existence of singular weighted cscK metrics} \label{sect_existence}
We shall combine the a priori estimates obtained in the previous section and the uniform coercivity established in \cite{Boucksom_Jonsson_Trusiani_2024} to construct singular weighted cscK metrics. 

\subsection{Setup}\label{sect_setting}
In the sequel, we always assume the following setting:

\begin{sett}\label{sett:set_sec3}
Let $X$ be an $n$-dimensional compact K\"ahler variety with log terminal singularities.
Fix a compact torus $T\subset \Aut_{\red}(X)$ and denote by $\mathfrak{t} = \Lie(T)$. 
Assume that $X$ admits a $T$-equivariant resolution of singularities $\pi: Y \to X$ with $K_Y= \pi^* K_X+ \sum_{i} a_i E_i$, $a_i>-1$, $Y$ is K\"ahler, and $\pi$ is an isomorphism over $X^{\reg}$.  Let $\om_Y$ be a $T$-invariant K\"ahler metric on $Y$.
Given a $T$-invariant K\"ahler metric $\om$ on $X$, by \cite[Thm.~3.17]{Boucksom_2004}, there exists a $T$-invariant quasi-psh function $\rho$ that is smooth on $Y \setminus E$ and has analytic singularities along $E$ such that $\pi^\ast \om + \ddc \rho$ is a K\"ahler current,  i.e. $\pi^\ast \om + \ddc \rho \geq C_\rho\omega_Y$ for some $C_\rho>0$.  Fix a $T$-invariant adapted measure $\mu_X$ on $X$, which is normalized by  $\int_ Y \log\left( \frac{\pi^* \mu_X}{\omega_Y^n}\right) \pi^* \omega^n=0$.
\end{sett}

For $\vep \in (0,1]$, we denote by $\omega_\vep := \pi^\ast \omega + \vep \omega_Y$ and $P_\vep := \im(m_{\omega_\vep}) = \im(m_{\pi^*\om}) + \vep \im(m_{\om_Y})$ which contained in a compact set of $\frak{t}^\vee$ for all $\vep$ small. 
Let $v,w\in \CC^\infty(\frak{t}^\vee, \mathbb{R})$ with  $v,w > 0$  on $P_\om$. 
Since $P_\om$ is compact, there are constants $C_v, C_w \geq 1$ such that for any $\vep$ sufficiently small, on $P_\vep$, 
\[
    C_v^{-1} \leq v + \sum_\af |v_\af| + \sum_{\af,\bt} |v_{\af\bt}| \leq C_v,
    \quad\text{and}\quad
    C_w^{-1} \leq |w| + \sum_\af |w_\af| + \sum_{\af,\bt} |w_{\af\bt}| \leq C_w.
\]
Denote by $\ell^{\ext}_\vep$ (reps. $\ell^{\ext}$) the extremal affine function on $\frak{t}^{\vee}$ associated to $\om_\vep$ (resp. $\om$). 
Since $\omega_\vep$ converges smoothly to $\pi^\ast \om$, the moment map $m_{\omega_\vep}$ converges smoothly to $m_{\pi^\ast \om}$ and $\ell^{\ext}_\vep\rightarrow \ell^{\ext}$ in $\frak{t} \oplus \BR$ (cf. \cite[Lem.~4.18]{Boucksom_Jonsson_Trusiani_2024}). 

\smallskip
Assume that a smooth pair $(\vph_\vep, F_\vep)$ solves
\begin{equation}\label{eq:wcscK_epsilon}
\begin{cases}
    v(m_{\vph_\vep})(\omega_\vep+dd^c \vph_\vep)^n = e^{ F_\vep} \omega_Y^n, \quad \sup_Y \vph_\vep=0, \\ 
    \Delta_{\vph_\vep, v} F_\vep = - (w\ell^{\ext}_\vep)(m_{\vph_\vep}) + \tr_{\vph_\vep, v} (\Ric(\om_Y)).
\end{cases}
\end{equation} 
Under Setting~\ref{sett:set_sec3}, one can conclude the following Theorem ~\ref{thm_weak_solution_cscK} by applying Theorem \ref{thm_est_cscK_general_2}. 

\begin{thm}\label{thm_weak_solution_cscK}
Under Setting~\ref{sett:set_sec3}, suppose that Condition~\ref{cond_A} holds and $\log v$ is concave. 
Let $(\vph_\vep, F_\vep)$ be a solution to \eqref{eq:wcscK_epsilon}. 
If there is a constant $C > 0$ independent of sufficient small $\vep > 0$ such that
\[
    \H_\vep(\vph_\vep) 
    \leq \int_Y \log \lt(\frac{\om_{\vep,\vph_\vep}^n}{\om_Y^n}\rt) \om_{\vep,\vph_\vep}^n \leq C, 
\]
then there exists a uniform constant $C_0 > 0$ such that 
\[
    \|\vph_\vep\|_{L^\infty} \leq C_0,
    \quad C_0 \geq F_\vep \geq K_2 \rho - C_0, 
\]
 where $K_2$ is a positive constant such that
$\Ric(\om_Y) \leq K_2(\pi^\ast \om + \ddc \rho)$.
Moreover, for any compact subset $K \subset Y \setminus \Exc(\pi)$ and $l \in \BN$, we have $\|\vph_\vep\|_{\CC^l(K,\om_Y)} \leq C_{K,l}$ for some uniform constant $C_{K,l} > 0$.
\end{thm}

\subsection{Proof of Theorem~\ref{bigthm_existence_0}}
We now prove the following result on the existence of singular weighted cscK metrics.

\begin{thm}\label{thm_existence}
Under Setting~\ref{sett:set_sec3}, suppose that Condition~\ref{cond_A} is fulfilled and $\log v$ is concave. 
If the weighted Mabuchi functional $\M_{\omega,v,w\ell^{\ext}}$ is $T_\BC$-coercive, then $X$ admits a singular $(v,w)$-extremal metric (i.e. $(v,w\ell^{\ext})$-cscK metric) in $\{\om\}$, and it is also a minimizer of $\M_{\omega,v,w\ell^{\ext}}$.
\end{thm}

\begin{proof}
\noindent{\bf Part 1: existence.}
Since $\M^{\rm rel}_X:=\M_{\om, v, w l^{\ext}}$ is  $T_\BC$-coercive, there are positive constants $A_0$ and $B_0$ such that
\[
    \M^{\rm rel}_X \geq A_0 d_{1,T}  - B_0
    \quad \text{on } \CE^1_{\nmlz}(X,\om)^T.
\]
By \cite[Thm.~A]{Boucksom_Jonsson_Trusiani_2024}, for any $A \in (0,A_0)$, there exists $B > 0$ such that for all sufficiently small $\vep > 0$, 
\begin{equation}\label{eq_coer_uniform_proof}
    \M_{\vep}^{\rm rel} \geq A d_{1,\om_\vep,T} - B
    \quad\text{on}\ \, \CE^1_\nmlz(Y,\om_\vep)^T,
\end{equation}
where $d_{1,\om_\vep,T}$  is  $d_{1,T}$ with respect to $\om_\vep$ in \eqref{eq_def_d_1} and
\[
        \M_{\vep}^{\rm rel} 
    := \M_{\om_\vep,v,w \ell^{\ext}_\vep}= \H_{\vep, v, \om_Y^n} +\E_{\vep, v}^{-\Ric(\om_Y) }  + \E_{\vep,vw \ell^{\ext}_\vep}
\]
is the weighted Mabuchi functional on $(Y, \om_\vep)$.
By the existence result obtained in \cite{DiNezza_Jubert_Lahdili_2024_b, Han_Liu_2024}, there is a smooth pair $(\vph_\vep, F_\vep)$ solving
\begin{equation}\label{eq_proof}
\begin{cases} 
    v(m_{\vph_\vep})(\omega_\vep+dd^c \vph_\vep)^n = e^{ F_\vep} \omega_Y^n, \quad \E_{\omega_\vep} (\vph_\vep)=0 \\ 
    \Dt_{\vph_\vep, v} F_\vep = - (w l^{\ext}_\vep)(m_{\vph_\vep})+ \tr_{\vph_\vep, v} (\Ric(\om_Y))
\end{cases}
\end{equation} 
Namely, $\om_{\vep, \vph_\vep}$ is a $(v, w l^{\ext}_\vep)$-cscK metric and $\vph_\vep$ minimizes $\M_{\vep}^{\rm rel}$. 

\smallskip
From Lemma \ref{lem_minimizer}, we know that $\vph_\vep $ minimizes the relative Mabuchi functional $\M_{\vep}^{\rm rel} $, so $\M_{\vep}^{\rm rel} (\vph_\vep) \leq \M_{\vep}^{\rm rel} (0)\leq C_1$, which gives 
\[
    d_{1,\om_\vep,T}(\vph_\vep,0) \leq \frac{C_1+B}{A}
\] 
using the coercivity condition. 
By the definition of $d_{1,\om_\vep,T}$, for each $\vep$ there exists $\sigma_\vep \in T_\BC$ such that $d_{1,\om_\vep}(\sigma_\vep \cdot \vph_\vep, 0) \leq \frac{C_1+B}{A}$. 
Set $\tilde \vph_\vep := \sigma_\vep \cdot \vph_\vep$. 
By the $T_\BC$-action, we have 
\[
    \om_{\vep, \tilde \vph_\vep} = \sigma_\vep^* \omega_{\vep, \vph_\vep},\quad
    m_{\tilde \vph_\vep} =  \sigma_\vep^* m_{\vph_\vep}
\] 
(see e.g. \cite[Lem.~3.19]{Boucksom_Jonsson_Trusiani_2024}). 
Pulling back the first equation in \eqref{eq_proof} by $\sigma_\vep$, we get 
\[
    v(m_{\tilde \vph_\vep})(\omega_\vep+dd^c \tilde \vph_\vep)^n 
    = \sm_\vep^\ast \lt(v(m_{\vph_\vep}) (\om_\vep + \ddc \vph_\vep)^n\rt)
    = e^{\sigma^*_\vep F_\vep} (\sigma^*\omega_Y)^n 
    = e^{\tilde F_\vep}\omega_Y^n
\] 
where $\tilde F_\vep= \sigma_\vep^*F_\vep +\log\frac{(\sigma_\vep^* \omega_Y)^n}{\om_Y^n}$. 
For any $T$-invariant K\"ahler metric $\omega$ on $Y$ and $\sm \in T_\BC$, we have  
\[
    \sm^\ast \lt(\tr_{\om,v} \bullet\rt) = \tr_{\sm^\ast\om,v} (\sm^\ast \bullet)
    \quad\text{and}\quad
    \sm^\ast \lt(\Dt_{\om,v} \bullet\rt) = \Dt_{\sm^\ast\om,v} (\sm^\ast \bullet).
\]
Combining this with the second equation in \eqref{eq_proof} yields
\begin{align*}
    \Dt_{\tilde \vph_\vep, v} \tilde F_\vep 
    &= \Dt_{\sm_\vep^\ast \om_{\vep, \vph_\vep},v} \sm_\vep^\ast F_\vep + \Dt_{\tilde \vph_\vep, v}  \log\frac{(\sigma_\vep^* \omega_Y)^n}{\om_Y^n}
    = \sigma_\vep^*\Dt_{ \vph_\vep, v} F_\vep +  \Dt_{\tilde \vph_\vep, v}  \log\frac{(\sigma_\vep^* \omega_Y)^n}{\om_Y^n}\\
    &= - (w l^{\ext}_\vep)(m_{ \tilde \vph_\vep})+ \tr_{\tilde \vph_\vep, v} (\Ric( \sigma_\vep^* \om_Y)) + \tr_{\tilde \vph_\vep, v} (\Ric( \om_Y) -\Ric(\sigma_\vep^*\omega_Y)) \\
    &=   - (w l^{\ext}_\vep)(m_{ \tilde \vph_\vep})+ \tr_{\tilde \vph_\vep, v} (\Ric( \om_Y)). 
\end{align*}
Therefore, $(\sigma_\vep \cdot \vph_\vep, \tilde F_\vep)$ also solves \eqref{eq_proof}. 
Replacing $\vph_\vep$ by $\sigma_\vep \cdot \vph_\vep$ and $F_\vep$ by $\tilde{F}_\vep$, we thus obtain 
\begin{equation}\label{eq_d1_control}
    d_{1,\om_\vep}(\vph_\vep, 0) \leq \frac{C_1+B}{A}. 
\end{equation}
Since $\M_\vep^{\rm rel}$ is $T_\BC$-invariant, this still minimize $\M_\vep^{\rm rel}$. 

\smallskip
Recall that $\pi^\ast \om + \ddc \rho \geq C_\rho \om_Y$.
We shall verify the assumptions in Theorem~\ref{thm_uniform_cscK} with $\rho_2 = \rho$ and $\mu = \om_Y^n$. 
It is obvious that the condition \ref{cond_skoda} holds. 
We have $\Ric(\om_Y) \leq C \om_Y$ for some constant $C > 0$ and $C_\rho \om_Y \leq \pi^\ast\om +\ddc \rho \leq \om_\vep +\ddc \rho$; hence, $\Ric(\om_Y) \leq K_2 (\om_\vep + \ddc \rho)$ for some constant $K_2 > 0$.
We now show that the uniform control on the entropies $\H_\vep(\vph_\vep)$, which will confirm the condition \ref{cond_ent}. 

\smallskip
If a closed $(1,1)$-current $\Ta \geq - A \pi^\ast \om$ for some constant $A > 0$, then
\begin{equation}\label{eq:twisted_energy_bound}
    \E_{\vep, v}^{\Ta}(\psi) \leq C(A,v,\Ta) (d_{1,\vep}(\psi, 0) + 1).
\end{equation}
for any $\vep \in (0,1]$ and for all $\psi \in \CE^1(Y,\om_\vep)$ (c.f \cite[Lem.~4.22]{Boucksom_Jonsson_Trusiani_2024}).  
By Condition~\ref{cond_A} and the strong openness \cite{Berndtsson_2013_openness_conjecture, Guan_Zhou_2015_strong_openness}, there is a constant $a > 1$ such that $\int_Y e^{-a K_1 \rho_1}\om_Y^n < +\infty$. 
Set $\mu_Y := e^{-aK_1\rho_1}\om_Y^n$.
Similar to \cite[Sec.~4.4]{Boucksom_Jonsson_Trusiani_2024}, we consider 
\[
    \widehat{\H}_{\vep,v}(\psi) 
    := \H_{\mu_Y}(\MA_{\vep,v}(\psi))
    \quad\text{and}\quad
    \widehat{\Ric} := \Ric(\om_Y) + K_1 \ddc \rho_1,
\]
 then $\widehat{\Ric} \geq - K_1 \pi^\ast \om \geq - K_1 \om_\vep$ by assumption.  
It follows from \cite[Lem.~A.2]{Boucksom_Jonsson_Trusiani_2024} and \cite[Prop.~3.44]{Boucksom_Jonsson_Trusiani_2024} that 
\begin{equation}\label{eq_H}
   \frac{1}{a}\widehat{\H}_{\vep, v} (\psi)=  \frac{1}{a}\H_{\vep, v} (\psi) + \int_Y K_1\rho_1 \MA_{\vep, v} (\psi)
\end{equation}
and 
\begin{equation}\label{eq_E}
    \E^{-\widehat{\Ric}}_{\vep , v} (\psi)= \E^{-{\Ric(\om_Y)}}_{\vep , v}(\psi) -\int_Y K_1\rho_1  \MA_{\vep, v} (\psi)+ \int_Y K_1\rho_1 \MA_{\vep, v}(0).
\end{equation} 
Therefore, we get
\begin{equation}\label{eq:BJT_Mab_ineq}
    \frac{1}{a} \H_{\vep,v,\om_Y^n}(\psi) + \E_{\vep,v}^{-\Ric(\om_Y)}(\psi) 
   =  \frac{1}{a} \widehat{\H}_{\vep,v}(\psi) + \E_{\vep,v}^{-\widehat{\Ric}}(\psi) - \int_Y K_1\rho_1 \MA_{\vep,v}(0),
\end{equation} 
with $ \int_Y K_1\rho_1 \MA_{\vep,v}(0)\rightarrow   \int_Y K_1\rho_1 v(m_{\pi^*\om} ) (\pi^* \omega)^n$ as $\vep \to 0$. 
Here without loss of generality, we assume that  $\int_Y\rho_1 v(m_{\pi^*\om} ) (\pi^* \omega)^n=0$. 
 Then one can derive
\begin{equation}\label{eq:Mab_ineq}
\begin{split}
        \M_{\vep}^{\rm rel} (\psi) 
    &= \H_{\vep,v,\om_Y^n}(\psi) + \E_{\vep,v}^{-\Ric(\om_Y)}(\psi) + \E_{\vep, v w l^{\ext}}(\psi) \\
    &\geq \lt(1-\frac{1}{a}\rt) \H_{\vep,v,\om_Y^n}(\psi) + \E_{\vep,v}^{-\widehat{\Ric}}(\psi) + \E_{\vep, v w l^{\ext}}(\psi) -C.
\end{split}    
\end{equation}
 By  \cite[Lem. 4.20]{Boucksom_Jonsson_Trusiani_2024},  $|\E_{\vep} (\psi)| \leq  C_2 d_{1, \om_\vep}(\psi, 0)$ for all $\psi \in \CE^1(Y,\om_\vep)$, hence
\eqref{eq_d1_control}  implies that
\[
    |\E_{\vep,vw \ell^{\ext}} (\vph_\vep)| \leq C_3= C(A,B, C_1,C_2, C_v, C_w), 
\]
for all $\vep$ small, and then \eqref{eq:twisted_energy_bound} and \eqref{eq:Mab_ineq} imply that
\[
    \lt(1-\frac{1}{a}\rt) \H_{\vep, v, \omega_Y^n}(\vph_\vep) 
    \leq C - \E_{\vep,vw \ell^{\ext}}(\vph_\vep) + \E^{\widehat{\Ric}}_{\vep,v}(\vph_\vep) 
    \leq C_1 + C_3 + C(K_1,v,\widehat{\Ric})C_4.
\]
Therefore, $\H_{\vep, v, \omega_Y^n}(\vph_\vep) \leq C'$ for all sufficiently small $\vep > 0$ and for some uniform constant $C' > 0$. 
We have
\begin{align*}
    &|\H_{\vep,v, \omega_Y^n}(\vph_\vep) - \H_{\vep, \omega_Y^n}(\vph_\vep)| 
    = \lt|\int_Y \log \lt(\frac{v(m_{\vph_\vep}) \om_{\vep,\vph_\vep}^n}{\om_Y^n}\rt) v(m_{\vph_\vep})\om_{\vep,\vph_\vep}^n 
    - \int_Y \log \lt(\frac{\om_{\vep,\vph_\vep}^n}{\om_Y^n}\rt) \om_{\vep,\vph_\vep}^n\rt| \\
    &= \lt|\int_Y \lt(1-\frac{1}{v(m_{\vph_\vep})}\rt) \log \lt(\frac{v(m_{\vph_\vep}) \om_{\vep,\vph_\vep}^n}{\om_Y^n}\rt) v(m_{\vph_\vep})\om_{\vep,\vph_\vep}^n 
    + \int_Y \log(v(m_{\vph_\vep})) \om_{\vep,\vph_\vep}^n\rt|
    \leq C_v' C' + C_v', 
\end{align*}
where $C'_v > 0$ depend only on $C_v$; hence, we obtain a uniform upper bound for $\H_{\vep, \omega_Y^n}(\vph_\vep)$ as required.

\smallskip
By Theorem~\ref{thm_est_cscK_general_2}, we obtained uniform $L^\infty$ and local $\CC^l$-estimates for $\vph_\vep -\sup_Y  \vph_\vep$. 
The Arzel\`a--Ascoli theorem shows that there is a subsequence $\vph_\vep -\sup_Y  \vph_\vep$ converging in $\CC_{\loc}^{\infty}(Y \setminus E)$ to $\vph_0 \in \PSH(Y \setminus E, \pi^*\omega) \cap \CC^{\infty}(Y \setminus E)$ which satisfies $(v,w\ell^{\ext})$-cscK equations on $Y \setminus E$. 
Moreover, since $\vph_\vep -\sup_Y  \vph_\vep$ is uniformly bounded on $Y$, we infer that $\sup_{Y \setminus E}|\vph_0|\leq C_0$ for some $C_0>0$. 
As $\vph_0$ is bounded from above on $Y \setminus E$, the upper semi-continuous extension of $\vph_0$ to the whole $Y$ is an $\pi^\ast \om$-psh, and it is also bounded on $Y$ by a lower bound of $\vph_0$ on $Y \setminus E$. 
We still denote this extension by $\vph_0$ afterward.  
This deduces the existence of a singular weighted cscK metric $\omega+dd^c\psi_0$ in $\{\omega\}$ on $X$, where $\psi_0\in \PSH(X, \omega)\cap  L^\infty(X) \cap \CC^{\infty}(X^\reg)$ and $ \varphi_0=\pi^*\psi_0$. 

\smallskip
\noindent{\bf Part 2: Minimizer.}
It remains to show that $\psi_0$ is a minimizer of $\M_{X}^{\rm rel}$ on $\CE^{1,T}_\om$. 
Note that $\vph_\vep$ are uniformly bounded and converging locally smoothly to $\pi^\ast \psi_0$ on $Y \setminus E$. 
Once can derive $\E_{\vep,vwl^{\ext}}(\vph_\vep) \to \E_{vwl^{\ext}}(\pi^\ast \psi_0)$ as $\vep \to 0$. 
Thus, by \eqref{eq:BJT_Mab_ineq} with $a =1$ and by the semi-continuity with respect to strong convergence of $\widehat{\H}_{\vep,v}$ and $\E_{\vep,v}^{-\widehat{\Ric}}$ 
(cf. \cite[Lem.~4.22]{Boucksom_Jonsson_Trusiani_2024}) , one gets
\[
    \liminf_{\vep \to 0}     \M_{\vep}^{\rm rel}  (\vph_{\vep}) 
    \geq \widehat{\H}_{v}(\pi^* \psi_0) 
    + \E^{-\widehat{\Ric}}_{v}(\pi^* \psi_0) 
    + \E_{vwl^{\ext}}(\pi^* \psi_0), 
\] 
where we used the fact that $ \int_Y K_1\rho_1 \MA_{\vep,v}(0)\rightarrow \int_Y K_1\rho_1 v(m_{\pi^*\om} ) (\pi^* \omega)^n =0$. 
It follows from \cite[Lem.~4.26]{Boucksom_Jonsson_Trusiani_2024} that 
\[
    \widehat{\H}_{v}(\pi^* \psi_0) + \E^{-\widehat{\Ric}}_{v}(\pi^* \psi_0) \geq \H_{\mu_X} (\psi_0)+   \E^{-{\Ric(\mu_X)}}_{v}(\psi_0). 
\]
This shows that
\begin{equation}\label{eq:lsc_Mab_soln}
    \liminf_{\vep \to 0}     \M_{\vep}^{\rm rel} (\vph_\vep)
    \geq     \M_X^{\rm rel} (\psi_0).
\end{equation}

\smallskip
Fix an arbitrary $u \in \CE^{1,T}_\om$. 
Without loss of generality, assume that $\H_{v, \mu_X}(u) < +\infty$ and denote by $f = \om^n_u/\omega^n$ which is $T$-invariant.
By \cite[Lem.~3.4]{Pan_To_Trusiani_2023}, there exists $0 \leq f^j \in \CC^\infty(X)$ converging to $f$ in $L^\chi$ as $j \rightarrow 0$ with $\chi(s):= (s+1)\log(s+1)-s$. 
Consider $f^{T,j}(x):= \int_T f^j(\sigma \cdot x)d\mu(\sigma)$, where $\mu$ is the  Haar measure on $T$ normalized by $\int_T d\mu(\sigma) = 1$.
We claim that $f^{T,j} \in \CC^\infty(X)$ also converges to $f$ in $L^\chi$. 
Indeed, since $\chi$ is convex, 
\[
    \chi( |f^{ T,j}-f|(x))  
    = \chi \lt(\lt|\int_T(f^{ T,j}-f)(\sigma\cdot x)d\mu(\sigma) \rt| \rt) 
    \leq \int_T \chi(|f^j-f|(\sigma\cdot x)) d\mu(\sigma).
\] 
By Fubini's theorem and $T$-invariance of $\om$, we obtain that 
\begin{align*}
    &\int_{x \in X} \chi(|f^{j, T}-f|(x)) \om^n(x)
    \leq \int_{x \in X} \lt(\int_{\sm \in T} \chi(|f^j-f|(\sigma\cdot x)) d\mu(\sigma)\rt) \om^n(x)\\
    &= \int_{\sm \in T} \lt(\int_{x \in X} \chi(|f^j-f| (\sm \cdot x)) \omega^n(x)\rt) d\mu(\sigma) 
    = \int_{\sm \in T} \lt(\int_{x \in X} \chi(|f^j-f| (\sm \cdot x)) \omega^n(\sm \cdot x)\rt) d\mu(\sigma).
\end{align*}
Hence, $f^{T,j}$ converges to $f$ in $L^\chi$.  
Consider $u_j \in \CC^\infty(X^\reg) \cap L^\infty(X)$ solving 
\[
    (\om + \ddc u_j)^n = c_j f^{T,j} \om^n 
    \quad\text{with}\quad 
    \sup_X u_j = 0
\] 
where $c_j > 0$ is a normalizing constant. 
Moreover, it follows from \cite[Thm.~1.3]{Cho_Choi_2024} that $u_j$ is continuous on $X$. 
By \cite[Lem.~3.4]{Pan_To_Trusiani_2023}, $u_j$ converges strongly to $u$, and $\H_{\omega^n}(u_j) \rightarrow \H_{\omega^n}(u)$. 
Since $u_j$ converges to $u$ strongly, $u_j$ converges to $u$ in $W^{1,2}(\om^n)$ (cf. \cite[Lemma 1.9]{BBEGZ_2019}) and thus $v(m_{u_j})$ converges almost everywhere to $v(m_u)$ (see Definition \ref{def_weighted_setting} (i)). 
We have 
\begin{equation}\label{eq:wH_split}
    \H_{\mu_X,v}(u_j) 
     = \H_{\omega^n, v}(u_j) + \int_X  v(m_{u_j})  c_j f^{T, j} \log \left(\frac{\om^n}{\mu_X}  \right)  \omega^n.
\end{equation}
Note that $g := \log(\om^n/\mu_X)$ is quasi-psh. 
For the second term in \eqref{eq:wH_split}, consider
\begin{align*}
    \int_X v(m_{u_j}) c_j g f^{T, j} \omega^n 
    - \int_X v(m_u) g f \omega^n
    &= \int_X v(m_{u_j}) c_j g (f^{T, j}-f) \omega^n 
    + \int_X [v(m_{u_j}) c_j-v(m_u)] g f \omega^n.
\end{align*}
By H\"older--Young inequality, 
\[
    \left|\int_X v(m_{u_j}) c_j g (f^{T, j}-f) \omega^n\right| \leq \|f^{T,j}-f\|_{L^\chi} \|v(m_{u_j}) c_j g\|_{L^{\chi^\ast}} \xrightarrow[j \to \infty]{} 0
\]
as $\|v(m_{u_j}) c_j g\|_{L^{\chi^\ast}}$ is uniformly bounded by the uniform control of $v$, $c_j$ and exponential integrability of $g$ by quasi-plurisubharmonicity.  On the other hand, it follows from Lemma \ref{lem_convergence_entropy}, we have $\H_{\om^n,v}(u_j) \to \H_{\om^n, v}(u)$.
This shows that $\H_{\mu_X, v}(u_j) \to \H_{\mu_X,v}(u)$ and thus, $\M_{v,wl^{\ext}}(u_j) \to \M_{v,wl^{\ext}}(u)$. 

\smallskip
Then for each $j$ fixed, one can find a family of functions $u_{j, \vep} \in \PSH(Y, \om_\vep)$ such that 
\begin{itemize}
    \item $(u_{j,\vep})_{\vep \in (0,1)}$ are uniformly bounded and continuous;
    \item $u_{j,\vep}$ converges locally smoothly on $Y \setminus E$ as $\vep \to 0$;
    \item $u_{j,\vep}$ decreases to $\pi^\ast u_j$ as $\vep \to 0$.
\end{itemize}
Since $\pi^* u_j$ is continuous on $Y$, Dini's theorem implies $ u_{j,\vep} $ converges uniformly to $\pi^*u_{j}$ as $\vep\rightarrow 0$. 
Indeed, for a fixed $j$, consider $u_{j,\vep}$ the unique solution to the following equation
\begin{equation*}
    (\om_\vep+dd^c u_{j,\vep})^n
    = e^{u_{j,\vep} - \pi^\ast u_j} \pi^* (c_j f^{T,j} \om^n). 
\end{equation*}  
Then by \cite{EGZ_2009}, one has a uniform $\CC^0$-estimate for $u_{j,\vep}$ on $Y$, local $\CC^2$-estimate for $u_{j,\vep}$ away from $E$, and $u_{j, \vep}$ converges locally smoothly to $\pi^* u_j$ in $Y \setminus E$. 
We now check that $u_{j,\vep}$ is decreasing as $\vep \to 0$.
For $0 < \vep_1 < \vep_2$, we have 
\[
    (\om_{\vep_2} + \ddc u_{j, \vep_1})^n 
    \geq (\om_{\vep_1} + \ddc u_{j, \vep_1})^n 
    = e^{u_{j,\vep_1} - \pi^\ast u_j} \pi^* (c_j f^{T,j} \om^n),
\]
so $u_{j,\vep_1}$ is a subsolution to the equation 
\[
    (\om_{\vep_2} + \ddc \vph) = e^{\vph} e^{-\pi^\ast u_j} \pi^* (c_j f^{T,j} \om^n).
\] 
Thus, we obtain that $u_{j,\vep_1} \leq u_{j,\vep_2}$ for any $0 < \vep_1 < \vep_2$. 

\smallskip
Combining with \eqref{eq:lsc_Mab_soln}, we get 
\[
    \M_X^{\rm rel} (\psi_0)
    \leq \liminf_{\vep \to 0} \M_{\vep}^{\rm rel} (\vph_{\vep}) 
    \leq \liminf_{\vep \to 0} \M_{\vep}^{\rm rel}(u_{\vep, j})  = \M_X^{\rm rel} (u_j), 
\]
where the second inequality follows from the fact that $\vph_\vep $ is a minimizer for $\M_{\vep, v, w l^{\ext}}$ and the last equality follows from Lemma \ref{lem_convergence_Mabuchi_smt} below. 
Letting  $j\rightarrow +\infty$, we obtain that $\M_{v,wl^{\ext}}(\psi_0) \leq \M_{v,wl^{\ext}}(u)$ for arbitrary $u \in \CE^{1,T}_\om$ and this finishes the proof.
\end{proof}

We recall a corollary of the Lebesgue dominated convergence theorem:
\begin{lem}\label{lem_convergence_entropy}
Fix a measure $\mu_X$ on $X$ and $\mu = f\mu_X$ another measure on $X$ with $f \in L^1(X,\mu_X)$ and $\H_{\mu_X}(\mu) < +\infty$. 
Assume that $\mu_j = f_j \mu_X$ is a sequence of measures such that $f_j$ converges to $f$ almost everywhere with respect to $\mu_X$, $\mu_j(X)\rightarrow\mu(X)$ and $\H_{\mu_X}(\mu_j)\rightarrow \H_{\mu_X}(\mu)$. 
If $(v_j)_{j\geq 1}$ is a sequence of positive functions converging to $v$ almost everywhere with respect to $\mu_X$ as $j\rightarrow +\infty$, and uniformly bounded from above, then
\[
    \H_{\mu_X} (v_j\mu_j)\underset{j\rightarrow +\infty}{\longrightarrow}  \H_{\mu_X}(v\mu).
\]
\end{lem}

\begin{proof}
From the hypothesis, we have
\[
    \int_Xf_j\log f_j \mu_X \underset{j\rightarrow +\infty}{\longrightarrow} \int_X f\log f \mu_X < +\infty.
\] 
Now we write
\begin{align*}
    \H_{\mu_X}(v_j\mu_j) 
    &= \int_{X} \log(v_j f_{j}) (v_jf_{j}) \mu_X 
    =  \int_X v_j\log(v_j)f_{j} \mu_X + \int_{X} v_j \log(f_{j}) f_{j} \mu_X \\
    &= \int_X v_j\log(v_j) f_{j} \mu_X + \int_{X} v_j (f_j\log(f_{j}) + e^{-1}) \mu_X 
    - e^{-1}\int_X v_j \mu_X. 
\end{align*}
As $(v_j)_j$ are uniformly bounded, we have a constant $M > 10$ such that $0 < v_j + v\leq M$ for all $j$. 
Hence, 
\[
    0 \leq |v_j\log(v_j)| f_j \leq M\log M f_j, 
\]
and 
\[
    0\leq v_j (f_j\log(f_j) +e^{-1})
    \leq M (f_j\log(f_j) +e^{-1}).
\]
As $v_j \to v$ and $f_j \to f$ a.e., by a generalized Lebesgue dominated convergence theorem (cf. \cite[Exercise 2.20, p.59]{Folland}), we have  
\[
    \lim_{j\to+\infty}\int_X v_j \log(v_j) f_{j} \mu_X  
    = \int_X v \log(v) f \mu_X,  
\] 
and  
\[
    \lim_{j\to+\infty} \int_{X} v_j (f_j\log(f_{j}) + e^{-1}) \mu_X  
    = \int_{X} v (f\log(f) + e^{-1}) \mu_X. 
\]
From the above to limit and the hypothesis, we obtain 
$\H_{\mu_X} (v_j\mu_j)\rightarrow \H_{\mu_X}(v\mu)$ as $j \to +\infty$.
\end{proof}

{ 
\begin{lem}\label{lem_convergence_Mabuchi_smt} 
Consider $u \in \PSH(X, \om)^T \cap \CC^0(X)$ such that $u$ is smooth on $X^\reg$ and $\H_{\mu_X}(u) < +\infty$. 
Assume that $u_{\vep} \in \PSH(Y, \om_\vep)^T \cap \CC^{0}(Y)\cap \CC^\infty(Y\setminus E)$ converges smoothly to $\pi^* u$ in $Y \setminus E$, uniformly on $Y$  and $\H_{\vep, \om^n_Y} (u_\vep)\rightarrow \H_{\om^n_Y}(\pi^* u)$ as $\vep \to 0$. 
Then $\M_{\vep}^{\rm rel}(u_\vep) \rightarrow\M^{\rm rel}_X (u)$ as $\vep \rightarrow 0$. 
\end{lem}

\begin{proof}
The proof follows a similar approach to that of \cite[Prop.~4.16]{Boucksom_Jonsson_Trusiani_2024} where the authors assume $u$ to be smooth on $X$ instead of $X^\reg$. 
Here, we only provide a brief outline and indicate the necessary modifications. 

\smallskip
Set $D = K_Y - \pi^* K_X = \sum_i a_i E_i$ and define $\rho_{Y/X} = \log \frac{\pi^\ast \mu_X}{\om_Y^n}$. 
Then we have 
\[
    [D] - dd^c \rho_{Y/X} = \pi^* \Ric(\mu_X) - \Ric(\om_Y^n).
\]
For $u \in \PSH(X, \omega) \cap L^\infty(X)$ and any smooth $(1,1)$-form $\theta$ (or positive closed $(1,1)$-current), 
\[
    \E^{\theta}_{\pi^* \om}(\pi^* u) := \sum_{j=0}^{n-1} \int_Y (\pi^*u) \theta\wedge (\pi^*\om + dd^c\pi^* u)^j \wedge (\pi^* \om)^{n-1-j}.
\]
Note that $\E^{[D]}_{\pi^* \om}(\pi^* u) = 0$ for any $u \in \PSH(X,\om) \cap L^\infty(X)$. 
Following the same argument in \cite[Lem.~4.15]{Boucksom_Jonsson_Trusiani_2024}, one can infer that for any $u \in \PSH(X, \om)\cap L^\infty(X)$ and smooth on $X^\reg$ we have
\[
    (\H_{ \om_Y^n, v} (\pi^* u) + \E^{-\Ric(\omega_Y)}_{\pi^*\om}(\pi^*u ))-( \H_{\mu_X, v}(u)+ \E^{-\Ric(\mu_X)}_{\om}(u))  = \E^{[D]}_{\pi^*\om} (\pi^* u)=0;
\]
therefore,
\begin{equation}\label{eq_ent_ener}
    \H_{ \om_Y^n, v} (\pi^* u) + \E^{-\Ric(\omega_Y)}_{\pi^*\om}(\pi^*u )= \H_{\mu_X, v}(u)+ \E^{-\Ric(\mu_X)}_{\om}(u).
\end{equation}

\smallskip
Since $u_\vep$ converges uniformly to $\pi^* u$, we have 
$\E_{\om_\vep}^{-\Ric(\om_Y)} (u_\vep) \to \E_{\pi^* \om}^{-\Ric(\om_Y)} (\pi^*u)$ as $\vep \to 0$. 
Moreover, since $\ell^{\ext}_\vep \rightarrow \ell^{\ext}$ in $\frak{t} \oplus \BR$, we have 
$\E_{\vep, vw\ell_\vep^{\ext}} (u_\vep) \to \E_{\om, vw\ell^{\ext}} (u)$. 
Since $u_\vep \to \pi^\ast u$ locally smoothly on $Y \setminus E$, we have $v (m_{u_\vep})$ converges to $v(m_{\pi^*u})$ almost everywhere. 
In addition, $\H_{\vep, \om_Y^n} (u_\vep)\rightarrow \H_{\om_Y^n}(\pi^* u)$, so Lemma \ref{lem_convergence_entropy} implies $\H_{\vep, \om_Y^n, v} (u_\vep)\rightarrow \H_{\om_Y^n, v}(\pi^* u)$. 
All in all, these yield 
\begin{align*}
    \lim_{\vep \to 0} \M_{\vep}^{\rm rel}(u_\vep) &= \H_{\om_Y^n,v}(\pi^* u) + \E_{\pi^* \om}^{-\Ric(\om_Y)} (\pi^*u) + \E_{\om, vw\ell^{\ext}} (u)\\
    &= \H_{\mu_X, v}(u)+ \E^{-\Ric(\mu_X)}_{\om}(u) + \E_{\om, vw\ell^{\ext}} (u) 
    = \M_X^{\rm rel} (u),
\end{align*}
as required, where the second equality follows from \eqref{eq_ent_ener}.
\end{proof}
}

\subsubsection{Singular cscK and extremal metrics} 
We begin by considering the problem of finding a singular cscK metric.

\begin{proof}[Proof of  Corollary \ref{bigthm:cscK}]
Take $T \subset \Aut_{\red}(X) $ to be a maximal torus and $\pi: Y\rightarrow X$ to be the $T$-equivariant resolution of singularities in Condition \ref{cond_A}. For cscK metrics, we have $v=1$, $w=1$, and $\ell^{\ext}= \bar{s}$ where $\bar{s} := \frac{n}{V_\om} c_1(X)\cdot \{\om\}^{n-1}$, so   
$\M_\om=\M_{v,\ell^{\ext}}$. By the hypothesis, $\M_{v,\ell^{\ext}}$ is  $T_\BC$-coercive, the existence of a singular cscK metric on $X$ now follows from Theorem \ref{thm_existence}. In this case, $\ell^{\ext}_\vep $ is some affine function on $\frak{t}^\vee$, converging to $\ell^{\ext}= \bar{s}$ as $\vep \to 0$ by \cite[Lem.~4.18]{Boucksom_Jonsson_Trusiani_2024}.
Hence, this singular cscK metric on $X$ is approximated by extremal metrics in $[\omega_\vep]$.
\end{proof} 

Let $X$ be a normal compact K\"ahler variety with log terminal singularities. 
Fix a K\"ahler class $\alpha$ and take a maximal torus $T \subset \Aut_{\red}(X)$.   Let $\pi: Y\rightarrow X$ be a $T$-equivariant resolution of singularities. 
We now consider the problem of the existence of singular extremal metrics in the K\"ahler class $\alpha$.

\smallskip
Let $\xi_{\ext}$ be the extremal vector field defined by $T$ and $\alpha$ (cf. Section \ref{eg_extremal}) and let $\om \in \alpha$ be a $T$-equivariant K\"ahler metric. 
A singular extremal metric in this setting is defined as a positive current of the form
$\pi^*\om + dd^c \vph$ where $\vph \in \PSH(Y, \pi^\ast \om) \cap L^\infty (Y)$ and $\vph$ is smooth away from $\Exc(\pi)$. 
Additionally, $\pi^* \om + dd^c \vph$ is a genuine $(1, \ell^{\ext})$-cscK metric on $Y \setminus \Exc(\pi)$ where $\ell^{\ext}(p) = \langle \xi_{\ext}, p \rangle+ \bar{s}$ , where $\bar{s}:= \frac{n}{V_\om} c_1(X)\cdot \{\om\}^{n-1}$. 

\smallskip
Under Condition \ref{cond_A}, Theorem~\ref{thm_existence} implies the following existence result of singular extremal metrics: 

\begin{thm}\label{thm_extremal}
Under the above setting, moreover, assume that $X$ satisfies Condition~\ref{cond_A}. 
If the weighted Mabuchi functional $\M_{\omega, v, \ell^{\ext}}$ on $X$ is $T_\BC$-coercive, then $X$ admits a singular extremal metric in $\alpha$. 
\end{thm}

\subsection{Constructing examples of singular cscK metrics}\label{sec:examples}

We shall give a way to construct examples of singular cscK metric in the spirit of Arezzo--Pacard \cite{Arezzo_Pacard_2006}   (see also \cite{Arezzo_Pacard_2009,Arezzo_Pacard_Singer_2011,Szekelyhidi_2015}) and use variational argument and our existence result. 
Before illustrating the process, we need the following lemma: 

\begin{lem}\label{lem:openness_coercivity}
Let $(X, \om)$ be a compact K\"ahler variety with log terminal singularities and let $f: Y \to X$ be a blowup along a compact submanifold $S \subset X^{\reg}$ of codimension $\geq 2$. 
Consider $\om_Y$ a K\"ahler metric on $Y$ and $\om_\vep = f^\ast \om + \vep \om_Y$ for $\vep \in [0,1]$. 
If $\M_\om$ is coercive, then $\M_{\om_\vep}$ is coercive for all sufficiently small $\vep > 0$. 
\end{lem}

\begin{proof}
The proof follows the same strategy in \cite[Thm.~4.11]{Pan_To_Trusiani_2023} and \cite[Thm.~A]{Boucksom_Jonsson_Trusiani_2024}.
Since $f$ is a blowup of $S$, for $m \in \BN^\ast$ so that $mK_Y$ and $mK_X$ are both $\BQ$-Cartier, there is a smooth hermitain metric $h_Y$ of $mK_Y$ such that 
\begin{equation}\label{eq:lt_modifi_cond}
    -\frac{1}{m}\Ta(mK_Y,h_Y) \geq - D f^\ast \om
\end{equation}
for some constant $D \geq 0$. 
Denote by $\Ta_Y = \frac{1}{m}\Ta(mK_Y,h_Y)$ and $\Ta_X = \frac{1}{m}\Ta(mK_X,h_X)$ for some hermitian metric $h_X$ on $mK_X$. 
We also set $\mu_Y$ (resp. $\mu_X$) to be the corresponding probability measure of $h_Y$ (resp. $h_X$). 

\smallskip
Recall that from \cite[Sec.~4]{Pan_To_Trusiani_2023}, 
\[
    \M_{\om_\vep} 
    = \H_{\mu_Y,\om_\vep} + \bar{s}_\vep \E_{\om_\vep} - n \E_{\Ta_Y,\om_\vep} - C_\vep
    \quad\text{on } \CE^1(Y,\om_\vep)
\]
and
\[
    \M_\om 
    = \H_{\mu_X,\om} + \bar{s} \E_{\om} - n \E_{\Ta_X,\om} - C
    \quad\text{on } \CE^1(X,\om)
\]
where $\bar{s}_\vep = \frac{n c_1(Y) \cdot [\om_\vep]^{n-1}}{[\om_\vep]^n}$, $\bar{s} = \frac{n c_1(X) \cdot [\om]^{n-1}}{[\om]^n}$, $C_\vep = \H_{\mu_Y,\om_\vep}(0) = \int_Y \log \lt(\frac{\om_\vep^n}{\mu_Y}\rt) \mu_Y$ and $C = \H_{\mu_X,\om_\vep}(0) = \int_X \log \lt(\frac{\om^n}{\mu_X}\rt) \mu_X$. 
Let $A_0, B_0>0$ be two constant such that $\M_{\om} \geq A_0 (-\E_\om) - B_0$ on $\CE^1_{\nmlz}(X,\om)$. 
We claim that for all $0 < A < A_0$ there are $\vep_0 > 0$ and $B > 0$ such that for all $\vep \in (0, \vep_0)$, 
\[
    \M_{\om_\vep} \geq A(-\E_{\om_\vep}) - B
\]s
on $\CE^1_{\nmlz}(Y, \om_\vep)$.

\smallskip
Suppose otherwise, for an $A \in (0, A_0)$, there are $\vep_k \to 0$, $B_k \to +\infty$, as $k \to +\infty$, and $u_k \in \CE^1_{\nmlz}(Y,\om_{\vep_k})$ such that 
\[
    \M_{\om_{\vep_k}}(u_k) < A(-\E_{\om_{\vep_k}} (u_k)) - B_k.
\]
Without loss of generality, one can assume that $u_k$ is bounded. 
Otherwise, by \cite[Lem.~3.4]{Pan_To_Trusiani_2023}, one can find a sequence bounded $\om_k$-psh functions $(v_{k,j})_j$ converging strongly to $u_k$ and their entropies also converge to $\H_{\mu_Y, \om_{\vep_k}}(u_k)$. 
Then for sufficiently large $j$, we have $\M_{\om_{\vep_k}}(v_{k,j}) < A(-\E_{\om_{\vep_k}}(v_{k,j})) - B_k$. 

\smallskip
Note that from \eqref{eq:lt_modifi_cond}, for all $\psi \in \CE_{\nmlz}^1(Y,\om_\vep)$, 
\begin{align*}
    -n \E_{\Ta_Y,\om_{\vep}}(\psi)
    &= - \frac{1}{V_\vep} \sum_{j=0}^{n-1} \int_Y \psi \Ta_Y \w \om_{\vep,\psi}^j \w \om_\vep^{n-1-j}\\
    &\geq \frac{1}{V_\vep} \sum_{j=0}^{n-1} \int_Y \psi D f^\ast \om \w \om_{\vep,\psi}^j \w \om_\vep^{n-1-j}
    \geq (n+1)D \E_{\om_\vep}(\psi).
\end{align*}
Hence, 
\[
    \H_{\mu_Y, \om_\vep}(u_k) + (\bar{s}_{\vep_k} + (n+1)D) \E_{\om_{\vep_k}}(u_k)
    \leq \M_{\om_{\vep_k}}(u_k) < A(-\E_{\om_{\vep_k}}(u_k)) - B_k.
\]
After enlarging $D$, one may assume that for all $\vep \in [0,1]$, $(n+1) D + A - \bar{s}_\vep > \dt$ for a uniform $\dt > 0$.  
Then we have $B_k < ((n+1)D + A - \bar{s}_\vep) (-\E_{\om_{\vep_k}}(u_k))$ and this implies that $d_k := -\E_{\om_{\vep_k}}(u_k) \to + \infty$ as $k \to +\infty$.
Let $g_k(s)$ be the $d_1$-geodesic connecting $0$ and $u_k$ in $\CE^1_{\nmlz}(Y,\om_{\vep_k})$. 
Fix a constant $R > 0$ and define $v_k := g_k(R)$.
By the convexity of Mabuchi functional \cite[Prop.~4.7]{Pan_To_Trusiani_2023}, 
\[
    \M_{\om_{\vep_k}}(v_k) \leq \frac{d_k - R}{d_k} \M_{\om_{\vep_k}}(0) + \frac{R}{d_k} \M_{\om_{\vep_k}}(u_k)
    \leq \frac{R}{d_k} (A d_k - B_k) \leq AR.
\]
From the expression of Mabuchi functional, we have $\H_{\mu_Y, \om_{\vep_k}}(v_k) \leq ((n+1)D + A - \bar{s}_{\vep_k})R$. 

\smallskip
In the argument below, although $Y$ is singular, corresponding proofs in  \cite{Boucksom_Jonsson_Trusiani_2024} proceed exactly the same.
By the strong compactness \cite[Thm.~2.10]{Boucksom_Jonsson_Trusiani_2024}, up to a subsequence, $v_k$ converges in $v_0 \in \CE^1_{\nmlz}(Y, f^\ast \om)$ and $\E_{\om_{\vep_k}}(v_k) \to \E_{f^\ast \om} (v_0)$ as $k \to +\infty$.
Note that $v$ can descend to a function in $\CE^1_{\nmlz}(X,\om)$, which we still denote by $v$.  
By \cite[Lem.~4.16]{Boucksom_Jonsson_Trusiani_2024}, $\H_{\mu_X,\om}(v) - n\E_{\Ta_X,\om}(v) \leq \H_{\mu_Y,f^\ast\om}(v) - n\E_{\Ta_Y,f^\ast\om}(v)$
and \cite[Lem.~4.6 and Lem.~4.9]{Boucksom_Jonsson_Trusiani_2024} shows that $\M_{\om_{0}}(v_0) \leq \liminf_{k \to +\infty} \M_{\om_{\vep_k}}(v_k)$.
All in all, we obtain
\[
    A_0 R - B_0 \leq \M_{\om}(v_0) \leq \liminf_{k \to +\infty} \M_{\om_{\vep_k}}(v_k) \leq A R.
\]
Letting $R = \frac{B_0}{A_0 - A} + 1$, this yields a contradiction. 
\end{proof}

\subsubsection{Construct singular cscK on blowups of singular KEs} \label{sect_examples}
Let $(X,\om)$ be a compact K\"ahler variety with log-terminal singularity. 
Suppose that $K_X$ is $m$-Cartier for some $m \in \BN^\ast$, and pick a hermitian metric $h$ on $mK_X$.
Assume that either
\begin{itemize}
    \item $K_X$ is ample and $\om = \frac{\ii}{m}\Ta(mK_X, h)$; or
    \item $K_X$ numerically trivial and $\frac{\ii}{m}\Ta(mK_X, h) = 0$; or else
    \item $K_X$ is anti-ample, $\om = -\frac{\ii}{m}\Ta(mK_X, h)$, and $(X, K_X)$ is K-stable.
\end{itemize}
Note that in the above cases, $\{\om\}$ contains a unique singular K\"ahler--Einstein metric and the Mabuchi function with respect to $\om$ is coercive. 
We further assume that $X$ admits a resolution of Fano type $\pi: \widehat{X} \to X$. 

\smallskip
Let $f: Y \to X$ be a blowup of $N$ distinct points in the smooth locus of $X$. 
Denote the irreducible components of the exceptional divisor of $f$ by $(D_i)_i$. 
Since $f$ is an isomorphism near the singularities of $Y$ and $X$, one can verify that $Y$ also admits a resolution of Fano type. 

\smallskip
For any $a := (a_1,\cdots,a_N) \in \BR_{>0}^n$, there exists a constant $\dt_a > 0$, such that for all $\dt \in (0,\dt_a)$, the class $\{f^\ast\om\} - \dt \sum_{i=1}^N a_i c_1(\CO(D_i))$ contains a K\"ahler metric $\om_{Y,a,\dt}$. 
By Lemma~\ref{lem:openness_coercivity}, the Mabuchi functional with respect to $f^\ast \om + \vep \om_{Y,a,\dt}$ is coercive for all sufficiently small $\vep$.  
Corollary~\ref{bigthm:cscK} then ensures the existence of a singular cscK metric in the class $(1+\vep)\{f^\ast\om\} - \vep\dt \sum_{i=1}^N a_i c_1(\CO(D_i))$ on $Y$. 

\smallskip
If $X$ is a K\"ahler variety with log terminal singularities, $\Aut(X)$ is discrete and it admits a crepant resolution $\pi: Y \to X$, it is not difficult to check that $\pi: Y \to X$ satisfies Condition~\ref{cond_A}. 
In the case of K\"ahler--Einstein varieties with log terminal singularities and discrete automorphism groups that admit crepant resolutions, the above construction provides a method to produce numerous singular cscK metrics on their blowups at points within the smooth locus.

\smallskip
In dimension two, all surfaces with canonical singularities admit crepant resolutions. 
In dimension three, the famous result of \cite[Thm. 1.2]{Bridgeland_King_Reid_2001} establishes that singular varieties locally modeled on $\BC^3/G$, where $G < \RSL(3,\BC)$ is finite, also admit crepant resolutions. 
Three-dimensional ODP singularity also admits a crepant resolution, and it is not a quotient singularity. 

\smallskip
We also extract the following example from \cite[Rmk.~34]{Szekelyhidi_2024} and \cite[Example 4.10]{Boucksom_Jonsson_Trusiani_2024} that is not a crepant resolution.

\begin{eg}[Isolated cone singularities] \label{eg_iso_sing}
Let $V$ be a {smooth} projective variety and let $L$ be an ample line bundle on $V$. 
Set $C_a(V, L):= \Spec \sum_{m \geq 0} H^0(V, L^m)$ the corresponding affine cone. 
Assuming $K_V\sim_{\BQ} r\cdot L$ for some $r\in \BQ$, by \cite[Lem.~3.1]{Kollar_2013}, $C_{a}(V, L)$ is klt if and only if $r<0$. 
Moreover, $C_{a}(V, L)$ is canonical if and only if $r \leq -1$.
Therefore, one can choose $r\in (-1, 0)$ to get a klt isolated singularity which is not canonical.

\smallskip
Assume that $X$ has klt isolated singularities, and each singular point is locally isomorphic to an affine cone $C_a(V,L)$ where $V$ is a Fano manifold.  
Then blowing up the singularities yields a resolution of singularities such that $\pi: Y\rightarrow X$ is projective.  In this case, $-K_Y = -\pi^* K_X-\sum_i a_i E_i$ is $\pi$-nef if and only if $X$ has canonical singularities, i.e. $a_i\geq 0$.  
When $-K_Y$ is $\pi$-nef, Condition~\ref{cond_A} holds by choosing $\rho_1=0$.

\smallskip
We now consider a slightly more general situation where $-K_Y$ is not $\pi$-nef, i.e. $r\in (-1, 0)$ in the setting above and explain Condition \ref{cond_A} in this setting.
Let $\mu_X$ be a $T$-invariant adapted measure on $X$ (cf. Definition \ref{def_adapted_measure}) and let $\mu_Y$ be a $T$-invariant volume form on $Y$. 
Recall that $\Ric(\mu_X)$ is defined in Section \ref{sect_singular_recall}, which is a smooth form on $X$, so $\Ric(\mu_X)\geq -K_1\omega$ for some $K_1>0$.  
Remark that for any $T$-invariant function $\psi_0 = -\sum_{i} a_i\log|s_i|^2_{h_i}$ for some smooth hermitian metric $h_i$ on $\mathcal{O}_X(E_i)$, we have $\sum_i a_i\Theta_{h_i}= \ddc \psi_0 +  \sum_i [E_i] \in c_1(K_Y - \pi^\ast K_X)$ since $K_Y = \pi^* K_X+\sum_i a_i E_i$. 
Therefore, by $\ddc$-lemma, $\ddc \psi_0 +  \sum_i [E_i] = \sum_i a_i \Ta_{h_i} = \pi^\ast \Ric(\mu_X)-\Ric(\mu_Y)+ \ddc f$ for some $f\in C^\infty(Y)$; hence 
\begin{equation}\label{eq_ric_pi_ric}
    \Ric(\mu_Y)+ dd^c \psi = \pi^\ast \Ric(\mu_X) - \sum_i a_i  [E_i],
\end{equation}
where $\psi:= - f+\psi_0=-f -\sum_{i} a_i\log|s_i|^2_{h_i}$ for some smooth hermitian metric $h_i$ on $\mathcal{O}_X(E_i)$. 
Since $X$ has only klt isolated cone singularities, we have a resolution by a single blow-up at each isolated point; hence, $E_i$'s are disjoint, so there exists $K_2>0$ such that (cf. the construction in \cite[Lem.~2.2]{Coman_Ma_Marinescu_2017})
\begin{equation}
  K_2\{\pi^\ast \omega\} -\sum_{a_i>0}a_i c_1(\mathcal{O}_X(E_i)) \geq 0.   
\end{equation}
Therefore, there exist a smooth hermitian metrics $h'_i$ on $\mathcal{O}_X(E_i)$, such that 
  \begin{equation}\label{eq_K_1}
    - \sum_{a_i>0 } a_i [E_i] + \sum_{a_i>0 } dd^c a_i\log|s_i|^2_{h'_i}    \geq - K_2\pi^\ast \omega.
\end{equation}
Using \eqref{eq_ric_pi_ric}, \eqref{eq_K_1} and the fact that $\Ric(\mu_X)\geq -K_1\omega$ on $X$, we get
\begin{align*}
    \Ric(\mu_Y)+ dd^c (\psi+ \sum_{a_i>0}a_i\log|s_i|^2_{h'_i}) &=   \pi^\ast \Ric(\mu_X) - \sum_i a_i  [E_i]+ dd^c\sum_{a_i>0}a_i\log|s_i|^2_{h'_i}  \\
   &= \pi^\ast \Ric(\mu_X) - \sum_{a_i<0} a_i  [E_i]  + ( - \sum_{a_i>0 } a_i [E_i]   +  \sum_{a_i>0 } dd^c a_i\log|s_i|^2_{h'_i} ) \\
    & \geq -K_1 \pi^*\omega -K_2 \pi^*\omega =: -K\pi^*\omega.
\end{align*} 
Therefore $ \Ric(\mu_Y)+ \ddc \phi\geq -K\pi^*\omega$ where $\phi:=\psi+ \sum_{a_i>0}a_i\log|s_i|^2_{h_i'}= u- \sum_{ a_j<0}a_j\log|s_j|^2_{h_j}$ for some $u \in \CC^\infty(Y)$.
One can obtain $\int_X e^{-\phi} \mu_Y^n < +\infty$ as $a_j>-1$. 
Hence, Condition \ref{cond_A} holds by taking $\rho_1 = \phi/K$.
\end{eg}

\subsubsection{Mixing construction with smoothing}
Let us stress that this construction also works on the $\BQ$-Gorenstein smoothable setting. 
Consider a $\BQ$-Gorenstein smoothing $f: (\CX,\om_\CX) \to \BD$ of $(X_0,{\om_{\CX}}_{|X_0})$ where $X_0$ is a K\"ahler--Einstein variety with log terminal singularities and $\Aut(X)$ is discrete, and $\{{\om_{\CX}}_{|X_0}\}$ is a K\"ahler--Einstein class. 
Denote by $\CZ$ the singular set of $f$. 
Take a finite set of points $\{p_1, \cdots, p_N\}$ in $X_0^\reg$. 
There exists smooth curves $C_1, \cdots, C_N$ in $\CX$ such that these curves are disjoint, each $C_i$ intersects $X_0$ transversely at the single point $p_i$, $C_i \subset \CX \setminus \CZ$, and the restriction $f_{|C_i}: C_i \to \BD$ is an isomorphism. 
Now consider the blowup map $\mu: \CY \to \CX$ along all the curves $(C_i)_i$ and let $Y_0 = \Bl_{\{p_1,\cdots,p_N\}}X$. 
Then $\pi = f \circ \mu: \CY \to \BD$ is a $\BQ$-Gorenstein smoothing of $(Y_0, {\om_{\CY,\vep}}_{|Y_0})$ where $\om_{\CY,\vep} = \mu^\ast \om_{\CX} + \vep \om_\CY$. 
Here $\om_\CY$ is a hermitian metric on $\CY$ and relatively K\"ahler that defined by
\[
    \om_\CY = \mu^\ast \om_\CX - \sum_{i = 1}^N a_i \Ta_{h_i}(\CO(E_i))
\]
where $E_i$ is the exceptional divisor over $C_i$, $a_i > 0$ and $h_i$ is some hermitian metric on $\CO(E_i)$. 
Since the Mabuchi functional $\M_{{\om_\CX}_{|X_0}}$ on $X_0$ is coercive, it follows from Lemma~\ref{lem:openness_coercivity} $\M_{{\om_{\CY,\vep}}_{|Y_0}}$ on $Y_0$ as well. 
By \cite[Thm.~C]{Pan_To_Trusiani_2023}, this implies the existence of a cscK metric in the class $\{{\om_{\CY,\vep}}|Y_0\}$. 

\smallskip
For examples of smoothable Calabi--Yau varieties, we refer the reader to \cite[Sec.~8]{Druel_Guenancia_2018}. 
For the smoothable Fano case, \cite{Liu_Xu_2019, Liu_2022} prove that mildly singular cubic varieties in dimensions three and four are K-stable. 
Additionally, explicit examples of K-stable singular cubic threefolds can be found, for instance, in \cite[Sec.~3, 4]{Cheltsov_Tschinkel_Zhang_2024_equivariant_1} and \cite[Sec.~5]{Cheltsov_Tschinkel_Zhang_2024_equivariant_2}.

\appendix
\section{Weighted Aubin--Yau inequality}\label{sec:appx_A}
In this section, for the reader's convenience,  we provide detailed proof of a Laplacian inequality (see also \cite[Lem.~5.6]{DiNezza_Jubert_Lahdili_2024}), which generalizes \cite[p. 98--99]{Siu_book} to the weighted setting.

\begin{lem}\label{lem:weighted_AY} 
Let $\om$ and $\om_X$ be two $T$-invariant K\"ahler metrics. 
Assume that $\vph \in \CH^T_\om$ satisfies
\[
    v(m_\vph)(\omega + dd^c \vph)^n = e^{F} \omega_X^n.
\]
Then there exist positive constants $\CB_0, \CC_0$ such that 
\begin{align*}
    \Dt_{\om_\vph, v} \log \tr_{\om_X} \om_\vph 
    &\geq \frac{\Delta_{\omega_X} F}{\tr_{\om_X} \om_\vph} - \CB_0 \tr_{\om_\vph}\om_X - \CC_0
    - \frac{1}{\tr_{\om_X} \om_\vph}
    \sum_{\af,\bt} (\log v)_{\af\bt}(m_\vph)  
    \langle d m^{\xi_\af}_\vph, d m^{\xi_\bt}_\vph \rangle_{\om_X}. 
\end{align*} 
In particular, if $\mathfrak{t}^\vee \ni p \mapsto \log v(p) \in \BR$ is concave, then 
\[
    \Dt_{\om_\vph, v} \log \tr_{\om_X} \om_\vph 
    \geq \frac{\Delta_{\omega_X} F}{\tr_{\om_X} \om_\vph} - \CB_0 \tr_{\om_\vph}\om_X - \CC_0.
\] 
We note that the constants $\CB_0 = B + \max_X |\Scal(\omega_X)|$ and $\CC_0 = C_v^2 (C_2 + C_v^2 C_\xi)$ where
\begin{itemize}
    \item $\Bisec(\om_X) \geq -B$ is a negative lower bound for the bisectional curvature of $\om_X$; 
    
    \item $C_v > 0$ is a constant such that
    \begin{equation}\label{eq:appx_v_bdd}
    \begin{split}
        C_v^{-1} 
        &\leq v + \sum_\af |v_\af| + \sum_{\af,\bt} |v_{\af\bt}| 
        \leq C_v \quad{\text{on $P = \im (m_\om)$}},\\
        C_v^{-1} 
        &\leq |v| + \sum_\af |v_\af| + \sum_{\af,\bt} |v_{\af\bt}| 
        \leq C_v \quad{\text{on $P_{\om_X} = \im (m_{\om_X})$}};
    \end{split}    
    \end{equation}

    \item $C_2$ depending only on $n$ and $C_1$ where $C_1 > 0$ is a constant so that for any $\af$,
    \begin{equation}\label{eq:appx_L_om_X}
        -C_1 \om_X \leq \CL_{J \xi_\af} \om_X \leq C_1 \om_X.
    \end{equation}
\end{itemize}
\end{lem}

\begin{rmk}
The concavity condition on $\log v(p)$ holds in many interesting cases. 
Here we extract some examples from \cite[Sec.~3]{Lahdili_2019}: 
\begin{itemize}
    \item cscK and extremal metrics: $v(p) = 1$,
    \item K\"ahler--Ricci solitons: $v(p)= e^{\langle \xi, p \rangle}$ for some fixed $\xi \in \frak{t}$,
    \item K\"ahler metrics given by the generalized Calabi construction: $v(p)=\Pi_j (\langle \xi_j, p \rangle + a_j)^{d_j} $ with $d_j>0$. 
\end{itemize}
\end{rmk}

\begin{proof}
Before entering the proof, we recall a basic equality of Lie derivative about the quotient of volume forms. 
Suppose that $\af$ is an $(n,n)$-form and $\bt$ is a volume form on $X$. 
For all vector fields $V$, one can derive the following
\begin{equation*}
    \CL_V \lt(\frac{\af}{\bt}\rt) = \frac{\CL_V \af}{\bt} - \frac{\af}{\bt} \cdot \frac{\CL_V \bt}{\bt}.
\end{equation*}
As a consequence, we have
{\small
\begin{equation}\label{eq:appx_L_quot_vol}
    \CL_V \lt(\frac{\om_\vph^n}{\om_X^n} \rt)
    = \frac{n (\CL_V \om_\vph) \w \om_\vph^{n-1}}{\om_\vph^n} - \frac{n (\CL_V \om_X) \w \om_X^{n-1}}{\om_X^n}\cdot \frac{\om_\vph^n}{\om_X^n}
\end{equation}
and 
\begin{equation}\label{eq:appx_L_tr}
    \CL_V \lt(\tr_{\om_X} \om_\vph\rt) 
    = \frac{n (\CL_V \om_\vph) \w \om_X^{n-1} + n(n-1) (\CL_V \om_X) \w \om_\vph \w \om_X^{n-2}}{\om_X^n} - (\tr_{\om_X} \om_\vph) \frac{n (\CL_V \om_X) \w \om_X^{n-1}}{\om_X^n}
\end{equation}
}
\smallskip
By the standard Aubin--Yau's inequality, we have
\begin{align*}
   \Dt_{\om_\vph} \log \tr_{\om_X} \om_\vph 
    &\geq - \frac{\tr_{\om_X} \Ric(\omega_\vph ) }{\tr_{\om_X} \om_\vph} - B \tr_{\om_\vph}\omega_X
    \\
    & =\frac{\Dt_{\om_X} (F - \log v(m_\vph))}{\tr_{\om_X} \om_\vph} - \frac{\tr_{\om_X} (\Ric(\omega_X))}{\tr_{\om_X} \om_\vph}   - B \tr_{\om_\vph} \om_X\\
    &\geq \frac{\Dt_{\om_X} (F - \log v(m_\vph))}{\tr_{\om_X} \om_\vph}   - B' \tr_{\om_\vph} \om_X,
\end{align*}
where $B' = \max_X |\Scal(\omega_X)| + B $,  and $ (\tr_{\om_\vph}\omega_X) (\tr_{\omega_X}\omega_{\vph}) \geq 1$,
and thus, 
\begin{equation}\label{eq:appx_AY_1}
\begin{split}
    &\Dt_{\om_\vph, v} \log \tr_{\om_X} \om_\vph\\ 
    &\geq \frac{\Dt_{\om_X} (F - \log v(m_\vph))}{\tr_{\om_X} \om_\vph} - B' \tr_{\om_\vph} \om_X +\tr_{\om_\vph} [d \log v(m_\vph) \w d^c \log \tr_{\om_X} \om_\vph]\\
    &= \frac{\Dt_{\om_X} F}{\tr_{\om_X} \om_\vph} - B' \tr_{\om_\vph} \om_X
    + \frac{1}{\tr_{\om_X} \om_\vph} \bigg\{\tr_{\om_\vph} [d \log v(m_\vph) \w d^c \tr_{\om_X} \om_\vph] - \Dt_{\om_X} \log v(m_\vph)\bigg\}.
\end{split}
\end{equation}

\smallskip
It suffices to establish a suitable lower bound for the terms in $\{\cdots\}$ in \eqref{eq:appx_AY_1}.
Note that $P$ does not depend on $\vph$ and $d^c = J^{-1} \circ d \circ J$.  
By Cartan's formula, $i_V d^c f = - i_{JV} df = - \CL_{JV} f$.
Using \eqref{eq:appx_L_tr}, \eqref{eq:appx_L_om_X} and \eqref{eq:appx_v_bdd}, we obtain the following:
{\small
\begin{equation}\label{eq:appx_tr_dlog_dtr}
\begin{split}
    &\tr_{\om_\vph} (d \log v(m_\vph) \w d^c \tr_{\om_X} \om_\vph)
    = \sum_\af (\log v)_\af(m_\vph) i_{\xi_\af} d^c \tr_{\om_X} \om_\vph 
    = - \sum_\af (\log v)_\af(m_\vph) (\CL_{J\xi_\af} \tr_{\om_X} \om_\vph)\\
    &= - \sum_\af (\log v)_\af(m_\vph) \lt(\frac{n(\CL_{J\xi_\af}\om_\vph) \w \om_X^{n-1} + n(n-1) (\CL_{J \xi_\af} \om_X) \w \om_\vph \w \om_X^{n-2}}{\om_X^n}
    - \tr_{\om_X} \om_\vph \cdot \frac{n (\CL_{J \xi_\af} \om_X) \w \om_X^{n-1}}{\om_X^n}\rt)\\
    &\geq 
    - \sum_\af (\log v)_\af(m_\vph) \lt(\frac{n(\CL_{J\xi_\af}\om_\vph) \w \om_X^{n-1}}{\om_X^n}\rt) - C_2 C_v^2 \tr_{\om_X} \om_\vph \\
    &= \sum_\af (\log v)_\af(m_\vph) \Dt_{\om_X} m_\vph^{\xi_\af} - C_2 C_v^2 \tr_{\om_X} \om_\vph \\
    &= \frac{1}{v(m_\vph)} \lt(\Dt_{\om_X} v(m_\vph) - \sum_{\af,\bt} v_{\af\bt}(m_\vph) \langle d m_\vph^{\xi_\af}, d m_\vph^{\xi_\bt} \rangle_{\om_X}\rt) - C_2 C_v^2 \tr_{\om_X} \om_\vph \\
    &= \Dt_{\om_X} \log v(m_\vph) 
    -\sum_{\af,\bt} 
    (\log v)_{\af\bt}(m_\vph)
    \langle d m_\vph^{\xi_\af}, d m_\vph^{\xi_\bt} \rangle_{\om_X} - C_2 C_v^2 \tr_{\om_X} \om_\vph 
\end{split}
\end{equation}
}%
where $C_2 > 0$ is a constant depending only on $C_1$ and $n$.
To conclude, combining \eqref{eq:appx_AY_1} and \eqref{eq:appx_tr_dlog_dtr}, we finally obtain
{
\begin{align*}
    &\Dt_{\om_\vph,v} \log \tr_{\om_X} \om_\vph \geq \frac{\Dt_{\om_X} F}{\tr_{\om_X}\om_\vph} - B' \tr_{\om_\vph} \om_X  - C_2C_v^2 
    - \frac{1}{\tr_{\om_X} \om_\vph}
    \sum_{\af,\bt} 
    (\log v)_{\af\bt}(m_\vph) 
    \langle d m^{\xi_\af}_\vph, d m^{\xi_\bt}_\vph \rangle_{\om_X}
\end{align*}
}%
as required.
\end{proof}

\section{Weighted local Chen--Cheng's $\CC^2$-estimate} \label{sect_local_chen_cheng}
This section aims to give details for the $\CC^2$-estimate in Theorem~\ref{thm_local_higher_est}.
The proof follows a similar argument in \cite[Prop.~4.1]{Chen_Cheng_2021_1} with further analysis of the weighted terms. Recall that we have two local equations on $B_1(0)\subset \BC^n$,
\[
    v(m_{dd^c \phi}) \det (\phi_{i\bar{j}}) = e^G 
    \quad\text{and}\quad
    \Dt_{\phi,v} G = -S = -w(m_{dd^c \phi}).
\]
Here we denote $\Delta f := \sum_{k=1}^n \partial_{k} \partial_{\bar k} f$, $\Dt_{\phi,v} f:= \sum_{i,j=1}^n \phi^{i\bar j}  \partial_i \partial_{\bar j} f+ \frac{1}{2} \langle \dd \log v(m_{\ddc \phi}), \dd f \rangle_{\ddc \phi}$ and 
\begin{equation}\label{eq_grdient_prod}
    \tr_\phi(df_1 \w d^c f_2) = \frac{n (\ii \pl f_1 \w \db f_2 + \ii \pl f_2 \w \db f_1)\w (\ddc \phi)^{n-1}}{(\ddc \phi)^n} = \sum_{i,j} \phi^{i\bar{j}} \Re((f_1)_i (f_2)_{\bar{j}}),
\end{equation}
which are different from the scaling we used before. 
We first prove the following lemma:
\begin{lem}\label{lem:appxB}
There exist positive constants $K,C$ depending on $v,w, \|G\|_{L^\infty(B_1(0))}$ and $\|\phi\|_{L^\infty(B_1(0))}$ such that 
\[
    \Delta_{\phi, v}u\geq -C \cdot (\Delta\phi) \cdot u \quad\text{on $B_1(0)$}
\]
where $u = e^{\frac{G}{2}} |d G|_\phi^2 + K \Dt \phi$.
\end{lem}

\begin{proof}
We first remark that by arithmetic and geometric means inequality, one has {\small
\[
    \Dt \phi \geq n\det(\phi_{i\bar{j}})^{1/n} = n e^{G/n} v(m_{dd^c \phi})^{-1/n} \geq c
    \quad\text{and}\quad
    \tr_\phi \om \geq n \det(\phi_{i\bar{j}})^{-1/n}
    = n e^{-G/n} v(m_{dd^c \phi})^{1/n} \geq c.
\] }
where $\om$ is the Euclidean metric and $c > 0$ is a constant depending only on $n$, $\|G\|_{L^\infty(B_1(0))}$ and $C_v$.
We next review the following estimate by Chen and Cheng \cite[p.~16, (4.3)]{Chen_Cheng_2021_1}.  
Under the normal coordinates with respect to $\phi_{i\bar{j}}$, we may assume that $\om_{i\bar{j}}(x_0) = \dt_{ij}$, $\phi_{i\bar{j}}(x_0) = \ld_i \dt_{ij}$ and $\pl_k \phi_{i\bar{j}}(x_0) = 0$ at a fixed point $x_0$, then, at $x_0$, we have 
{\small
\begin{equation}\label{eq:appxB_CC}
\begin{split}
    e^{-\frac{G}{2}} \Dt_{\phi} (e^{\frac{G}{2}} |d G|_\phi^2)
    &= \Dt_{\phi} |\nabla^\phi G|_\phi^2 + \frac{1}{2} \phi^{i\bar{i}}(\pl_i G \cdot \pl_{\bar{i}} |\nabla^\phi G|_\phi^2 + \pl_{\bar{i}} G \cdot \pl_i |\nabla^\phi G|_\phi^2) + \frac{1}{4} |\nabla^\phi G|_\phi^4 + \frac{1}{2} \Dt_\phi G \cdot |\nabla^\phi G|_\phi^2 \\
    &= \phi^{i\bar{i}} (\pl_i \Dt_\phi G \cdot \pl_{\bar{i}} G + \pl_{\bar{i}} \Dt_\phi G \cdot \pl_i G) + \frac{1}{2}\Ric_{\phi}(\nabla^\phi G, \nabla^\phi G) + |\nabla^\phi \nabla^\phi G|_{\phi}^2 + |\nabla^\phi \overline{\nabla^\phi} G|_{\phi}^2 \\
    &\quad + \frac{1}{2} \phi^{i\bar{i}} \phi^{j\bar{j}} ( G_i G_j G_{\bar{i}\bar{j}} + G_i G_{\bar{j}} G_{j \bar{i}} + G_{\bar{i}} G_{\bar{j}} G_{ij} + G_{\bar{i}} G_j G_{i \bar{j}}) + \frac{1}{4} |\nabla^\phi G|_\phi^4 + \frac{1}{2} \Dt_\phi G \cdot |\nabla^\phi G|_\phi^2\\
    &\geq \phi^{i\bar{i}} (\pl_i \Dt_\phi G \cdot \pl_{\bar{i}} G + \pl_{\bar{i}} \Dt_\phi G \cdot \pl_i G) + \frac{1}{2}\Ric_{\phi}(\nabla^\phi G, \nabla^\phi G) + |\nabla^\phi \overline{\nabla^\phi} G|_{\phi}^2 \\ 
    &\quad + \frac{1}{2} \phi^{i\bar{i}} \phi^{j\bar{j}} (G_i G_{\bar{j}} G_{j \bar{i}} + G_{\bar{i}} G_j G_{i \bar{j}}) + \frac{1}{2} \Dt_\phi G \cdot |\nabla^\phi G|_\phi^2
\end{split}
\end{equation}
}%
where $\Ric_\phi := \Ric(dd^c \phi) = - dd^c \log \det(\phi_{i\bar{j}})$.
The last inequality comes from the fact that
{\small
\[
    \frac{1}{4} |\nabla^\phi G|_\phi^4 + \frac{1}{2} \phi^{i\bar{i}} \phi^{j\bar{j}} (G_i G_j G_{\bar{i}\bar{j}} + G_{\bar{i}} G_{\bar{j}} G_{ij}) + |\nabla^\phi \nabla^\phi G|_\phi^2 
    = \lt|G_{ij} + \frac{1}{2} \sqrt{\ld_i^{-1} \ld_j^{-1}} G_i G_j\rt|_\phi^2 \geq 0.
\] }
From the first equation, we have $\Ric(dd^c \phi) = - dd^c (G-\log v(m_{dd^c \phi}))$.
Rewrite \eqref{eq:appxB_CC} as follows
{\small
\begin{align*}
    e^{-\frac{G}{2}} \Dt_{\phi} (e^{\frac{G}{2}} |d G|_\phi^2)
    &\geq \phi^{i\bar{i}} (\pl_i \Dt_\phi G \cdot \pl_{\bar{i}} G + \pl_{\bar{i}} \Dt_\phi G \cdot \pl_i G) 
    -  \phi^{i\bar{i}} \phi^{j\bar{j}} G_{i\bar{j}} G_i G_{\bar{j}}\\
    &\quad + \phi^{i\bar{i}} \phi^{j\bar{j}} (\log v(m_{dd^c \phi}))_{i\bar{j}} G_i G_{\bar{j}} + \phi^{i\bar{i}} \phi^{j\bar{j}} G_{i\bar{j}} G_i G_{\bar{j}} + |\nabla^\phi \overline{\nabla^\phi} G|_{\phi}^2 + \frac{1}{2} \Dt_\phi G \cdot |d G|_\phi^2\\
    &= 2\tr_\phi (d \Dt_\phi G \w d^c G) 
    + \phi^{i\bar{i}} \phi^{j\bar{j}} (\log v(m_{dd^c \phi}))_{i\bar{j}} G_i G_{\bar{j}} + |\nabla^\phi \overline{\nabla^\phi} G|_{\phi}^2 + \frac{1}{2} \Dt_\phi G \cdot |d G|_\phi^2.
\end{align*}
}%
Considering the weighted version of the above inequality, one can infer 
{\small
\begin{equation}\label{appxB_first_term_1}
\begin{split}
    e^{-\frac{G}{2}} \Dt_{\phi,v} (e^{\frac{G}{2}} |d G|_\phi^2)
    &\geq 2\tr_\phi (d \Dt_\phi G \w d^c G) 
    + \phi^{i\bar{i}} \phi^{j\bar{j}} (\log v(m_{dd^c \phi}))_{i\bar{j}} G_i G_{\bar{j}} + |\nabla^\phi \overline{\nabla^\phi} G|_{\phi}^2 + \frac{1}{2} \Dt_\phi G \cdot |d G|_\phi^2 \\ 
    &\quad + \frac{1}{2} \tr_\phi (d \log v(m_{dd^c \phi}) \w d^c G) |d G|_\phi^2
    + \tr_\phi (d \log v(m_{dd^c \phi}) \w d^c |d G|_{\phi}^2) \\
    &= 2\tr_\phi (d \Dt_\phi G \w d^c G) 
    + \phi^{i\bar{i}} \phi^{j\bar{j}} (\log v(m_{dd^c \phi}))_{i\bar{j}} G_i G_{\bar{j}} + |\nabla^\phi \overline{\nabla^\phi} G|_{\phi}^2 + \frac{1}{2} \Dt_{\phi,v} G \cdot |d G|_\phi^2 \\ 
    &\quad + \tr_\phi (d \log v(m_{dd^c \phi}) \w d^c |d G|_{\phi}^2)
\end{split}
\end{equation}
}%
Under normal coordinates with respect to $\phi_{i\bar{j}}$ at $x_0$ and using \eqref{eq_grdient_prod}, we obtain 
{
\begin{align*}
    \tr_\phi (d \log v(m_{dd^c \phi}) \w d^c |d G|_{\phi}^2) 
    &= \phi^{i\bar{i}}  \Re \lt((\log v(m_{dd^c \phi}))_i [\phi^{j\bar{j}} G_j G_{\bar{j}}]_{\bar{i}}\rt)\\
    &=\phi^{j\bar{j}} \Re
    \Big(G_{\bar{j}} \phi^{i\bar{i}} (\log v (m_{\ddc\phi}))_i G_{j\bar{i}} 
    + G_j \phi^{i\bar{i}} (\log v(m_{\ddc\phi}))_i G_{\bar{j}\bar{i}} \Big),
\end{align*}
}%
  and 
\begin{align*}
 &2\tr_\phi(d G \w d^c [\tr_\phi (d \log v(m_{dd^c \phi}) \w d^c G)]) =  \phi^{i\bar i} \Big( G_i \partial_{\bar i} \Re ( \phi^{k \bar j} (\log v(m_{\ddc\phi}))_k G_{\bar j})  +  G_{\bar i} \partial_{i} \Re ( \phi^{k \bar j} (\log v(m_{\ddc\phi}))_k G_{\bar j})  \Big)  \\
 &= \phi^{i\bar i} G_i \phi^{j \bar j} \Re[\log v(m_{\ddc\phi}))_{j\bar i} G_{\bar j} + (\log v(m_{\ddc\phi}))_{j} G_{\bar j \bar i}  ] +  \phi^{i\bar i} G_{\bar i} \phi^{j \bar j}  \Re[\log v(m_{\ddc\phi}))_{j i} G_{\bar j} + (\log v(m_{\ddc\phi}))_{j} G_{\bar j  i}  ], 
\end{align*}
hence
\begin{align*}
\tr_\phi (d \log v(m_{dd^c \phi}) \w d^c |d G|_{\phi}^2) &= 2\tr_\phi(d G \w d^c [\tr_\phi (d \log v(m_{dd^c \phi}) \w d^c G)])
   \\
   &\quad - \Re\lt( \phi^{i\bar{i}}\phi^{j\bar{j}} (\log v(m_{dd^c \phi}))_{ij} G_{\bar{i}} G_{\bar{j}}\rt) 
    - \phi^{i\bar{i}}\phi^{j\bar{j}} (\log v(m_{dd^c \phi}))_{i\bar{j}} G_{\bar{i}} G_j.
\end{align*}
Therefore, the inequality \eqref{appxB_first_term_1} can be expressed as
{\small
\begin{equation}\label{appxB_first_term_2}
\begin{split}
    e^{-\frac{G}{2}} \Dt_{\phi,v} (e^{\frac{G}{2}} |d G|_\phi^2)
    &\geq 2 \tr_\phi (d \Dt_{\phi,v} G \w d^c G) 
    + |\nabla^\phi \overline{\nabla^\phi} G|_{\phi}^2 + \frac{1}{2} \Dt_{\phi,v} G \cdot |\nabla^\phi G|_\phi^2 \\ 
    &\quad - \Re\lt(\phi^{i\bar{i}}\phi^{j\bar{j}} (\log v(m_{dd^c \phi}))_{ij} G_{\bar{i}} G_{\bar{j}}\rt).
\end{split}
\end{equation} }
Recall that 
\[
    d m^{\xi_\af}_{dd^c \phi} 
    = - i_{\xi_\af} dd^c \phi 
    = - i_{\xi_\af} (2 \iii \phi_{p\bar{q}} dz^p \w d\bar{z}^q) 
    = - 2 \iii \phi_{p\bar{q}}(\xi_\af')^p d\bar{z}^q + 2 \iii \phi_{p\bar{q}}(\xi_\af'')^{\bar{q}} dz^p.
\]
Hence,  {\small
\begin{align}
\label{eq:appxB_D2log}
    &\pl_{ij}^2 \log v(m_{dd^c \phi}) 
    = \pl_i \lt\{\sum_\af  (\pl_\af \log v)(m_{\ddc\phi}) \cdot 2\iii \phi_{j\bar{l}} (\xi''_\af)^{\bar{l}}\rt\} \\
    &= 4\sum_{\af,\bt} (\pl_{\af\bt}^2 \log v)(m_{\ddc\phi}) \cdot [-\phi_{i \bar{q}} \phi_{j \bar{l}} (\xi''_\bt)^{\bar{q}} (\xi''_\af)^{\bar{l}}]  
    + 2\sum_\af (\pl_\af \log v)(m_{\ddc\phi}) \cdot \iii [\phi_{j\bar{l}i} (\xi''_\af)^{\bar{l}} + \phi_{j\bar{l}} (\xi''_\af)^{\bar{l}}_i]. \nonumber
\end{align} }
Note that $\xi'_{\af}$ is a holomorphic vector field for any $\af$. 
Therefore, $(\xi'_\af)^j_{\bar{i}} = 0 = (\xi''_\af)^{\bar{j}}_i$ for any $i,j \in \{1, \cdots, n\}$. 

\smallskip
Combining \eqref{appxB_first_term_2}, \eqref{eq:appxB_D2log}, and the fact that $(\xi''_\af)^{\bar{l}}_i = 0$, under normal coordinates with respect to $\phi$ at $x_0$, we get {\small
\begin{align*}
    e^{-\frac{G}{2}} \Dt_{\phi,v} (e^{\frac{G}{2}} |d G|_\phi^2)
    &\geq 2\tr_\phi (d \Dt_{\phi,v} G \w d^c G) 
    + |\nabla^\phi \overline{\nabla^\phi} G|_{\phi}^2 + \frac{1}{2} \Dt_{\phi,v} G \cdot |\nabla^\phi G|_\phi^2 \\ 
    &\quad + 4\Re \sum_{\af,\bt} (\pl_{\af\bt}^2 \log v)(m_{dd^c \phi}) (\xi''_\bt)^{\bar{i}} (\xi''_\af)^{\bar{j}} G_{\bar{i}} G_{\bar{j}} \\
    &\geq 2\tr_\phi (d \Dt_{\phi,v} G \w d^c G) 
    + |\nabla^\phi \overline{\nabla^\phi} G|_{\phi}^2 + \frac{1}{2} \Dt_{\phi,v} G \cdot |\nabla^\phi G|_\phi^2 
    - 4 C_v^4 \max_\af |\xi_\af''.G|^2\\ 
    &\geq 2 \tr_\phi (d \Dt_{\phi,v} G \w d^c G) 
    + |\nabla^\phi \overline{\nabla^\phi} G|_{\phi}^2 + \frac{1}{2} \Dt_{\phi,v} G \cdot |\nabla^\phi G|_\phi^2 - 4 C_v C_\xi |d G|_{\om}^2 
\end{align*} }
where $C_\xi > 0$ is a constant such that $|\xi|^2_{\om} \leq C_\xi$. 
Using the second equation $\Dt_{\phi,v} G = - w(m_{dd^c \phi})$, we then derive   
\begin{align*}
    \tr_\phi(d \Dt_{\phi,v} G \w d^c G)
    &= -\tr_\phi(d w(m_{dd^c \phi}) \w d^c G) 
    = - \sum_\af w_\af(m_{dd^c \phi}) \tr_\phi(d m_{dd^c \phi}^{\xi_\af} \w d^c G) \\
    &= \sum_\af w_\af(m_{dd^c \phi}) 
    \Re (\xi'_\af)^i G_i
    \geq -C_\xi C_w |d G|_{\om}\\
    &\geq -C_1 -|d G|^2_{\om} \geq - C_1 - (\Dt\phi) \cdot |d G|_{\phi}^2
\end{align*}
where $C_1 >0$ depending only on $C_\xi, C_v$. 
Then
\begin{equation}\label{eq:appxB_grad_G}
    e^{-\frac{G}{2}} \Dt_{\phi,v} (e^{\frac{G}{2}} |d G|_\phi^2)
    \geq - 2C_1 - C_2 (\Dt \phi) |d G|_\phi^2 + |\nabla^\phi \overline{\nabla^\phi} G|_\phi^2
\end{equation}
where $C_2 > 0$ depend only on $C_\xi, C_v, C_w$.

\smallskip
Next, we compute under normal coordinates with respect to $\omega$, chosen so that $\phi_{i\bar{j}}$ is diagonal. 
We obtain  
\begin{align*}
\Delta (G - \log v(m_{dd^c \phi})) &= \partial_i \partial_{\bar i} (\log \det  (\phi_{k\bar \ell})) \\
&= \Delta_{\phi} (\Delta \phi) - \frac{|\phi_{ p\bar q i} |^2}{\phi_{p\bar p} \phi_{q\bar q}}\leq    \Delta_{\phi} (\Delta \phi),
\end{align*}
and therefore, 
\begingroup
\allowdisplaybreaks {\small
\begin{align*}
    \Dt_{\phi,v} (\Dt \phi)
    &= \Dt_\phi (\Dt \phi) + \tr_\phi(d \log v(m_{dd^c \phi}) \w d^c \Dt \phi) \\
    &\geq \Dt G - \Dt \log v(m_{dd^c \phi}) + \tr_\phi (d \log v(m_{dd^c \phi}) \w d^c \Dt \phi) \\
    &\geq \Dt G - \sum_{\af,\bt} (\pl_{\af\bt}^2 \log v)(m_{dd^c \phi}) \langle d m_{dd^c \phi}^{\xi_\af}, d m_{dd^c \phi}^{\xi_\bt} \rangle - C_3 \Dt \phi\\
    &\geq \Dt G 
    - 2 C_v^2 (\Dt \phi)^2 
    - C_3 \Dt \phi, 
\end{align*} }
\endgroup
where $C_3$ depend only on $n$ and a constant $C>0$ such that $-C \om \leq \CL_{\xi_\af} \om \leq C \om$. 
The fourth line follows from the same estimate \eqref{eq:appx_tr_dlog_dtr}.

\smallskip
Then we obtain
\begin{align*}
    \Dt_{\phi,v} (e^{\frac{G}{2}} |d G|_\phi^2 + K \Dt \phi)
    &\geq -C_1 e^{\frac{G}{2}} - C_2 (\Dt \phi) e^{\frac{G}{2}} |d G|_\phi^2 
    + e^{\frac{G}{2}} |\nabla^\phi \overline{\nabla^\phi} G|_\phi^2 
    + K \Dt G
    - 2C_v^2 (\Dt \phi)^2 - C_3 \Dt \phi.
\end{align*}
Recall that $|\nabla^\phi \overline{\nabla^\phi} G|_\phi^2 = \sum_{i,j} \phi^{i\bar{i}} \phi^{j\bar{j}} |G_{i\bar{j}}|^2$ and $\|G\|_{L^\infty}\leq C_4$. 
By Cauchy--Schwarz inequality, we have the following estimate {\small
\begin{align*}
   e^{\frac{G}{2}} |\nabla^\phi \overline{\nabla^\phi} G|_\phi^2 +  e^{C_4}\frac{K^2}{4} (\Dt \phi)^2 
    \geq   e^{-C_4}\sum_{i,j} \phi^{i\bar{i}} \phi^{j\bar{j}} |G_{i\bar{j}}|^2 + e^{C_4}\frac{K^2}{4} \sum_{i,j} \phi_{i\bar{i}} \phi_{j\bar{j}} 
    \geq \sum_{i,j} K |G_{i\bar{j}}| 
    \geq K \sum_{i} |G_{i\bar{i}}| \geq K |\Dt G|.
\end{align*} }
Moreover, since $\Dt \phi \geq c$; hence, we derive the following estimate
\begin{align*}
\Dt_{\phi,v}u &=\Dt_{\phi,v} (e^{\frac{G}{2}} |d G|_\phi^2 + K \Dt \phi)
    \geq -C_1 e^{\frac{G}{2}} - C_2 (\Dt \phi) e^{\frac{G}{2}} |d G|_\phi^2 
      -C_5(\Delta \phi)^2-C_3 \Delta \phi\\
    &\geq   -C_2 \Delta\phi (e^{\frac{G}{2}} |d G|_\phi^2   )  -C_6 (\Delta\phi)^2  \quad \text{ ($C_6$ depending on $c, C_1, C_3, \|G\|_{L^\infty}$)}\\
    &\geq -C_2(\Delta\phi ) \left(e^{\frac{G}{2}} |d G|_\phi^2 + K \Dt \phi\right)
    = -C_2 \cdot (\Delta \phi) \cdot u.
\end{align*}
for $K =C_6 C_2^{-1}$ which only depends on $C_v, C_w$ and $\|G\|_{L^\infty(B_1(0))}$.
\end{proof}

Using Lemma~\ref{lem:appxB}, one can obtain Theorem~\ref{thm_local_higher_est} following an argument in Chen--Cheng's article \cite{Chen_Cheng_2021_1}: 

\begin{proof}[Proof of Theorem~\ref{thm_local_higher_est}]

\smallskip
Let $\eta$ be a positive smooth function with compact support in $B_1(0)$, $\eta \equiv 1$ on $B_{3/4}(0)$ and $0 \leq \eta \leq 1$. 
We first claim that there is a positive constant 
\[
    D_1 = D_1(n, \|G\|_{L^\infty(B_1(0))}, v, w, \|\Dt\phi\|_{L^{n+1}(B_1(0))}, C_\xi, C_\eta)
\] 
such that $\||dG|_\phi^2\|_{L^1(B_{3/4}(0))} \leq D_1$, where $C_\eta > 0$ is a constant such that $|d \eta|_{\om}, |\ddc \eta|_{\om} \leq C_\eta$ and $C_\xi > 0$ is a constant such that $|\xi_\af|_{\om} \leq C_\xi$.

\smallskip
To see the claim, set $C_G > 0$ a constant such that $|G| \leq C_G$ and $e^{-G} \leq C_G$ on $B_1(0)$.
Then we have 
\begingroup
\allowdisplaybreaks
{\small
\begin{align*}
    &\int_{B_{3/4}(0)} |dG|_\phi^2 \om^n
    \leq \int_{B_1(0)} \eta |dG|_\phi^2 e^{-G} v(m_{\ddc\phi}) (\ddc\phi)^n
    \leq C_G C_v \int_{B_1(0)} \eta dG \w \dc G \w (\ddc\phi)^{n-1} \\
    &= C_G C_v \int_{B_1(0)} - \eta G \ddc G \w (\ddc\phi)^{n-1}
    - G d \eta \w \dc G \w (\ddc \phi)^{n-1} \\
    &= C_G C_v \int_{B_1(0)} \eta G n [w(m_{\ddc\phi}) (\ddc \phi)^n + d \log v(m_{\ddc\phi}) \w \dc G \w (\ddc\phi)^{n-1}]  \\
    &\qquad + \frac{1}{2} C_G C_v \int_{B_1(0)} G^2 \ddc \eta \w (\ddc \phi)^{n-1}\\
    &\leq n e^{C_G} C_G^2 C_v^2 C_w \int_{B_1(0)} \om^n 
    + \frac{1}{2} C_G^3 C_v C_\eta \int_{B_1(0)} (\Dt \phi)^{n-1} \om^n \\
    &\qquad - \frac{1}{2} n C_G C_v \lt[\int_{B_1(0)} G^2 \eta \ddc \log v(m_{\ddc \phi}) \w (\ddc \phi)^{n-1} + G^2 d \log v(m_{\ddc\phi}) \w \dc \eta \w (\ddc\phi)^{n-1}\rt]\\
    &\leq n e^{C_G} C_G^2 C_v^2 C_w \int_{B_1(0)} \om^n 
    + \frac{1}{2} C_G^3 C_v C_\eta \int_{B_1(0)} (\Dt \phi)^{n-1} \om^n 
     + \frac{1}{2} n C_G C_v \lt[\int_{B_1(0)} e^{C_G} C_G^2 C_v^3 C_\xi^2 (\Delta \phi) \om^n 
    + e^{C_G} C_G^2 C_v^3 C_\xi C_\eta \om^n\rt] \\
    &\leq D_1.
\end{align*}
}%
\endgroup
Hence, the claim follows. 
We provide more details about the inequality next to the last one. Since 
{\small
\[
    \partial_{\bar j}\log v(m_{dd^c \phi}) 
    =  \frac{1}{v(m_{dd^c \phi})} \sum_\af v_\af(m_{dd^c \phi}) (-2\iii\phi_{k\bar{j}} (\xi_\af')^{k}),
\] }
we have $\tr_{\phi} (d \log v(m_{\ddc\phi}) \w \dc \eta) \leq C_v^2 C_\xi C_\eta$; thus, {\small
\[
    d \log v(m_{\ddc\phi}) \w \dc \eta \w (\ddc \phi)^{n-1}
    \leq C_v^2 C_\xi C_\eta (\ddc\phi)^n 
    = C_v^2 C_\xi C_\eta \frac{e^G}{v(m_{\ddc\phi})} \om^n
    \leq e^{C_G} C_v^3 C_\xi C_\eta \om^n.
\] }
On the other hand, for the term involving $dd^c \log v(m_{dd^c \phi})$, 
\begin{align*}
    &\pl_{i\bar{j}}^2 \log v(m_{dd^c \phi}) 
    = \pl_i \lt\{\sum_\af (\pl_\af \log v)(m_{dd^c \phi}) (-2\iii\phi_{k\bar{j}} (\xi_\af')^{k})\rt\} \\
    &= 4\sum_{\af,\bt} (\pl_{\af\bt}^2 \log v)(m_{dd^c \phi}) \phi_{i \bar{q}} \phi_{k \bar{j}} (\xi'_\bt)^{\bar{q}} (\xi_\af')^{k}
    + 2\sum_\af (\pl_\af \log v)(m_{dd^c \phi}) (-\iii) [\phi_{k\bar{j}i} (\xi'_\af)^{k} + \phi_{k\bar{j}} (\xi'_\af)^{k}_i].
\end{align*}
Note that $\CL_{\xi_\af} dd^c \phi = 0$ from the $G$-invariant assumption. 
Therefore, {\small
\begin{align*}
    0 &= \frac{1}{2} \CL_\xi dd^c \phi 
    = \frac{1}{2} d (i_{\xi} dd^c \phi) 
    = d \lt(\iii \phi_{k\bar{l}} (\xi')^k d\bar{z}^l 
    - \iii \phi_{k\bar{l}} (\xi'')^{\bar{l}} dz^k\rt) \\
    &= \lt(\iii \phi_{k\bar{l}} (\xi')^k_i d z^i \w d \bar{z}^l - \iii \phi_{k\bar{l}} (\xi'')^{\bar{l}}_{\bar{j}} d \bar{z}^j \w d z^k\rt) 
    -\iii \phi_{k\bar{l}}(\xi'')^{\bar{l}}_i dz^i \w dz^k 
    + \iii \phi_{k\bar{l}} (\xi')^k_{\bar{j}} d\bar{z}^j \w d\bar{z}^l \\
    &= \iii \lt(\phi_{\gm \bar{\bt}} (\xi')^\gm_\af + \phi_{\af \bar{\dt}} (\xi'')^{\bar{\dt}}_{\bar{\bt}} \rt) dz^\af \w d\bar{z}^\bt
    -\iii \phi_{k\bar{l}}(\xi'')^{\bar{l}}_i dz^i \w dz^k
    + \iii \phi_{k\bar{l}} (\xi')^k_{\bar{j}} d\bar{z}^j \w d\bar{z}^l
\end{align*} }
and this implies $ \phi_{\gm \bar{\bt}} (\xi')^\gm_\af + \phi_{\af \bar{\dt}} (\xi'')^{\bar{\dt}}_{\bar{\bt}} = 0$. 
In the normal coordinates of $dd^c\phi$, we get
\begin{align*}
    \phi^{i\bar j} \phi_{k\bar{j}} (\xi'_\af)^{k}_i
    = \frac{1}{2} \lt(\phi^{i\bar j } \phi_{k\bar{j}} (\xi'_\af)^{k}_i +\overline{\phi^{i\bar j} \phi_{\ell\bar{j}} (\xi'_\af)^{\ell}_i} \rt)
    = \frac{1}{2} \phi^{i\bar j } \lt(\phi_{k\bar{j}} (\xi'_\af)^{k}_i + \phi_{i\bar \ell} (\xi''_\af)^{\bar \ell}_{\bar j}\rt) = 0.
\end{align*}
This yields that $\tr_{\phi}dd^c \log v(m_{dd^c \phi}) \leq C_v^2 C_\xi^2 (\Delta \phi)$ and 
\[
    \ddc \log v(m_{\ddc\phi}) \w (\ddc\phi)^{n-1}
    \leq C_v^2 C_\xi^2 (\Delta \phi) (\ddc\phi)^n
    \leq e^{C_G} C_v^3 C_\xi^2 (\Delta \phi) \om^n.
\]

\smallskip
We also remark that 
$\Delta_{\phi,v} f = \frac{1}{2v(m_{dd^c \phi})} \partial_\af(v(m_{dd^c \phi}) g_\phi^{\af\bt}\partial_{\bt} f)$ where $\af, \bt$ are real coordinates and $[(g_\phi)_{\af\bt}]_{1 \leq \af, \bt \leq 2n}$ is the Riemannian metric associate to $\ddc\phi$. 
Therefore, it follows from Lemma~\ref{lem:appxB}, that
\[
    \pl_\af(v(m_{\ddc\phi}) g_\phi^{\af\bt} \pl_\bt u) \geq -2 C v(m_{\ddc\phi}) (\Dt\phi) u.
\]
Using \cite[Lem.~6.3, arXiv version]{Chen_Cheng_2021_1}), we derive that 
\begin{equation}\label{eq:appxB_CC_6.3_2}
    \|u\|_{L^\infty( B_{1/2}(0))} \leq D_2 (\|u\|_{L^1(B_{3/4}(0))}+1)\leq C,
\end{equation}
where $D_2$ depends on $C_v$, $C_w$, $p$, $\|S\|_{L^\infty}$, $\|\Delta \phi\|_{L^p (B_1(0))}$, $\|\tr_\phi \om\|_{L^p(B_1(0))}$. 
Recall that $u = e^{\frac{G}{2}} |d G|_\phi^2 + K \Dt \phi$; hence we get $$\|u\|_{L^1(B_{3/4}(0))}\leq e^{C_G/2}\||dG|^2_\phi\|_{L^1(B_{3/4}(0))} + K \| \Delta \phi\|_{L^1(B_{3/4}(0))}\leq e^{C_G/2}D_1+  K \| \Delta \phi\|_{L^1(B_{3/4}(0))}.$$
Then combining this with \eqref{eq:appxB_CC_6.3_2}  implies the uniform $L^\infty$-estimates of $|dG|_{\phi}$ and $\Delta \phi$ on $B_{1/2}(0)$ which only depend on $C_v$, $C_w$, $p$, $\|S\|_{L^\infty}$, $\|\Delta \phi\|_{L^p (B_1(0))}$, $\|\tr_\phi \om\|_{L^p(B_1(0))}$. 
Then the standard Evans--Krylov estimate and bootstrapping argument imply higher order estimates on  $B_{1/2}(0)$ for $\phi, G$ which depend on $C_v$, $C_w$, $p$, $\|S\|_{L^\infty}$, $\|\Delta \phi\|_{L^p (B_1(0))}$, $\|\tr_\phi \om\|_{L^p(B_1(0))}$.
\end{proof}

\bibliographystyle{smfalpha_new}
\bibliography{biblio}
\end{document}